\documentclass[a4paper, 10pt]{amsart}
\usepackage{fullpage}
\usepackage[utf8]{inputenc}

\usepackage{amsmath}
\usepackage{amsthm}
\usepackage{ mathrsfs }
\usepackage{ amssymb }
\usepackage{mathtools}
\usepackage{enumitem}  

\usepackage{tikz}
\usetikzlibrary{arrows}

\usepackage{geometry}

\usepackage{bigstrut}
\usepackage{multirow}
\usepackage{listings} 
\usepackage{tikz-cd}		
\usepackage{graphicx}
\newtheorem{thm}{Theorem}[section]
\newtheorem{lem}[thm]{Lemma}

\newtheorem{cor}[thm]{Corollary}
\newtheorem{prop}[thm]{Proposition}

\newtheorem*{thm*}{Theorem}
\newtheorem*{cor*}{Corollary}
\newtheorem{thmll}{Theorem}

\theoremstyle{definition}

\newtheorem{defi}[thm]{Definition}
\newtheorem{Warn}[thm]{Warning}

\theoremstyle{remark}
\newtheorem{rmk}[thm]{Remark}

\newtheorem*{Ack}{Acknowledgment}

\newcommand{\colim}{\operatorname{colim}}
\newcommand{\Hom}{\operatorname{Hom}}

\def\C{\mathbb C}
\def\R{\mathbb R}
\def\N{\mathbb N}

\title{$1$-minimal models for $C_{\infty}$-algebras and flat connections}
\author{Claudio Sibilia}
\begin{document}
\begin{abstract}
 Given a smooth manifold $M$ equipped with a properly and discontinuous smooth action of a discrete group $G$, the nerve $M_{\bullet}G$ is a simplicial manifold and its vector space of differential forms $\operatorname{Tot}_{N}\left(A_{DR}(M_{\bullet}G)\right)$ carry a $C_{\infty}$-algebra structure $m_{\bullet}$. We show that each $C_{\infty}$-algebra $1$-minimal model $g_{\bullet}\: : \: \left(W, {m'}_{\bullet} \right) \to \left( \operatorname{Tot}_{N}\left(A_{DR}(M_{\bullet}G)\right),m_{\bullet}\right) $ gives a flat connection $\nabla$ on a smooth trivial bundle $E$ on $M$ where the fiber is the Malcev Lie algebra of $\pi_{1}(M/G)$ and its monodromy representation is the Malcev completion of $\pi_{1}(M/G)$. This connection is unique in the sense that different $1$-models give isomorphic connections. In particular, the resulting connections are isomorphic to Chen's flat connection on $M/G$. If the action is holomorphic and $g_{\bullet}$ has holomorphic image (with logarithmic singularities), $(\nabla,E)$ is holomorphic and different $1$-models give (holomorphically) isomorphic connections (with logarithmic singularities).
 These results are the equivariant and holomorphic version of Chen's theory of flat connections. 
\end{abstract}
\maketitle
%%%%%%%%%%%%%%%%%%%%%%%%%%%
% abstract, keywords and Subject classification are optional.
%%%%%%%%%%%%%%%%%%%%%%%%%%%

% Most people don't use these, so they are "commented out"
% by starting the lines with a "%"
%\begin{keywords}
%   \LaTeX, typesetting
%\end{keywords}

%\begin{AMS}
%   50C60, 18C25
%\end{AMS}

%%%%%%%%%%%%%%%%%%%%%%
% % Here is the start of the Text
%%%%%%%%%%%%%%%%%%%%%%
\section*{Introduction}
Let $H$ be an abstract group and let $\Bbbk$ be a field of characteristic zero. We denote by $\Bbbk\left[ H\right] $ the group ring. It is a Hopf algebra where the coproduct is given by $\Delta\left( g\right) :=g\otimes g$. Let $J$ be the kernel of the augmentation map $\Bbbk\left[ H\right]\to \Bbbk$ that sends each element of $H$ to $1$. The powers of $J$ (with respect to the multiplication) define a filtration $J^{i}$ such that the completion $\Bbbk\left[ H\right]^{\wedge }$ is a complete Hopf algebra, i.e a complete vector space such that the structure maps are continuous (see \cite{benoitfresse}). A group $H$ is said to be \emph{Malcev complete} if it is isomorphic to the group like elements of $\Bbbk\left[ G\right]^{\wedge }$ for some group $G$. For a group $H$ the primitive elements of $\Bbbk\left[ H\right]^{\wedge }$ are called the (complete) Malcev Lie algebra of $H$. The Malcev completion of a group $H$ (over $\Bbbk$) consists of a Malcev complete group $\widehat{H}$ and a universal group homomorphism $H\to \widehat{H}$. It can be constructed by taking the group like elements of $\Bbbk\left[ H\right]^{\wedge }$ and the adjunction between Hopf algebras and Malcev complete groups (see \cite{benoitfresse}).  Let $M$ be a connected smooth manifold such that $H^{i}\left(M, \C \right) $ is finite dimensional for any $i>0$. In \cite{extensionChen} K.T. Chen  gives a formula for the Malcev completion (over $\R$ or $\C$) of $H=\pi_{1}\left(M,p \right) $ for $p\in M$. In fact, he constructs a flat connection on a trivial bundle on $M$ such that its monodromy representation is the Malcev completion of $\pi_{1}\left(M,p \right) $. The monodromy representation can be explicitly calculated via iterated integrals. We denote by $A_{DR}^{\bullet}(M)$ the differential graded algebra of complex differential forms on $M$. In order to present the main theorem of \cite{extensionChen}, we need the following preliminaries: for any pronilpotent Lie algebra $\mathfrak{g}$, each 
\[
C\in A_{DR}^{1}(M)\widehat{\otimes}\mathfrak{g}
\] 
defines a connection $d-C$ on the trivial bundle $M\times\mathfrak{g}$, where the latter is considered to be equipped with the adjoint action. For a non-negatively graded vector space $V$ of finite type, we denote by $V_{+}$ its positively graded part and by $V[1]$ the shifted graded vector space such that $\left( V[1]\right)^{i}=V^{i+1}[1]=V^{i+1}$ for any $i$. Its graded dual is denoted by $V^{*}$. Let $\widehat{T}(V)$ be the complete free algebra on $V$. It is a Hopf algebra equipped with the shuffle coproduct. The primitive elements $\widehat{\mathbb{L}}(V)\subset\widehat{T}(V)$ are the complete free Lie algebra over $V$. Let $A\subseteq A_{DR}^{\bullet}(M) $ be a differential graded subalgebra equipped with a Hodge type decomposition i.e., a vector space decomposition
\[
A=W\oplus \mathcal{M}\oplus d\mathcal{M}
\]
where $W=\oplus_{p\geq 0}{W}^{p}$ is a graded vector subspace of closed elements such that the inclusion $W\hookrightarrow A$ is a quasi-isomorphism, and $\mathcal{M}$ is a graded vector subspace containing no exact elements except $0$. We assume that $A\hookrightarrow A_{DR}^{\bullet}(M)$ is a quasi-isomorphism, i.e. that $A$ is a model. 
\begin{thmll}[\cite{extensionChen}]\label{Chenthm}
There exist
\begin{enumerate}
\item a Lie ideal $\mathcal{R}_{0}\subset \widehat{\mathbb{L}}\left( (W_{+}^{1}[1])^{*}\right) $, and
\item an element $C_{0}\in A^{1}\widehat{\otimes}\mathfrak{u}$, where $\mathfrak{u}=\widehat{\mathbb{L}}\left( (W_{+}^{1}[1])^{*}\right)/\mathcal{R}_{0}$, 
\end{enumerate}
such that $\mathfrak{u}$ is the complete Malcev Lie algebra of $\pi_{1}\left(M,p \right)$ and $d-C_{0}$ defines a flat connection on the trivial bundle $M\times\mathfrak{u}$ whose monodromy represention at $p\in M$ is a group homomorphism
\[
\Theta\: : \: \pi_{1}\left(M,p \right)\to \widehat{T}\left( (W_{+}^{1}[1])^{*}\right)/\bar{\mathcal{R}}_{0}\subset \operatorname{End}\left(\mathfrak{u} \right) 
\]
where $\bar{\mathcal{R}}_{0}\subset \widehat{T}\left( (W_{+}^{1}[1])^{*}\right)$ is the ideal generated by $\mathcal{R}_{0}$. Moreover, $\Theta\left(\pi_{1}\left(M,p \right) \right)\subset \exp(\mathfrak{u})$ and $
\Theta\: : \: \pi_{1}\left(M,p \right)\to\exp(\mathfrak{u})
$ is the Malcev completion of $\pi_{1}\left(M,p \right)$. The element $C_{0}$ depends uniquely on the choice of $A$, of the Hodge type decomposition and of basis for $W$.
\end{thmll}
The Lie ideal $\mathcal{R}_{0}$ is constructed in \cite{extensionChen} as the image of a differential $\delta^{*}$ on $\widehat{\mathbb{L}}\left( (W_{+}^{1}[1])^{*}\right)$ and the connection $C_{0}$ is constructed as a degree zero part of a degree one element $C\in A\widehat{\otimes} \widehat{\mathbb{L}}\left( (W_{+}[1])^{*}\right) $. The pair $(C,\delta^{*})$ satisfies some algebraic conditions and they are unique up to the choice of Hodge type decomposition.
Chen first proves the existence of an element $C$ and a differential $\delta^{*}$ that satisfy such algebraic conditions (see \cite{extensionChen}, Theorem 1.3.1). Secondly, he proves that they induce $C_{0}$ and $\mathcal{R}_{0}$ that satisfy Theorem \ref{Chenthm} (see \cite{extensionChen}, Theorem 2.1.1). The proof for the existence of $C$ and $\delta^{*}$ in \cite{extensionChen} is rather geometrical and involves iterated integrals. A purely algebraic proof can be obtained by homological perturbation theory (see for instance \cite{Algchentwist}) or equivalently the homotopy transfer theorem (see \cite{Huebsch}). In fact, $\delta^{*}$ corresponds to a $C_{\infty}$-structure on $W$ induced by $A$ along the inclusion $W\hookrightarrow A$ and $C$ corresponds to a $C_{\infty}$-quasi-isomorphism between $W$ and $A$. The connection obtained by Theorem \ref{Chenthm} is holomorphic if $A$ is a holomorphic model, i.e it contains only holomorphic forms. The celebrated
Knizhnik–Zamolodchikov connection can be constructed via Theorem \ref{Chenthm}, using a holomorphic model for the configuration space of points due to Arnold (see \cite{Arnold}).
In the present paper, we generalize Theorem \ref{Chenthm} as follows. Let $M$ be a smooth (complex or real) manifold equipped with a smooth properly discontinuous action of a discrete group $G$. The nerve gives a simplicial manifold $M_{\bullet}G$ called the action groupoid. In particular, $A_{DR}\left(M_{\bullet}G \right) $ is a simplicial commutative differential graded algebra. Let $\operatorname{Tot}_{N}^{\bullet}\left(A_{DR}\left(M_{\bullet}G \right)\right)$ be the normalized total complex. An element $w$ of degree $n$ in this complex can be written as $w=\sum_{p+q=n}w^{p,q}$, where $w^{p,q}$ is a set map
\[
w^{p,q}\: : \: G^{p}\to A_{DR}^{q}(M)
\]
such that $w^{p,q}(g_{1}, \dots, g_{n})=0$ if $g_{i}=e\in G$ for some $i$. The elements $w^{p,q}$ are called elements of bidegree $(p,q)$. In particular, $\operatorname{Tot}_{N}^{0,q}\left(A_{DR}\left(M_{\bullet}G \right)\right)=A_{DR}^{q}(M)$. An element $w$ is said to be holomorphic if each $w^{p,q}$ takes holomorphic values for any $(p,q)$. Since the action is properly discontinuous, we can visualize $A_{DR}^{\bullet}(M/G)\subset A_{DR}^{\bullet}(M)$ as the differential graded algebra of $G$-invariant differential forms. Getzler and Cheng (see \cite{Getz}) have shown that the normalized total complex $\operatorname{Tot}_{N}^{\bullet}\left(A_{DR}\left(M_{\bullet}G \right)\right)$ carries a natural unital $C_{\infty}$-structure $m_{\bullet}$
such that 
\[
m_{1}=d,\quad m_{2}=\wedge,\quad m_{n}=0\text{ for }n\geq 3
\]
on $A_{DR}^{\bullet}(M/G)\subset\operatorname{Tot}_{N}^{0,\bullet}\left(A_{DR}\left(M_{\bullet}G \right)\right)$.
We denote the differential $m_{1}$ by $D$. In  \cite{Dupont} Dupont has shown that $H^{\bullet}\left(\operatorname{Tot}_{N}^{\bullet}\left(A_{DR}\left(M_{\bullet}G \right)\right),D \right)  $ corresponds to the singular cohomology of the fat geometric realization of $M_{\bullet}G$. Since the action is properly discontinuous,  $H^{\bullet}\left( ||M_{\bullet} G||,\C\right)\cong  H^{\bullet}\left( M/G,\C\right)$ and in particular the inclusion 
\[
\left( A_{DR}(M/G),d, \wedge\right) \hookrightarrow\left( \operatorname{Tot}_{N}\left(A_{DR}(M_{\bullet}G)\right),m_{\bullet}\right) 
\]
is a quasi-isomorphism of $C_{\infty}$-algebras. We assume that $M/G$ is connected and that $H^{i}\left(M/G, \C \right) $ is finite dimensional for any $i$. Let $B\subseteq \operatorname{Tot}_{N}\left(A_{DR}(M_{\bullet}G)\right)$ be a $C_{\infty}$-subalgebra. We assume that $\left( B,D\right) \hookrightarrow \left( \operatorname{Tot}_{N}\left(A_{DR}(M_{\bullet}G)\right),D\right) $ induces an isomorphism on the $0$-th and on $1$-th cohomology group and an injection on the $2$-th, i.e. $B$ is a $1$-model for  $\left( \operatorname{Tot}_{N}\left(A_{DR}(M_{\bullet}G)\right),m_{\bullet}\right)$. A $1$-minimal model for $\operatorname{Tot}_{N}\left(A_{DR}(M_{\bullet}G)\right)$ consists of a $C_{\infty}$-morphism $g_{\bullet}\: : \: \left(W, {m'}_{\bullet} \right) \to \operatorname{Tot}_{N}\left(A_{DR}(M_{\bullet}G)\right)$ such that $m_{1}=0$ (i.e. $W$ is minimal) and 
$g_{\bullet}$ induces an isomorphism on the $0$-th and on $1$-th cohomology group and an injection on the $2$-th one. Explicit examples of $1$-minimal models can be constructed via the homotopy transfer theorem (see \cite{Kadesh}, \cite{kontsoibel}, \cite{Markl} and \cite{Prelie}). In this paper, we introduce the category of $1-C_{\infty}$-algebras and $1-C_{\infty}$-morphisms. They are essentially ``$1$-truncated'' $C_{\infty}$-algebras, i.e. a graded vector space $A$ equipped with maps $\tilde{m}_{n}\: :\: \oplus_{i\leq n}\left(  A^{\otimes n}\right)^{i}\to A$ that satisfy the ordinary relations for $C_{\infty}$-algebras only on $\oplus_{i\leq n-1}\left(  A^{\otimes n}\right)^{i}$. There is a forgetful functor $\mathcal{F}$ from the category of $C_{\infty}$-algebras towards the category of $1-C_{\infty}$-algebras. The $1-C_{\infty}$-algebras have a well-defined notion of cohomology groups and morphisms. In particular, the second cohomology groups is generated by the higher Massey products between elements of lower degree. We introduce the notion of $1$-minimal models of $1-C_{\infty}$-algebras such that each ordinary $1$-minimal models gives a $1$-minimal model of $1-C_{\infty}$-algebras via $\mathcal{F}$.
A $1$-minimal model $g_{\bullet}\: : \: \left(W, {m'}_{\bullet} \right) \to  \mathcal{F}\left(\operatorname{Tot}_{N}\left(A_{DR}(M_{\bullet}G)\right), m_{\bullet}\right)$ is said to be \emph{holomorphic } if the image of $g_{\bullet}$ contains only holomorphic elements. Let $\left( N,\mathcal{D}\right) $ be a smooth complex manifold with a normal crossing divisor $\mathcal{D}$ and $G$ be a group acting holomorphically on $N$ preserving $\mathcal{D}$. Assume that $M=N-\mathcal{D}$, then $g_{\bullet}$ is said to be \emph{holomorphic with logarithmic singularities} if the image of $g_{\bullet}$ contains only holomorphic elements with logarithmic singularities along $\mathcal{D}$. In the present paper, we prove the following. For a $p\in M$ we denote its class in $M/G$ by $\overline{p}$.
\begin{thmll}\label{Final1}
 Let $M$ be a smooth complex  manifold equipped with a properly discontinuous action of a discrete group $G$.
For each $1$-minimal model $g_{\bullet}\: : \: \left(W, {m'}_{\bullet} \right) \to  \mathcal{F}\left(\operatorname{Tot}_{N}\left(A_{DR}(M_{\bullet}G)\right), m_{\bullet}\right) $ there exists
\begin{enumerate}
\item a Lie ideal $\mathcal{R}_{0}\subset \widehat{\mathbb{L}}\left( (W_{+}^{1}[1])^{*}\right) \subset \widehat{T}\left( (W_{+}^{1}[1])^{*}\right) $,
\item an element $r_{*}\left(\overline{C}\right) \in  A_{DR}^{1}(M)\widehat{\otimes}\mathfrak{u}$, where $\mathfrak{u}=\widehat{\mathbb{L}}\left( (W_{+}^{1}[1])^{*}\right)/\mathcal{R}_{0}$, and 
\item a fiber preserving smooth $G$-action on $M\times \mathfrak{u}$ such that $\mathfrak{u}$ is the complete Malcev Lie algebra of $\pi_{1}\left(M/G, {p} \right)$ and $d-r_{*}\left( \overline{C}\right)$ defines a flat connection on a trivial bundle $E:=\left( M\times \mathfrak{u}\right)/G $ on $M/G$ whose monodromy representation at $\overline{p}\in M/G$ is a group homomorphism
\[
\Theta\: : \: \pi_{1}\left(M/G,\overline{p} \right)\to \widehat{T}\left( (W_{+}^{1}[1])^{*}\right)/\bar{\mathcal{R}}_{0}\subset \operatorname{End}\left(\mathfrak{u} \right) 
\]
where $\bar{\mathcal{R}}_{0}\subset \widehat{T}\left( (W_{+}^{1}[1])^{*}\right)$ is the ideal generated by $\mathcal{R}_{0}$. Moreover, for $\bar{p}\in M$ we have\\ $\Theta\left(\pi_{1}\left(M/G,p \right) \right)\subset \exp(\mathfrak{u})$ such that $
\Theta\: : \: \pi_{1}\left(M/G,\bar{p} \right)\to\exp(\mathfrak{u})
$ is the Malcev completion of $\pi_{1}\left(M/G,\bar{p} \right)$. 
\item Assume that the action of $G$ on $M$ is holomorphic and $g_{\bullet}$ is holomorphic. Then $E$ is a holomorphic bundle on $M/G$ and $d-r_{*}\left(\overline{C}\right)$ defines a holomorphic flat connection on $E$.
\item Let $\left( N,\mathcal{D}\right) $ be a smooth complex manifold with a normal crossing divisor and $G$ be a group acting holomorphically on $N$ and that preserves $\mathcal{D}$. Assume that $M=N-\mathcal{D}$ and that $g_{\bullet}$ is holomorphic with logarithmic singularities. Then $E$ is a holomorphic bundle on $N/G$ and $d-r_{*}\left(\overline{C}\right)$ defines a holomorphic flat connection with logarithmic singularities on $E$.
\end{enumerate}
\end{thmll}
Theorem \ref{Chenthm} is special case of Theorem \ref{Final1}. Let $A\subset A_{DR}(M/G)$. The Hodge type decomposition $A=W\oplus \mathcal{M}\oplus d\mathcal{M}$ can be turned via the homotopy transfer theorem into a $1$-minimal model $g_{\bullet}\: : \: \left(W, {m'}_{\bullet} \right) \to A$ which is unique with respect some algebraic conditions relative to the decomposition. This $1$-minimal model corresponds via Theorem \ref{Final1} to a proof of \ref{Chenthm}. In this case, the group action on the bundle is trivial.\\
The strategy for the proof of Theorem \ref{Final1} is the following. The $1-C_{\infty}$-structure ${m'}_{\bullet}$ corresponds to a map $\delta'\: : \: \left( T^{c}(W_{+}[1])\right)^{0} \to \left( T^{c}(W_{+}[1])\right)^{1}$ and $\mathcal{R}_{0}$ is the Lie ideal generated by the image of ${\delta'}^{*}$. We construct a particular $L_{\infty}$-algebra called degree zero convolution Lie algebra and show that it contains some special Maurer-Cartan elements $\overline{C}$, each of which can be turned into a connection form $r_{*}\overline{C} \in A_{DR}(M)\widehat{\otimes}\left(  \widehat{\mathbb{L}}\left( (W_{+}[1])^{*}\right)/\mathcal{R}_{0}\right) $ that fulfills the above conditions. In general a holomorphic $1$-minimal model with logarithmic singularities for $ \mathcal{F}\left( A_{DR}(M/G),d, \wedge \right)$ does not exist, because there are not enough closed holomorphic forms. In some cases, it is possible to find a holomorphic $1$-minimal model with logarithmic singularities for $ \mathcal{F}\left( \operatorname{Tot}_{N}\left(A_{DR}(M_{\bullet}G)\right),m_{\bullet}\right)$
This is the case for punctured Riemann surfaces of positive genus. In the present paper, we give a comparison between Theorem \ref{Chenthm} and Theorem \ref{Final1}, and we analyze the dependence on the choice of $1$-minimal model. Let $M$ be a smooth complex manifold and let $E_{1}$, $E_{2}$ be two smooth vector bundles on $M$. Let $\left(d-\alpha_{1},E_{1} \right) $, $\left(d-\alpha_{2},E_{2} \right) $ be two smooth connections. They are \emph{isomorphic} if there exists a bundle isomorphism $T\: : \: E_{1}\to E_{2}$ such that 
\[
T\left( d-\alpha_{1}\right)T^{-1}=d-\alpha_{2}.
\]

\begin{thmll}\label{Final2}
Let $d-\alpha_{1}$ and $d-\alpha_{2}$ be two flat connections on two smooth vector bundles $E_{1}$ and $E_{2}$ on $M/G$ obtained by Theorem \ref{Final1} by choosing some $1$-minimal models. 
\begin{enumerate}
\item The connections $\left( d-\alpha_{1},E_{1}\right) $ and $\left( d-\alpha_{2},E_{2}\right) $ are isomorphic.
\item Assume that $d-\alpha_{1}$ and $d-\alpha_{2}$ are constructed using holomorphic $1$-models. The above isomorphism is holomorphic.
\item Let $\left( N,\mathcal{D}\right) $ be a smooth complex manifold with a normal crossing divisor and $G$ be a group acting holomorphically on $N$ and that preserves $\mathcal{D}$. Assume that $M=N-\mathcal{D}$ and that $d-\alpha_{1}$ and $d-\alpha_{2}$ are constructed using holomorphic $1$-models with logarithmic singularities. The above isomorphism is holomorphic and it extends to an holomorphic isomorphism between bundles on $N/G$.
\end{enumerate}
In particular, the connections are (smoothly) isomorphic to the connection obtained in Theorem \ref{Chenthm} on $M/G$.
\end{thmll}
Theorem \ref{Final1} and \ref{Final2} are Theorem \ref{wehaveabundle} and Theorem \ref{wehaveabundleuniq}. In a forthcoming paper, we construct two holomorphic $1$-minimal models and we show that the KZB connection on the punctured torus and on the configuration space of points of the punctured torus presented in \cite{Damien} can be constructed explicitly using the homotopy transfer theorem and Theorem \ref{Final1} (see \cite{Sibilia1}). The methods can be used to construct holomorphic flat connection on punctured surfaces of arbitrary genera (compare with \cite{Enriquezconnection}). We mainly work with smooth complex manifolds but all the result contained here work for real manifolds as well.

\subsection*{Plan of the paper.}
In Section \ref{first} we discuss $1-C_{\infty}$-algebras and $L_{\infty}$-algebras and we construct the convolution $L_{\infty}$-algebra. This section is purely algebraic. In Section \ref{sectGetztler}, we recall a natural $C_{\infty}$-structure on the total complex of a cosimplicial commutative algebra due to Getzler and Cheng (\cite{Getz}). It gives a natural $C_{\infty}$-structure on the differential forms on a simplicial manifold and we analyze some properties of such a $C_{\infty}$-structure. In Section \ref{third}, given a simplicial manifold $M_{\bullet}$, we consider the convolution $L_{\infty}$-algebra constructed using such a $C_{\infty}$-structure. Certain Maurer-Cartan elements of this $L_{\infty}$-algebra corresponds to flat connections on a trivial bundle on $M_{0}$. In Section \ref{newadded}, we restrict our attention on action groupoids and we give a proof of Theorem \ref{Final1}, \ref{Final2}.

\begin{Ack}
This article is based on the first part of my PhD project. I wish to thank my two advisors Damien Calaque and Giovanni Felder for their guidance and constant support and Benjamin Enriquez for his interesting questions about my PhD thesis. I'm very much indebted to Daniel Robert-Nicoud for long discussion about Section \ref{sectGetztler}. I thanks the SNF for providing an essential financial support through the ProDoc module PDFMP2 137153 ``Gaudin subalgebras, Moduli Spaces and Integrable Systems" I also thank the support of the ANR SAT and of the Institut Universitaire de France.
\end{Ack}

\subsection*{Notation} Let $\Bbbk$ be a field of characteristic zero.
We work in the unital monoidal tensor category of graded vector spaces $(grVect,\otimes,\Bbbk, \tau)$ where the field $\Bbbk$ is considered as a graded vector space concentrated in degree $0$, the twisting map is given $\tau(v\otimes w):=(-1)^{|v||w|}w\otimes v$ and the tensor product is the ordinary graded tensor product. For a graded vector space $V^{\bullet}$, $V^{i}$ is called the homogeneous component of $V$, and for $v\in V^{i}$ we define its degree via $|v|:=i$. For a vector space $W:=\oplus_{i\in I} W_{i}$ we denote by ${pro}_{W_{i}}\: : \: W\to W_{i}$ the projection. For a graded vector space $V^{\bullet}$ we denote by $V[n]$ the $n$-shifted graded vector space, where $\left( V[n]\right) ^{i}=V^{n+i}$. For example $\Bbbk[n]$ is a graded vector space concentrated in degree $-n$ (its $-n$ homogeneous component is equal to $\Bbbk$, the other homogeneous component are all equal to zero). A (homogeneous) morphism of graded vector spaces $f\: : \: V^{\bullet}\to W^{\bullet}$ of degree $|f|:=r$ is a linear map such that $f(V^{i})\subseteq V^{i+r}$. Given two graded vector spaces $V$, $W$, then the set of morphisms of degree $n$ is denoted  $\Hom^{n}_{gVect}\left(V,W \right)$. More generally the set of maps between $V$ and $W$ is again a graded vector space $\Hom^{\bullet}_{gVect}\left(V,W \right)$ for which the $i$-homogeneous elements are those of degree $i$. 
The tensor product of homogeneous morphisms is defined through the \emph{Koszul convention}: given two morphisms of graded vector spaces $f\: : \: V^{\bullet}\to W^{\bullet}$ and $g\: : \: {V'}^{\bullet}\to {W'}^{\bullet}$ then its tensor product $f\otimes g\: : \: \left( V\otimes W\right)^{\bullet}\to \left( V'\otimes W'\right)^{\bullet}$ on the homogeneous elements is given by 
\[
\left( f\otimes g\right) (v\otimes w ):=(-1)^{|g||v|}(f(v)\otimes g(w )).
\]
We denote by $s\: : \: V\to V[1]$, $s^{-1}\: : \: V[1]\to V$  the shifting morphisms that send $V^{n}$ to $V[1]^{n-1}=\Bbbk \otimes V^{n}=V^{n}$ resp. $V[1]^{n}=\Bbbk \otimes V^{n+1}=V^{n+1}$ to $V^{n+1}$. Those maps can be extended to a map $s^{n}\: : \: V\to V[n]$, (the identity map shifted by $n$).
Note that $s^{n}\in \Hom^{-n}\left(V, V[n] \right)$. A graded vector space is said to be of finite type if each homogeneous component is a finite vector space. A graded vector space $V^{\bullet}$ is said to be bounded below at $k$ if there is a $k$ such that $V^{l}=0$ for $l<k$. Analogously it is said to be bounded above at $k$ if there is a $k$ such that $V^{l}=0$ for $l>k$. For a non-negatively graded vector space we define the positively graded vector space
$W^{\bullet}_{+}$ as
\[
W_{+}^{0}:=0,\quad W_{+}^{i}:=W^{i}\text{ if }i\neq0
\]
Let $(V,d_{V})$ be a differential graded vector space, then $V^{\otimes n}$ is again a differential graded vector space with differential
\[
d_{V^{\otimes n}}(v_{1}\otimes\cdots\otimes  v_{n}):=\pm\sum_{i=1}^{n}v_{1}\otimes\cdots\otimes  d_{V}v_{i}\cdots\otimes  v_{n},
\]
where the signs follow from the Koszul signs rule.
Let $(V,d_{V})$, $(W,d_{W})$ be differential graded vector spaces, then $\Hom^{\bullet}_{gVect}\left(V,W \right)$ is a differential graded vector space with differential
\[
\partial f:= d_{W}f-(-1)^{|f|}fd_{v}.
\]

\section{$A_{\infty}$, $C_{\infty}$-structures and homological pairs}\label{first}
In this section we construct a $L_{\infty}$-algebra (called convolution $L_{\infty}$-algebra) and we study their Maurer-Cartan elements. 
\subsection{$A_{\infty}$, $C_{\infty}$, $L_{\infty}$-structures}\label{Intro A infty}
We introduce $A_{\infty}$, $C_{\infty}$, $L_{\infty}$-structures. Our reference is \cite{lodayVallette}. For a graded vector space $V$, let ${T^{c}}\left(V[1]\right)$ be the (graded) tensor coalgebra on shift of  $V$. Let $NTS(V)\subset {T^{c}}\left(V[1]\right)$ be the subspace of non-trivial shuffles, i.e the vector space generated by $\mu'(a,b)$ such that $\mu'$ is the graded shuffle product and $a,b\notin \Bbbk\subset {T^{c}}\left(V[1]\right)$.
\begin{defi}\label{defAinfty}
Let $V^{\bullet}$ be a graded vector space. An \emph{$A_{\infty}$-algebra structure} on $V^{\bullet}$ is a coderivation $\delta\: : \: {T^{c}}\left(V[1]\right)\to {T^{c}}\left(V[1]\right)$ of degree $+1$ such that $\delta ^{2}=0$. A \emph{$C_{\infty}$-algebra structure} on $V^{\bullet}$ is an $A_{\infty}$-structure such that $\delta(NTS(V[1]))=0$.
\end{defi}
Each coderivation is uniquely determined by the maps of degree $1$
\[
\begin{tikzcd}
\delta_{n}\: : \:A[1]^{\otimes n}\arrow[hook]{r}&{T^{c}}\left(A[1]\right)\arrow{r}{\delta}& 
 {T^{c}}\left(A[1]\right)\arrow{r}{{pro}_{A[1]}} & A[1].\end{tikzcd}
\]
For $n>0$ we define maps $m_{n}\: : \: A^{\otimes n}\to A$ of degree $2-n$ via
\[
\begin{tikzcd}
A^{\otimes n}\arrow{r}{\left( s\right)^ {\otimes n}}& A[1]^{\otimes n}\arrow{r}{\delta_{n}}&A[1]\arrow{r}{s^{-1}}&A.
\end{tikzcd}
\]
The condition $\delta^{2}=0$ for the maps $m_{n}\: : \: A^{\otimes n}\to A$ implies $m_{1}^{2}=0$ and the relations
\begin{equation}\label{rel}
\sum_{\substack{p+q+r=n\\ k=p+1+r\\k,q>1}}(-1)^{p+qr}m_{k}\circ \left(\mathrm{Id}^{\otimes p}\otimes m_{q}\otimes \mathrm{Id}^{\otimes r} \right)=\partial m_{n}, \quad n>0.
\end{equation}
Conversely, starting from maps $m_{\bullet}:=\left\lbrace m_{n} \right\rbrace _{n\geq0}$ that satisfy the above relations, we get a sequence of maps $\delta_{n}\: : \:A[1]^{\otimes n}\to A[1]$, for $n\geq0$ defined via
\[
\begin{tikzcd}
\delta_{n}\: : \:\left(  A[1]\right) ^{\otimes n}\arrow{r}{\left( s^{-1}\right)^ {\otimes n}}& A^{\otimes n}\arrow{r}{m_{n}}& A\arrow{r}{s}& A[1].
\end{tikzcd}
\]
In particular we have $m_{1}^2=0$. These maps can be viewed as the restriction of a coderivation $\delta$, which is a differential by \eqref{rel}. We denote an $A_{\infty}$- ( $C_{\infty}$- ) algebra $\left( {T^{c}}\left(V[1]\right),\delta\right)$ by $\left(A^{\bullet},m_{\bullet}  \right) $  as well. 
\begin{defi}
	An $A_{\infty}$-algebra $\left(A,m_{\bullet}  \right) $ is said to be \emph{unital} if there exists a $m_{1}$-closed element $1$ of degree zero such that $m_{2}(1,a)=1=m_{2}(a,1)$ and $m_{k}(a_{1}\dots,1,\dots, a_{k})=0$ for $k\geq 3$.  $\left(A,m_{\bullet}  \right) $ is said to be \emph{connected} if $A^{0}=\Bbbk1$ and $A$ is a non-negatively graded vector space.
	Let $A^{\bullet}$, $B^{\bullet}$ be two $A_{\infty}$-algebras. A \emph{morphism between $A_{\infty}$-algebras} is a morphism of differential graded coalgebras
\[
F\: : \: \left({T^{c}}\left(A[1]\right),{\Delta}, \delta_{A}\right) \to \left( {T^{c}}\left(B[1]\right),{\Delta}, \delta_{B}\right).
\]
Each morphism is completely determined by the degree zero maps, i.e.
\[
\begin{tikzcd}
F_{n}\: : \:A[1]^{\otimes n}\arrow[hook]{r}& {T^{c}}\left(A[1]\right)\arrow{r}{F}& 
{T^{c}}\left(B[1]\right)\arrow{r}{{pro}_{B[1]}} & B[1],\quad n>0.
\end{tikzcd}
\]
and $F_{0}(1):=F(1)=1$. $F$ is said to be a \emph{morphism of $C_{\infty}$-algebras} if $A$ and $B$ are $C_{\infty}$-algebras and $F_{n}(NTS(V[1])\cap V[1]^{\otimes n})=0$.
\end{defi}
We denote by $f_{\bullet}:=\left\lbrace f_{n} \right\rbrace_{n\in \N}$ the family of maps of degree $1-n$ given by
\[
\begin{tikzcd}
f_{n}\: : \: A^{\otimes n}\arrow{r}{\left( s\right)^{\otimes n}}& A[1]^{\otimes n}\arrow{r r}{{pro}_{B[1]}\circ F_{n}} & & B[1]\arrow{r}{s^{-1}}& B.
\end{tikzcd}
\]
Let $m_{n}^{A}$ be the degree $2-n$ maps obtained from $\delta_{A}$, and let $m_{n}^{B}$ be the ones from $\delta_{B}$. The condition $F\circ \delta_{A}=\delta_{B}\circ F$ implies the following equations $f_{1}m_{1}^{A}=m_{1}^{B}f_{1}$ and
\begin{align}\label{tricksign}
&\nonumber \sum_{\substack{p+q+r=n\\ k=p+1+r}}(-1)^{p+qr}f_{k}(Id^{\otimes p}\otimes m_{q}^{A}\otimes Id^{\otimes r})\\
&-\sum_{k\geq2,\,i_{1}+\dots +i_{k}=n}(-1)^{s}{m}_{k}^{B}\left(f_{i_1}\otimes\dots\otimes f_{i_k} \right)= \partial f_{n},\quad n> 1,
\end{align}
where $s=\left(k-1 \right) \left(i_{1}-1 \right) +\left(k-2 \right) \left(i_{2}-1 \right) +\dots+2\left(i_{k-2}-1 \right) +\left(i_{k-1}-1 \right) $. In particular $f_{1}$ is an ordinary cochain map. A family of maps $f_{\bullet}$ satisfying condition \eqref{tricksign} induces also a morphism of $A_{\infty}$-algebras $F\: : \: \left({T^{c}}\left(A[1]\right),{\Delta}, \delta_{A}\right) \to \left( {T^{c}}\left(B[1]\right),{\Delta}, \delta_{B}\right)$.
\begin{rmk}
We will adopt the following notation. We will use small dotted letters (e.g. $f_{\bullet}$) to denote morphisms of $A_{\infty}$, $C_{\infty}$-algebras. On the other hand, we will use capital letters (e.g. $F$) to denote morphism of $A_{\infty}$, $C_{\infty}$-algebras as morphism of quasi-free coalgebras.
\end{rmk}
\begin{defi}
A morphism between $A_{\infty}$ ( $C_{\infty}$ )-algebras is a \emph{quasi-isomorphism} if the cochain map
\[
f_{1}\: : \: \left(A, m_{1}^{A}\right)\to  \left(B, m_{1}^{B}\right)
\]
is a quasi-isomorphism. A morphism between $A_{\infty}$ ( $C_{\infty}$ ) -algebras $F\: : \: {T^{c}}\left(A[1]\right)\to {T^{c}}\left(B[1]\right)$ is an isomorphism if $f_{1}$ is an isomorphism. A morphism $f_{\bullet}$ is called strict if $f_{n}=0$ for $n>1$.
\end{defi}
 We denote by $\Omega(1)$ the free commutative graded algebra generated by $1,t_{0}$, $t_{1}$ of degree zero and by $dt_{0}$, $dt_{1}$ of degree $1$ such that
\[
t_{0}+t_{1}=1,\quad dt_{0}+dt_{1}=0.
\]
We put a differential $d$ on $\Omega(1)$ by $d1=0$ and $d(t_{j}):=dt_{j}$, for $ j=0,1$ such that $\Omega(1)$ is a differential free commutative graded algebra. Equivalently, we may define $\Omega(1)$ as the commutative free graded algebra generated by $1,t$ of degree $0$ and $dt$ in degree $1$. The differential here is given by  $d(1):=0, d(t):=dt$. The two presentations are isomorphic via the map $t\mapsto t_{0}$. We denote by $i_{j}\: : \: \Omega(1)\to \Bbbk$ the dg algebra map sending $t_{j}$ to $1$ and $dt_{j}$ to $0$ for $j=0,1$.
\begin{lem}\label{tensorinfinity}
	Let $f_{\bullet}\: : \: \left( A,m_{\bullet}^{A}\right)\to \left( B,m_{\bullet}^{B}\right) $ be a morphism of $A_{\infty}$-algebras and $g\: : \: \Omega(1)\to \Omega(1)$ be a morphism of differential graded algebras. 
	\begin{enumerate}
		\item  $\left( \Omega(1)\otimes A, m_{\bullet}^{ \Omega(1)\otimes A}\right) $ is an $A_{\infty}$-algebra via
		\[
		m_{n}^{ \Omega(1)\otimes A}(p_{1}\otimes a_{1}, \dots, p_{n}\otimes a_{n}):=\pm\left(  p_{1} \cdots p_{n}\right) \otimes m_{n}^{A}( a_{1},\otimes ,a_{n}),
		\]
		where the sign $\pm$ follows by the signs rule. In particular $i_{j}\otimes Id\: : \:  \Omega(1)\otimes A\to  A $ is a well-defined strict $A_{\infty}$-morphism. 
		\item The map $(g\otimes f)_{\bullet}\: : \:  \left( \Omega(1)\otimes A,m_{\bullet}^{ \Omega(1)\otimes A}\right)\to \left( \Omega(1)\otimes B,m_{\bullet}^{ \Omega(1)\otimes B}\right) $ defined by
		\[
		(g\otimes f)_{n}(p_{1}\otimes a_{1}, \dots, p_{n}\otimes a_{n}):= \pm g\left(  p_{1} \cdots p_{n}\right) \otimes f_{n}( a_{1},\dots,a_{n}),
		\]
		where the signs $\pm$ follows by the signs rule, is a morphism of $A_{\infty}$-algebras. If $g=id$ we have $\left( i_{j}\otimes Id\right) (Id\otimes f)_{\bullet} =f_{\bullet}\left( i_{j}\otimes Id\right) $ for $j=0,1$.
	\end{enumerate}
\end{lem}
\begin{proof}
	Straightforward calculation.
\end{proof}

\begin{defi}\label{homotopy}
	Let $f_{\bullet}, g_{\bullet}\: : \: A\to B$ be $A_{\infty}$-morphisms (resp. $C_{\infty}$-morphism). A \emph{homotopy} between $f_{\bullet}$ and $g_{\bullet}$ is an $A_{\infty} $ (resp. $C_{\infty}$) map $H_{\bullet}\: : \: A \to  \Omega(1)\otimes B$ such that
	\[
	\left( i_{0}\otimes Id\right)H_{\bullet}=f_{\bullet}, \quad \left( i_{1}\otimes Id\right)H_{\bullet}=g_{\bullet},
	\]
	two morphisms are homotopy equivalent if there exists a finite sequence of homotopy maps connecting them.
\end{defi}
By Lemma \ref{tensorinfinity} we have the following. Let $f_{\bullet}, g_{\bullet}\: : \: A\to B$ be $A_{\infty}$-morphisms. Let $H_{\bullet}\: : \: A \to  \Omega(1)\otimes B$ be a homotopy between $f_{\bullet}$ and $g_{\bullet}$, let $p^{1}_{\bullet}\: : \: A_{1}\to A$ and $p^{2}_{\bullet}\: : \: B\to A_{2}$ be $A_{\infty}$-maps. Then $H_{\bullet}(Id\otimes p^2)_{\bullet}$ is a homotopy between $f_{\bullet}p^{2}_{\bullet}$ and $g_{\bullet}p^{2}_{\bullet}$ and $(Id\otimes p^1)_{\bullet}H_{\bullet}$ is a homotopy between $p^{1}_{\bullet}f_{\bullet}$ and $p^{1}_{\bullet}g_{\bullet}$.\\
\begin{prop}\label{existence of a homotopical inverse}
Let $P_{\infty}$ be $A_{\infty}$ or $C_{\infty}$.  Then any $P_{\infty}$-quasi-isomorphism $f_{\bullet}\: : \: A\to B$ has a $P_{\infty}$-homotopical inverse, i.e. there exists a $P_{\infty}$-map $g_{\bullet}\: : \: B\to A$ such that $g_{\bullet}\circ f_{\bullet}\cong Id_{A}$ and $f_{\bullet}\circ g_{\bullet}\cong Id_{B}$.
\end{prop}
\begin{proof}
This is Theorem 3.6 of \cite{BrunoHomotopy}. The $P_{\infty}$ objects enjoy this property because they are fibrant-cofibrant objects in a certain model category structure where $\Omega(1)\otimes \left(-\right) $ is a functorial cylinder object.
\end{proof}

\begin{defi}
	Let $P_{\infty}$ be $A_{\infty}$ or $C_{\infty}$. We denote by $P_{\infty}-\operatorname{ALG}$ the category of bounded below $P_{\infty}$-algebras. We denote by  $ P_{\infty}-\operatorname{ALG}_{\geq0}$ and $ P_{\infty}-\operatorname{ALG}_{>0}$ the two full subcategories whose objects are non-negatively graded $P_{\infty}$-algebras and positively graded $P_{\infty}$-algebras respectively.
\end{defi}
For a bounded below $P_{\infty}$-algebra $\left(A, m_{\bullet}^{A} \right)$ we denote by $\left(\overline{A}, m_{\bullet}^{A} \right)$ its positively graded $P_{\infty}$-subalgebra where $\overline{A}:=\oplus_{i>0} A^{i}$.
Let $A=\oplus_{i\leq 2} A^{i}$ be a bounded below graded vector space. A \emph{$1$-truncated $A_{\infty}$-algebra} or \emph{$1-A_{\infty}$-algebra} consists in a sequence of coderivations 
\[
\begin{tikzcd}
\dots \arrow{r}{\delta^{A}} &\left( T^{c}\left(A[1] \right)\right)^{-2} \arrow{r}{\delta^{A}} & \left( T^{c}\left(A[1] \right)\right)^{-1} \arrow{r}{\delta^{A}} &  \left( T^{c}\left(A[1] \right)\right)^{0}\arrow{r}{\delta^{A}} & \left( T^{c}\left(A[1] \right)\right)^{1}
\end{tikzcd}
\]
such that $\left(\delta^{A} \right)^{2}=0$. Equivalently, they corresponds to a sequences of degree $n-2$ maps $m_{n}\: :\: \oplus_{i\leq n}\left(  A^{\otimes n}\right)^{i}\to A$ that satisfies \eqref{rel} on $\oplus_{i\leq n-1}\left(  A^{\otimes n}\right)^{i}$.  The structure is said to be $1-C_{\infty}$ if $\delta^{A}$ vanishes on non-trivial shuffles of degree smaller than $1$. A \emph{$1$-$A_{\infty}$-algebra morphism} consists in a graded coalgebra maps $$F\: : \: \oplus_{i\leq 1}\left( T^{c}\left(A[1] \right)\right)^{i}  \to  \oplus_{i\leq 1}\left( T^{c}\left(B[1] \right)\right)^{i} $$ such that $\left(F\otimes F \right) \Delta=\Delta F $ and $F\delta^{A}=\delta^{B} F$. Note that $F$ corresponds to a family degree $n-1$ maps $f_{n}\: :\: \oplus_{i\leq n+1}\left(  A^{\otimes n}\right)^{i}\to A\to B$ that satisfies \eqref{tricksign} on $\oplus_{i\leq n}\left(  A^{\otimes n}\right)^{i}$ for any $n>0$. We denote the $1-A_{\infty}$-morphism $F$ by ${f}_{\bullet}\: : \: \left(A, m_{\bullet}^{A} \right)\to \left(B, m_{\bullet}^{B} \right) $. If the two algebras are $1-C_{\infty}$ then $F$ is said to be $1-C_{\infty}$ if it vanishes on non-trivial shuffles of total degree smaller than $2$. 
\begin{defi}
	Let $P_{\infty}$ be $A_{\infty}$ or $C_{\infty}$. We denote by $1-P_{\infty}-\operatorname{ALG}$ the category of $1-P_{\infty}$-algebras and by $ P_{\infty}-\operatorname{ALG}_{\geq0}$ and $ P_{\infty}-\operatorname{ALG}_{>0}$ the two full subcategories whose objects are non-negatively graded $P_{\infty}$-algebras and positively graded $P_{\infty}$-algebras respectively.
\end{defi}
There is an obvious forgetful functor $\mathcal{F}\: : \:  P_{\infty}-\operatorname{ALG}\to  1-P_{\infty}-\operatorname{ALG}$ which is left adjoint to the functor $\mathcal{E}\: : \:  1-P_{\infty}-\operatorname{ALG}\to  P_{\infty}-\operatorname{ALG}$ obtained by extending the $1-P_{\infty}$ into a $P_{\infty}$-structure by $0$. We define the cohomology of a $1-P_{\infty}$ algebra $\left(A, m_{\bullet} \right)$ as the cohomology of the complex
\[
\begin{tikzcd}
\dots \arrow{r}{m_{1}} & A^{1}\arrow{r}{m_{1}} & A'\arrow{r}{0}& 0.
\end{tikzcd}
\]
where $A'\subset A^{2}$ is the subspace generated by $m_{n}(a_{1}, \dots a_{n})$ such that all the $a_{i}$ are $m_{1}$-closed. The cohomology of a $1-P_{\infty}$-algebra is again a a $1-P_{\infty}$-algebra and a $1-P_{\infty}$-morphism induces a $1-P_{\infty}$-morphism in cohomology. A $1-P_{\infty}$-morphism ${f}_{\bullet}\: : \: \left(A, m_{\bullet}^{A} \right)\to \left(B, m_{\bullet}^{B} \right) $ is an $1$-isomorphism if $f_{1}\: : \: A^{i}\to A^{i}$ is an isomorphism for $i\leq 1$ and $f_{1}\: : \: A^{2}\to A^{2}$ is an injection.  ${f}_{\bullet}$ is a $1$-quasi-isomorphism if ${f}_{\bullet}$ induces a $1$-isomorphism in cohomology. Let $f_{\bullet}, g_{\bullet}\: : \: A\to B$ be $1$-$P_{\infty}$-morphisms. A \emph{homotopy} between $f_{\bullet}$ and $g_{\bullet}$ is a $1$-$P_{\infty} $ map $H_{\bullet}\: : \: A \to  \Omega(1)\otimes B$ such that
\[
\left( i_{0}\otimes Id\right) H_{\bullet}=f_{\bullet}, \quad \left( i_{1}\otimes Id\right)_{\bullet}H_{\bullet}=g_{\bullet}.
\]
We conclude this subsection by defining $L_{\infty}$-structures. Let $S(V)$ be the symmetric algebra, then we denote by $S^{c}(V)$ the graded coalgebra given by the graded vector space $S(V)$ and the concatenation coproduct $\Delta$, i.e. $S^{c}(V)$ is the cocommutative cofree conilpotent coalgebra generated by $V$.
\begin{defi}
An \emph{$L_{\infty}$-structure} on a graded vector space $V$ is a coderivation $\delta\: : \: S^{c}(V[1])\to S^{c}(V[1])$ of degree $+1$ such that $\delta^{2}=0$.
\end{defi}
As for the $A_{\infty}$ case, the above definition is equivalent to family of maps $l_{n}\: : \: V^{\otimes n}\to V$ of degree $2-n$. These maps are skew symmetric and satisfy the relations
\[
\sum_{p+q=n+1,\,p,q>1 }\sum_{\sigma^{-1}\in Sh^{-1}(p,q)}\operatorname{sgn}(\sigma)(-1)^{(p-1)q}l_{p}\left(l_{q}\otimes Id^{\otimes \left( p-1\right) } \right)^{\sigma}=\partial l_{n}, 
\]
for $n\geq 1$.
The morphisms, quasi-isomorphisms and isomorphisms between $L_{\infty}$-algebras are defined in the same way, as for the $A_{\infty}$ case. \\
Given an associative algebra $A $, then it carries a Lie algebra structure where the bracket is obtained by anti-symmetrizing the product. The same is true between $A_{\infty}$ and $L_{\infty}$-algebras.
\begin{thm}\label{antilie}
Let $(V, m_{\bullet})$ be an $A_{\infty}$-algebra. The anti-symmetrized map $l_{n}\: : \: V^{\otimes}\to V$, given by
\[
l_{n}:=\sum_{\sigma\in S_{n}}\operatorname{sgn}(\sigma)m_{n}^{\sigma},
\]
define an $L_{\infty}$-algebra structure on $V$.
\end{thm}

\begin{defi}\label{Linftydef}
	Given a $L_{\infty}$-algebra $\mathfrak{g}$ with structure maps $l_{\bullet}$.
	\begin{enumerate}
 \item An \emph{$L_{\infty}$-ideal} $I\subset \mathfrak{g}$ is a subgraded vector space such that $l_{k}(a_{1}, \dots ,a_{k})\in I$ if one of the $a_{i}$ lies in $I$ (in particular $\left( \mathfrak{g}/I,l_{\bullet}\right) $ is an $L_{\infty}$-algebra). 
		\item An $L_{\infty}$-algebra is said to be \emph{filtered} if it is equipped with a filtration $F^{\bullet}$ of $L_{\infty}$-ideals such that for $a_{i}\in F^{n_{i}}(\mathfrak{g})$, we have $l_{k}(a_{1}, \dots ,a_{k})\in F^{n_{1}+\dots +n_{k}}(\mathfrak{g})$.  
		\item A filtered $L_{\infty}$-algebra is said to be \emph{complete} if $\mathfrak{g}\cong \lim_{i}\mathfrak{g}/F^{i}(\mathfrak{g})$ as a graded vector space.
		\item A \emph{Maurer-Cartan element} in a complete $L_{\infty}$-algebra $\mathfrak{g}$ is a $\alpha\in \mathfrak{g}^{1}$ such that
					 \begin{equation*}\label{infinitymaurercartan0}
					\partial(\alpha)+\sum_{k\geq 2}^{\infty}\frac{l_{k} \left(\alpha,\dots, \alpha \right)  }{k!}=0 .
					\end{equation*}
				We denote by $MC(\mathfrak{g})$ the set of Maurer-Cartan elements\footnote{ Notice that the above sum is well-defined in $\mathfrak{g}$ since it is complete.}.
		\item Let $\mathfrak{g} $ be a complete $L_{\infty}$-algebra. Then $\Omega(1)\widehat{\otimes}\mathfrak{g}$ is again a complete $L_{\infty}$-algebra. An \emph{homotopy} between two Maurer-Cartan elements $\alpha_{0}, \alpha_{1}\in MC(\mathfrak{g})$ is a Maurer-Cartan element $\alpha(t)\in  MC(\Omega(1)\widehat{\otimes}\mathfrak{g})$ such that $\alpha(0)=\alpha_{0}$ and $\alpha(1)=\alpha_{1}$. Two Maurer-Cartan elements are said to be \emph{homotopy equivalent} or \emph{ homotopic } if they are connected by a finite sequence of homotopies.
	\end{enumerate} 
\end{defi}

\begin{defi}
	We denote by $\left( \mathcal{L}_{\infty}-\operatorname{ALG}\right)_{p}$ be the category whose objects are $L_{\infty}$-algebras and the arrows are defined as follows: an arrow $\mathfrak{g}\to \mathfrak{g}' $ is a set map $f\: :\: MC(\mathfrak{g})\to MC(\mathfrak{g})$.
	\end{defi}

\subsection{Convolution $L_{\infty}$-algebras}
Let $C$ be a coalgebra and let $A$ be a differential graded algebra. The space of morphisms between graded vector spaces $\Hom^{\bullet}\left(C, A \right)$ is equipped with a differential graded Lie algebra structure called convolution algebra (see \cite{lodayVallette}, chapter 1).\\
Let $\left(V, m_{\bullet}^{V}\right) ,\left(A, m_{\bullet}^{A} \right) $ be $A_{\infty}$-algebras and assume they are both bounded below.
Let $\delta$ be the codifferential of $T^{c}(V[1])$ and $\delta^{A}$ be the codifferential on $T^{c}(A[1])$. Consider $A^{\bullet}$ as a graded cochain vector space equipped with $m_{1}^{A}$ as differential. We have a differential graded vector space of morphisms between graded vector spaces
\[
\Hom^{\bullet}\left(T^{c}(V[1]), A \right).
\]
If $m_{n}^{A}=0$ for $n>2$, there is a one to one correspondence between coalgebras morphism $F\: : \:T^{c}(V[1])\to T^{c}(A[1])$ and twisting cochains. For general $A_{\infty}$-structure, there is a similar property. For each $F$, we associate a graded map $\alpha\in \Hom^{1}\left(T^{c}(V[1]), A \right)$ defined as
\begin{equation}\label{12}
\begin{tikzcd}
T^{c}(V[1]) \arrow{r}{F} & T^{c}(A[1])\arrow{r}{\text{proj}_{A[1]}} & A[1]\arrow{r}{s^{-1}}& A.
\end{tikzcd}
\end{equation}
with $\alpha(1)=0$. The condition $F\circ \delta|_{V[1]^{\otimes n}}= \delta^{A}\circ F|_{V[1]^{\otimes n}}$ reads 
\[
\left( \alpha \circ \delta^{V}\right) _{V[1]^{\otimes n}}=\sum_{k\geq 1,\, i_{1}+\dots i_{k}=n}(-1)^{k+1}m_{k}^{A}\left(\alpha_{i_{1}}, \dots, \alpha_{i_{k}} \right)\circ \Delta^{k-1}
\]
where $\alpha_{i_{j}}:=\alpha|_{V[1]^{\otimes i_{j}}}$ and $\Delta^{k}$ is the iterated coproduct in the tensor coalgebra $T^{c}(V[1])$. It is an easy exercise to show that $$\tilde{m}^{A}_{k}:=(-1)^{k}m_{k}^{A}\: : \: A^{\otimes k}\to A$$ is again an $A_{\infty}$-structure on $A$. We conclude
\begin{equation}\label{newtwist}
\alpha \circ \delta=-\sum_{k\geq 1}\tilde{m}^{A}_{k}\left(\alpha , \dots, \alpha\right)\circ \Delta^{k-1}.
\end{equation}
The above equation can be interpreted as the $A_{\infty}$-version of the twisting cochain condition. 
For $n>1$, we define the maps $M_{n}\: : \: \left( \Hom^{\bullet}\left(T^{c}(V[1]), A \right)\right)^{\otimes n}\to \Hom^{\bullet}\left(T^{c}(V[1]), A \right)$ via
\[
M_{n}(f_{1}, \dots, f_{n}):=\tilde{m}^{A}_{n}\left(f_{1}, \dots, f_{n}\right)\circ \Delta^{n-1},
\]
the map $M_{1}\: : \: \Hom^{\bullet}\left(T^{c}(V[1]), A \right)\to \Hom^{\bullet}\left(T^{c}(V[1]), A \right)$ as $M_{1}(f):=\tilde{m}^{A}_{1}(f)$ and the map $\partial\: : \: \Hom^{\bullet}\left(T^{c}(V[1]), A \right)\to \Hom^{\bullet}\left(T^{c}(V[1]), A \right)$ as
\[
\partial(f):=\tilde{m}^{A}_{1}f- (-1)^{|f|}f\circ \delta.
\]
 We define $L_{V[1]^{*}}\left(A \right)$ be the set of morphisms $\Hom^{\bullet}(T^{c}(V[1]), A)$ whose kernel contains the set of non-trivial shuffles $NTS(V[1])$.

\begin{lem}\label{Algebrainfinitystruct}
Let $V$, $A$ as above. 
\begin{enumerate}
\item $\left(M_{\bullet}, \Hom^{\bullet}\left(T^{c}(V[1]), A \right) \right)$ is an $A_{\infty}$-algebra,
 \item $\left(\partial,\left\lbrace M_{n}\right\rbrace_{n\geq 2},  \Hom^{\bullet}\left(T^{c}(V[1]), A \right) \right)$ is an $A_{\infty}$-algebra,
 \item We denote by $l_{\bullet}'$ the maps on $\Hom^{\bullet}\left(T^{c}(V[1]), A \right)^{\otimes n}$ induced by the anti-symmetrization of the maps $\left(\partial,\left\lbrace M_{n}\right\rbrace_{n\geq 2}\right) $ via Theorem \ref{antilie} and by $l_{\bullet}$ the maps induced by the anti-symmetrization of the maps $\left(\partial, M_{2}, M_{3}, \dots\right) $.  Assume that $\delta\mu'(a,b)=0$ for any non-trivial shuffle $\mu'(a,b)\in T^{c}(V[1])$ and that the product $m_{2}^{A}$ is graded commutative (but not necessarily associative). Then we have two $L_{\infty}$-subalgebras 
 \[
  \left(l_{\bullet}', L_{V[1]^{*}}(A)\right) \subset \left(l_{\bullet}',  \Hom^{\bullet}\left(T^{c}(V[1]), A \right)\right) ,\quad \left(l_{\bullet}, L_{V[1]^{*}}(A)\right) \subset \left(l_{\bullet},  \Hom^{\bullet}\left(T^{c}(V[1]), A \right)\right) .
 \]
\end{enumerate}
\end{lem}
For the proof see the Appendix, Section \ref{proof}. 
\begin{defi}
Let $\left(V, m_{\bullet}^{V}\right) ,\left(A, m_{\bullet}^{A} \right) $ be $A_{\infty}$-algebras.
\begin{enumerate}
	\item We call $ \left( \partial,{\left\lbrace M_{n}\right\rbrace }_{n>1}\Hom^{\bullet}\left(T^{c}(V[1]), A \right) \right) $ the \emph{convolution $A_{\infty}$-algebra} \\associated to $\left(V, m_{\bullet}^{V}\right) ,\left(A, m_{\bullet}^{A} \right) $.
	\item We call $$
	\operatorname{Conv}_{A_{\infty}}\left(\left(V, m_{\bullet}^{V}\right) ,\left(A, m_{\bullet}^{A} \right) \right):= \left( {l'}_{\bullet},\Hom^{\bullet}\left(T^{c}(V[1]), A \right)\right)$$
	 the \emph{convolution $L_{\infty}$-algebra} associated to $\left(V, m_{\bullet}^{V}\right) ,\left(A, m_{\bullet}^{A} \right) $. 
	\item Let $\left(V, m_{\bullet}^{V}\right) ,\left(A, m_{\bullet}^{A} \right) $ be $C_{\infty}$-algebras. We call
	 \[
		\operatorname{Conv}_{C_{\infty}}\left(\left(V, m_{\bullet}^{V}\right) ,\left(A, m_{\bullet}^{A} \right) \right):= \left({l'}_{\bullet}, L_{V[1]^{*}}(A)\right) \subset \left(l_{\bullet}',  \Hom^{\bullet}\left(T^{c}(V[1]), A \right)\right) 
	 \]
	 the \emph{reduced convolution $L_{\infty}$-algebra } associated to $\left(V, m_{\bullet}^{V}\right) ,\left(A, m_{\bullet}^{A} \right) $. 
\end{enumerate}	
\end{defi}
The next proposition is a consequence of the discussion due at the beginning of this subsection.
\begin{prop}\label{Cinftydictionary}
Let $\left(V, m_{\bullet}^{V}\right) ,\left(A, m_{\bullet}^{A} \right) $ be $P_{\infty}$-algebras. There exists a one to one correspondence between
\begin{enumerate}
	\item  $P_{\infty}$-morphisms $f_{\bullet}\: : \: \left(V, m_{\bullet}^{V}\right) \to\left(A, m_{\bullet}^{A} \right) $. 
\item Morphisms of differential graded quasi-free coalgebras $F\: : \:T^{c}(V[1])\to T^{c}(A[1])$ (resp. such that $F(NTS(V[1]))=0$).
\item Maurer-Cartan elements $\alpha\in \operatorname{Conv}_{P_{\infty}}\left(\left(V, m_{\bullet}^{V}\right) ,\left(A, m_{\bullet}^{A} \right) \right)$ such that $\alpha(1)=0$.
\end{enumerate}
\end{prop}
\begin{defi}
Let $\left(V, m_{\bullet}^{V}\right) ,\left(A, m_{\bullet}^{A} \right) $ be $P_{\infty}$-algebras. Assume that $V$ is positively graded. An element $\alpha\in\operatorname{Conv}^{1}_{P_{\infty}}\left(\left(V, m_{\bullet}^{V}\right) ,\left(A, m_{\bullet}^{A} \right) \right)$ is a \emph{$1$-Maurer-Cartan element} if the corresponding tensor coalgebra morphism $F$ restricted on
\[
\left( T^{c}\left(V[1] \right)\right)^{0}\oplus  \left( T^{c}\left(V[1] \right)\right)^{1}
\]
is a $1-P_{\infty}$-algebra morphism. We say that two $1$-Maurer Cartan elements are equal if they coincide on $\left( T^{c}\left(V[1] \right)\right)^{0}\oplus \left( T^{c}\left(V[1] \right)\right)^{1}$. We denote the sets of $1$-Maurer-Cartan elements with $MC_{1}\left( 	\operatorname{Conv}_{P_{\infty}}\left(\left(V, m_{\bullet}^{V}\right) ,\left(A, m_{\bullet}^{A} \right)\right)\right) $.
\end{defi}
\subsection{ Degree zero convolution $L_{\infty}$-algebras}
Let $(V, m_{\bullet}^{V})$ be a positively graded $1-A_{\infty}$-algebra and let $(A, m_{\bullet}^{A})$ be an $A_{\infty}$ algebra. Let $V[1]^{0}$ be the degree $0$ part of $V[1]$, we consider $\Hom^{\bullet}(T^{c}(V[1]^{0}) , A)\subset \Hom^{\bullet}(T^{c}(V[1]) , A)$ as the graded vector subspace of morphisms with support in $T^{c}(V[1]^{0})$.  We define $L_{V[1]^{*}}^{0}\left(A \right) := L_{V[1]^{*}}\left(A \right)\cap \Hom^{\bullet}(T^{c}(V[1]^{0}) , A)$ and we denote $L_{V[1]^{*}}^{0}\left(\Bbbk\right)$ by $L_{V[1]^{*}}^{0}$. The restriction of the dual $\delta^{*}\: : \: \Hom^{\bullet}(\left( T^{c}(V[1])\right)^{1}\oplus T^{c}(V[1]^{0}) , A)\to\Hom^{\bullet}( T^{c}(V[1]^{0}) , A)$ vanishes on $\Hom^{\bullet}(T^{c}(V[1]^{0}) , A)$ if $V[1]$ is a non-negatively graded vector space (equivalently, if $V$ is positively graded). 
\begin{cor}\label{Linftystructnotquot} Let $\left(M_{\bullet}, \Hom^{\bullet}\left(T^{c}(V[1]), A \right) \right)$ be as above.
	\begin{enumerate}
		\item $\left(M_{\bullet}, \Hom^{\bullet}\left(T^{c}(V[1]^{0}), A \right) \right)$ is an $A_{\infty}$-algebra,
		\item Let $f_{1},\dots,f_{n}\in  \Hom^{\bullet}\left(T^{c}(V[1]^{0}), A \right)$ and assume that there is a $g$ with $\delta^{*}g=f_{i}$ for some $i$, then $M_{n}(f_{1},\dots, f_{n})\in \operatorname{Im}\left(\delta^{*} \right) $ for $n>1$. In particular $$\left(M_{\bullet}, \Hom^{\bullet}\left(T^{c}(V[1]^{0}), A \right)/\operatorname{Im}(\delta^{*}) \right)$$ is an $A_{\infty}$-algebra.
		\item Consider $\Hom^{\bullet}\left(T^{c}(V[1]^{0}), A \right)/{\operatorname{Im}(\delta^{*})}$ equipped with the $L_{\infty}$-structure $l_{\bullet}$ induced by the maps $M_{\bullet}$ via Theorem \ref{antilie}. Assume that $\delta\mu'(a,b)=0$ for any non-trivial shuffle $\mu'(a,b)\in T^{c}(V[1]^{0})$ and that the product $m_{2}^{A}$ is graded commutative. The subgraded vector space $ L_{V[1]^{*}}^{0}\left(A \right)/{\operatorname{Im}(\delta^{*})}$ equipped with $l_{\bullet}$ is a $L_{\infty}$-subalgebra of  $$\left(l_{\bullet}, \Hom^{\bullet}\left(T^{c}(V[1]^{0}), A \right)/{\operatorname{Im}(\delta^{*})}\right).$$
	\end{enumerate}
\end{cor}
For the proof see the Appendix, Section \ref{proof}.
\begin{defi}
We call $$
\operatorname{Conv}_{1-A_{\infty}}\left(\left(V, m_{\bullet}^{V}\right) ,\left(A, m_{\bullet}^{A} \right) \right):= \left( {l}_{\bullet},\Hom^{\bullet}\left(T^{c}(V[1]^{0}), A \right)/{\operatorname{Im}(\delta^{*})}\right) $$ the \emph{degree zero convolution $L_{\infty}$-algebra } associated to $\left(V, m_{\bullet}^{V}\right) ,\left(A, m_{\bullet}^{A} \right) $.\\ Assume that $\left(V, m_{\bullet}^{V}\right) ,\left(A, m_{\bullet}^{A} \right) $ are $1-C_{\infty}$ and $C_{\infty}$ respectively. We call
	   $$
	 	\operatorname{Conv}_{1-C_{\infty}}\left(\left(V, m_{\bullet}^{V}\right) ,\left(A, m_{\bullet}^{A} \right) \right):= \left({l}_{\bullet}, L_{V[1]^{*}}(A)/{\operatorname{Im}(\delta^{*})}\right) \subset \left(l_{\bullet},  \Hom^{\bullet}\left(T^{c}(V[1]^{0}), A \right)/{\operatorname{Im}(\delta^{*})}\right)$$ 
	 the \emph{degree zero reduced convolution $L_{\infty}$-algebra} associated to $\left(V, m_{\bullet}^{V}\right) ,\left(A, m_{\bullet}^{A} \right) $. 
\end{defi}
\begin{prop}\label{Cinftydictionary1}
Let $\left(V, m_{\bullet}^{V}\right) $ be a positively graded $1-P_{\infty}$-algebra and a $P_{\infty}$-algebra respectively. There exists a one to one correspondence between the following.
\begin{enumerate}
	\item  $1-P_{\infty}$-morphism $f_{\bullet}\: : \: \left(V, m_{\bullet}^{V}\right) \to\mathcal{F}\left(B, m_{\bullet}^{B} \right) $. 
	\item The set of degree zero maps $$F=F^{0}\oplus F^{1}\: : \: \left( T^{c}\left(V[1] \right)\right)^{0}\oplus \left( \left( T^{c}\left(V[1] \right)\right)^{0}\right) \to \left( T^{c}\left(B[1] \right)\right)^{0}\oplus \left( \left( T^{c}\left(B[1] \right)\right)^{0}\right)$$ such that \begin{enumerate}
		\item $\left(F\otimes F \right) \Delta=\Delta F $ and $F^1\delta=\delta F^0$, and
		\item if $P_{\infty}=C_{\infty}$ $F \left(NTS\left(V[1] \right) \cap \left( \left( T^{c}\left(V[1] \right)\right)^{0}\oplus  \left( T^{c}\left(V[1] \right)\right)^{0} \right)  \right) =0$,
	\end{enumerate} 
modulo the following equivalence relation: $F,G$ are equivalent if the kernel of $F^{1}-G^{1}$ contains $ \delta\left( \left( T^{c}\left(V[1] \right)\right)^{0}\right)$.
	\item  Maurer-Cartan elements $\alpha\in \operatorname{Conv}_{1-P_{\infty}}\left(\left(V, m_{\bullet}^{V}\right) ,\left(A, m_{\bullet}^{A} \right) \right)$ such that $\alpha(1)=0$. 
\end{enumerate}
Let $\left(V', m_{\bullet}^{V'}\right)$ be a positively graded $P_{\infty}$-algebra such that $\mathcal{F}\left(V', m_{\bullet}^{V'}\right)=\left(V, m_{\bullet}^{V}\right)$. Then the above elements are in one to one correspondence with $1$-Maurer-Cartan elements\\ $\alpha\in \operatorname{Conv}_{P_{\infty}}\left(\left(V', m_{\bullet}^{V'}\right) ,\left(A, m_{\bullet}^{A} \right) \right)$ such that $\alpha(1)=0$. 
\end{prop}
\begin{proof}
Let $F=F^{0}\oplus F^{1}$ be as in point 1. Then $F^{0}$ is a tensor coalgebra map and it is completely determined by an $\alpha \in \operatorname{Hom}^{1}\left( \left(T^{c}\left(V[1] \right)\right)^{0},A \right)= \operatorname{Hom}^{1} \left(T^{c}\left(V^{0}[1] \right),A \right) $. The condition $\delta^{A}F^{0}=F^{1}\delta^{V}$ correspond to the fact that $\alpha$ is a Maurer-Cartan element in the degree zero convolution algebra. Let  $\alpha$ be a Maurer-Cartan element in the degree zero convolution algebra. In particular, $\alpha \in  \operatorname{Hom}^{1} \left(T^{c}\left(V^{0}[1] \right),A \right) $ corresponds to a tensor coalgebra map $F^{0}\: : \: T^{c}\left(V^{0}[1] \right)\to T^{c}\left(A^{0}[1] \right)$ and the Maurer-Cartan equation implies that there exists a $F^{1}\: : \: \left( T^{c}\left(V[1] \right)\right)^{1}\to  \left( T^{c}\left(A[1] \right)\right)^{1}$ such that $\delta^{A}F^{0}=F^{1}\delta^{V}$. This proves the equivalence between $1$ and $2$. The last step follows in a similar way.
\end{proof}
\begin{Warn}
For any Maurer-Cartan $\alpha$ in a convolution algebra (reduced, degre zero,.. ), we will assume $\alpha(1)=0$. 
\end{Warn}
\subsection{Functoriality of convolution $L_{\infty}$-algebras}
Let $\left(V, m_{\bullet}^{V}\right)$ and $\left(W, m_{\bullet}^{W}\right)$ be positively graded $1-P_{\infty}$-algebras and let $\left(A, m_{\bullet}^{A} \right) $ be a $P_{\infty}$-algebra. A $1$-$P_{\infty}$-map $f_{\bullet}\: : \: \left(W, m_{\bullet}^{W}\right)\to \left(V, m_{\bullet}^{V} \right)$ corresponds via Proposition \ref{Cinftydictionary1} to certain morphism $F=F^{0}\oplus F^{1}$. In the same way, a Maurer-Cartan element $\alpha$ in $\operatorname{Conv}_{1-P_{\infty}}\left(\left(V, m_{\bullet}^{V}\right) ,\left(A, m_{\bullet}^{A} \right) \right)$ can be written as ${F'}={F'}^{1}\oplus {F'}^{2}$. We denote with $f^{*}(\alpha)$ the Maurer-Cartan element that corresponds to the $1$-morphism $F{F'}={F}^{1}{F'}^{1}\oplus{F}^{2} {F'}^{2}$, this gives a well-defined map 
$$f^{*}\: : \: \operatorname{Conv}_{1-P_{\infty}}\left(\left(V, m_{\bullet}^{V}\right) ,\left(A, m_{\bullet}^{A} \right) \right)\to \operatorname{Conv}_{1-P_{\infty}}\left(\left(W, m_{\bullet}^{W}\right) ,\left(A, m_{\bullet}^{A} \right) \right).$$ 
Let $\left(V', m_{\bullet}^{V'}\right)$ and $\left(W', m_{\bullet}^{W'}\right)$ be $P_{\infty}$-algebras and let $g_{\bullet}\: : \: \left(V', m_{\bullet}^{V'}\right)\to\left(W', m_{\bullet}^{W'}\right) $ be a $P_{\infty}$-morphism. In the same way we have a map
$$g^{*}\: : \: \operatorname{Conv}_{P_{\infty}}\left(\left(V', m_{\bullet}^{V'}\right) ,\left(A, m_{\bullet}^{A} \right) \right)\to \operatorname{Conv}_{P_{\infty}}\left(\left(W', m_{\bullet}^{W'}\right) ,\left(A, m_{\bullet}^{A} \right) \right).$$ 
Assume that $V'$ is positively graded.
Consider $\pi:={p} \circ r$, where 
\begin{equation}\label{restrquot}
r\: : \: \Hom^{\bullet}\left(T^{c}(V'[1]), A \right)\to \Hom^{\bullet}\left(T^{c}(V'[1]^{0}), A \right)
\end{equation} 
is the restriction map and ${p}\: : \: \Hom^{\bullet}\left(T^{c}(V'[1]^{0}), A \right)\to \Hom^{\bullet}\left(T^{c}(V'[1]^{0}), A \right)/{\operatorname{Im}(\delta^{*})}$ is the quotient map. We have a well defined map 
\begin{eqnarray}\label{notcompl}\pi\: : \: \operatorname{Conv}_{P_{\infty}}\left(\left(V', m_{\bullet}^{V'}\right) ,\left(A, m_{\bullet}^{A} \right) \right)\to \operatorname{Conv}_{1-P_{\infty}}\left(\left(V', m_{\bullet}^{V'}\right) ,\left(A, m_{\bullet}^{A} \right) \right).
\end{eqnarray}

\begin{prop}\label{pushforward}
Let $g_{\bullet}, f_{\bullet}$ as above and let $\left(A, m_{\bullet}^{A} \right) $ be a $P_{\infty}$-algebra. Assume that $V$ is positively graded.
	\begin{enumerate}
\item The map
		$$f^{*}\: : \: \operatorname{Conv}_{1-P_{\infty}}\left(\left(V, m_{\bullet}^{V}\right) ,\left(A, m_{\bullet}^{A} \right) \right) \to\operatorname{Conv}_{1-P_{\infty}}\left(\left(W, m_{\bullet}^{W}\right) ,\left(A, m_{\bullet}^{A} \right) \right) $$ 
			is a strict morphism of $L_{\infty}$-algebras.
 \item The map
 		$$g^{*}\: : \: \operatorname{Conv}_{P_{\infty}}\left(\left(V', m_{\bullet}^{V'}\right) ,\left(A, m_{\bullet}^{A} \right) \right)\to \operatorname{Conv}_{P_{\infty}}\left(\left(W', m_{\bullet}^{W'}\right) ,\left(A, m_{\bullet}^{A} \right) \right) $$ 
 			is a strict morphism of $L_{\infty}$-algebras.			
			
		\item Assume that $\left(V', m_{\bullet}^{V'}\right)$, $\left(W', m_{\bullet}^{W'}\right)$ are positively graded $P_{\infty}$-algebras such that\\ $\mathcal{F}\left(V', m_{\bullet}^{V'}\right)=\left(V, m_{\bullet}^{V}\right)$ and $\mathcal{F}\left(W', m_{\bullet}^{W'}\right)=\left(W, m_{\bullet}^{W}\right)$. $f^{*}$ is a well defined map between set of $1$-Maurer-Cartan elements. 
			\item Assume that $\left(V', m_{\bullet}^{V'}\right)$, $\left(W', m_{\bullet}^{W'}\right)$ as in point 3. $\pi$ is a strict morphism of $L_{\infty}$-algebras. In particular, it sends $1$-Maurer-Cartan elements to Maurer-Cartan elements and $f^{*}\pi= \pi f^{*}$.
	\end{enumerate} 
\end{prop}
\begin{proof}
	Direct verification.
\end{proof}
 Let $\left(B, m_{\bullet}^{B} \right) $ be a $P_{\infty}$-algebra. For a $P_{\infty}$-map $h_{\bullet}\: : \: \left(A, m_{\bullet}^{A} \right)\to \left(B, m_{\bullet}^{B} \right)$ we have a well-defined map 
$$h^{*}\: : \: \operatorname{Conv}_{1-P_{\infty}}\left(\left(V, m_{\bullet}^{V}\right) ,\left(A, m_{\bullet}^{A} \right) \right)\to \operatorname{Conv}_{1-P_{\infty}}\left(\left(V, m_{\bullet}^{V}\right) ,\left(B, m_{\bullet}^{B} \right) \right)$$ 
defined as follows. Let $\alpha\in \operatorname{Conv}_{1-P_{\infty}}\left(\left(V, m_{\bullet}^{V}\right) ,\left(A, m_{\bullet}^{A} \right) \right)$ and let $H$ and $F$ be the coalgebras morphism corresponding to $h_{\bullet}$ and $\alpha$ respectively. We consider the composition $F H$. It  corresponds to a Maurer-Cartan element $h_{*}(\alpha)\in \operatorname{Conv}_{1-P_{\infty}}\left(\left(V, m_{\bullet}^{V}\right) ,\left(B, m_{\bullet}^{B} \right) \right)$ given by
\[
h_{*}(\alpha)|_{V[1]^{\otimes n}}=\sum _{l=1}^{n}\pm h_{l}\left(\sum_{i_{1}+\dots i_{l}=n} \alpha|_{V[1]^{\otimes i_{1}}}\otimes \cdots \otimes \alpha|_{V[1]^{\otimes i_{l}}} \right) 
\]
where the signs are a consequence of the Koszul convention. In particular, if $h_{\bullet}$
is strict we have $h_{*}(\alpha)=h \alpha$.
\begin{prop}\label{pullback}\label{reductiontoquotient morphism}
	Let $h_{\bullet}$ be as above.
	\begin{enumerate}
	\item  The map
\[
h_{*}\: : \:MC\left( \operatorname{Conv}_{1-P_{\infty}}\left(\left(V, m_{\bullet}^{V}\right) ,\left(A, m_{\bullet}^{A} \right) \right)\right)  \to MC\left( \operatorname{Conv}_{1-P_{\infty}}\left(\left(V, m_{\bullet}^{V}\right) ,\left(B, m_{\bullet}^{B} \right) \right) \right)  
\]
is well-defined.
	\item The map
	\[
	h_{*}\: : \:MC\left( \operatorname{Conv}_{P_{\infty}}\left(\left(V', m_{\bullet}^{V'}\right) ,\left(A, m_{\bullet}^{A} \right) \right)\right)  \to MC \left( \operatorname{Conv}_{P_{\infty}}\left(\left(V', m_{\bullet}^{V'}\right) ,\left(B, m_{\bullet}^{B} \right) \right) \right). 
	\]
	\item Let $\overline{h}_{\bullet}\: : \: \mathcal{F}\left(A, m_{\bullet}^{A} \right)\to\mathcal{F} \left(B, m_{\bullet}^{B} \right)$ be a $1-P_{\infty}$-morphism. Then $\overline{h}_{*}$ is well-defined and the above statements are true as well if we replace ${h}_{*}$ by $\overline{h}_{*}$.
\item  Assume that $\left(V', m_{\bullet}^{V'}\right)$ is a positively graded $P_{\infty}$-algebras such that $\mathcal{F}\left(V', m_{\bullet}^{V'}\right)=\left(V, m_{\bullet}^{V}\right)$. The map $h_{*}$ preserves $1$-Maurer-Cartan elements and $h_{*}\pi= \pi h_{*}$.
\end{enumerate}
\end{prop}
\begin{proof}
	Direct verification.
\end{proof}
The propositions above and corollaries can be summarized as follows.
\begin{thm}\label{functorLinfty}
 We have functors
\begin{align*}
	&\operatorname{Conv}_{P_{\infty}}\: : \: \left(  {P}_{\infty}-\operatorname{ALG}\right)^{op}\times  {P}_{\infty}-\operatorname{ALG}\to \left( \mathcal{L}_{\infty}-\operatorname{ALG}\right)_{p},\\	
	&\operatorname{Conv}_{1-P_{\infty}}\: : \: \left( 1-P_{\infty}-\operatorname{ALG}_{>0}\right)^{op}\times  {P}_{\infty}-\operatorname{ALG}\to \left( \mathcal{L}_{\infty}-\operatorname{ALG}\right)_{p}.
\end{align*}
\end{thm}

\subsection{Formal power series}
We want to express convolution $L_{\infty}$-algebras of the previous section in terms of formal power series, to do that we need to put some filtrations on $T(V)$.
We fix a finite type graded vector space $V$ which is bounded below at $k$. We define two filtrations.
Let $\epsilon\: : \: T(V)\to \Bbbk$ be the augmentation map, let $I:=\ker (\epsilon)$ be the augmentation ideal. The sequence given by the powers of $I$
\[
I^{0}:=T(V)\subset I\subset I^{2}\subset \dots
\]
is a filtration $I^{\bullet}$ on $T(V)$. We define the filtration $F^{\bullet}$ on $V$ as
$
F^{l}\left( V^{\bullet}\right):=\oplus_{j\geq k+l}V^{j}.
$
 Let $\left\lbrace v^{l}_{i}\right\rbrace_{i\in J(l)} $ be a basis of $V^{l}$, hence $$\bigcup_{i\geq 0 } \bigcup_{j\in J(k+i) }v^{k+i}_{j}$$ is a basis of $V$. We denote such a basis by $\left\lbrace v_{i}\right\rbrace_{i\in J(l), l\geq 0}$. An element of $T(V)$ can be written in a unique way as a non commutative polynomial
 \[
 \displaystyle\sum_{p\geq 0}^{N} \displaystyle \sum_{(i_{1},\dots ,i_{p})}\lambda_{i_{1},\dots ,i_{p}} v_{i_{1}}\cdots v_{i_{p}},
 \]
 such that only finitely many $\lambda_{i_{1},\dots ,i_{p}} $ are different from $0$. The elements $f\in I^{i}\left( T(V)\right) \subset T(V)$ can be written in unique way as 
\[
 \displaystyle\sum_{p\geq i}^{N} \displaystyle \sum_{(i_{1},\dots ,i_{p})}\lambda_{i_{1},\dots ,i_{p}} v_{i_{1}}\cdots v_{i_{p}},
\]
 such that only finitely many $\lambda_{i_{1},\dots ,i_{p}} $ are different from $0$. Let $A_{i}^{I}$ be the set of monomials $v_{i_{1}}\cdots v_{i_{p}}$
 such that $p<i$, it forms a basis for $T(V)/I^{i}$. The tensor algebra $T\left(V/F^{j}\left(V \right)\right)$ may be considered as the graded algebra of non commutative polynomial $\Bbbk \left\langle \left\lbrace v_{i}\right\rbrace_{i\in J(l), l < j}\right\rangle $, where $\left\lbrace v_{i}\right\rbrace_{i\in J(l), l< j}$ is a basis of $V^{k}\oplus \dots V^{k+j-1}$. For each $j\geq 0$, let $G^{j}\left(T(V)\right)\subset T(V)$ be the subspace generated by
\[
 \displaystyle\sum_{p\geq 0}^{N} \displaystyle \sum_{(i_{1},\dots ,i_{p})}\lambda_{i_{1},\dots ,i_{p}} v_{i_{1}}\cdots v_{i_{p}}\in \Bbbk \left\langle \left\lbrace v_{i}\right\rbrace_{i\in J(l), l\geq 0}\right\rangle,  
\]
such that for each $v_{i_{1}}\cdots v_{i_{p}}$ there exists a $v_{i_{n}}\in V^{s}$, $s\geq j+k$. In particular
\[
G^{j}\left(T(V) \right)\oplus T\left(V/F^{j}\left(V \right)\right) =T(V).
\]
A basis $A^{G}_{j}$ of $T(V)/G^{j}\left(T(V) \right)= T\left(V/F^{j}\left(V \right)\right)$ is given by the monomials
$
v_{i_{1}}\cdots v_{i_{p}}
$
such that each $v_{i_{j}}\in V^{k}\oplus \dots V^{k+j-1}$. Let $$H_{i,j}:=\left(  T(V)/I^{i}(T(V))\right)/G^{j}( T(V)),$$ we identify $H_{i,j}$ as the subspace of $T(V)$ generated by $A_{i,j}:=A^{G}_{j}\setminus A^{I}_{i}$. They form a diagram where all the maps are inclusion and two objects $H_{a,b}$, $H_{c,d}$ are connected by a map $H_{a,b}\hookrightarrow H_{c,d}$ if $a\leq b$ and $c\leq d$, in particular $\colim_{i,j} H_{i,j}=T(V)$. Let $W$ be a graded vector space equipped with the trivial filtration. We consider the tensor product of filtrations
\[
I^{i}(W\otimes T(V)):=W\otimes I^{i}(T(V)), \quad G^{j}(W\otimes T(V)):=W\otimes G^{j}(T(V)).
\]
Then $W\otimes  T(V)\cong W\otimes \Bbbk \left\langle \left\lbrace v_{i}\right\rbrace_{i\in I(l), l\geq 0}\right\rangle$. In particular each element can be written 
in unique way as 
\begin{equation}\label{exampleform}
 \displaystyle\sum_{p\geq 0}^{N} \sum_{q\geq 0}^{M}\sum_{v_{i_{1}}\cdots v_{i_{p}}\in A_{p,q}} w_{i_{1},\dots ,i_{p}}\otimes v_{i_{1}}\cdots v_{i_{p}}
\end{equation}
 for some $N$, $M$. Let $W\widehat{\otimes}_{I} \Bbbk \left\langle \left\langle \left\lbrace v_{i}\right\rbrace_{i\in J(l), l\geq 0} \right\rangle  \right\rangle$ be the completion of $W\otimes  T(V)$ with respect to $I^{\bullet}$. It is the graded vector space of formal power series of the form $\eqref{exampleform}$ with $N=\infty$. On the other hand, let $W\widehat{\otimes}_{G} \Bbbk \left\langle \left\lbrace v_{i}\right\rbrace_{i\in J(l), l\geq 0}\right\rangle$ be the completion of $W\otimes  T(V)$ with respect to $G^{\bullet}$, it is the formal graded vector space of formal power series $\eqref{exampleform}$ with $M=\infty$. We denote by $\mathcal{I}_{G}^{\bullet}$ the filtration obtained by the completion of $I^{\bullet}$ with respect to $G^{\bullet}$ on $W\widehat{\otimes}_{G} \Bbbk \left\langle \left\lbrace v_{i}\right\rbrace_{i\in J(l), l\geq 0}\right\rangle$. We denote by $\mathcal{G}_{I}^{\bullet}$ the filtration obtained by the completion of $G^{\bullet}$ with respect to $I^{\bullet}$ on $W\widehat{\otimes}_{I} \Bbbk \left\langle \left\lbrace v_{i}\right\rbrace_{i\in J(l), l\geq 0}\right\rangle$. The completion of $W\widehat{\otimes}_{G} \Bbbk \left\langle \left\lbrace v_{i}\right\rbrace_{i\in J(l), l\geq 0}\right\rangle$ with respect to $\mathcal{I}_{G}^{\bullet}$ coincides with the completion of $W\widehat{\otimes}_{I} \Bbbk \left\langle \left\lbrace v_{i}\right\rbrace_{i\in J(l), l\geq 0}\right\rangle$ with respect to $\mathcal{G}_{I}^{\bullet}$. We denote the obtained vector space by $ W\widehat{\otimes}\widehat{T}(V)$, it is the vector space of formal power series
\[
   \displaystyle\sum_{p\geq l}^{\infty} \sum_{q\geq 0}^{\infty}\sum_{v_{i_{1}}\cdots v_{i_{p}}\in A_{p,q}} w_{i_{1},\dots ,i_{p}}\otimes v_{i_{1}}\cdots v_{i_{p}}
 \]
 We denote the induced filtrations on the completion by $\mathcal{I}$ and resp. $\mathcal{G}$. There exists a canonical isomorphism of complete graded vector spaces
\[
\Psi\: : \: \Hom\left(T^{c}\left(V[1] \right) , W \right) \to \widehat{T}\left(\left( V[1]\right)^{*}  \right)\widehat{\otimes} W
\]
In particular $I^{\bullet}$ and $G^{\bullet}$ induce a filtration on the Lie algebra of primitive elements of $\mathbb{L}\left(\left( V[1]\right)^{*} \right)\subset T\left(\left( V[1]\right)^{*}  \right)$ and $\Psi$ restricts to an isomorphism 
\[
\Psi\: : \: L_{V[1]^{*}}\left(W \right)\to \widehat{\mathbb{L}}\left(\left( V[1]\right)^{*} \right)\widehat{\otimes} W.
\]
The filtration $\mathcal{I}$ on $\Hom^{\bullet}\left(T^{c}(V[1])^{\otimes n}, W \right)$ is given as follows: $\mathcal{I}^{i}$ is the graded vector subspace of morphisms in $\Hom^{\bullet}\left(T^{c}(V[1])^{\otimes n}, W \right)$ such that $f|_{I^{i}}\cong 0$. For each $n>0$, we consider $\Hom^{\bullet}\left(T^{c}(V[1]), A \right)^{\otimes n}$ equipped with the tensor product filtration $\mathcal{I}^{\otimes n}$. 
Let $(A, m_{\bullet}^{A})$ and $(V',m_{\bullet}^{V'})$ be $P_{\infty}$-algebras. If the dual $\Hom\left( \delta, A\right) \: : \: \Hom^{\bullet}\left(T^{c}(V'[1]), A \right)\to \Hom^{\bullet}\left(T^{c}(V'[1]), A \right)$ preserves the filtration $\mathcal{I}$ then $$\left( \operatorname{Conv}_{P_{\infty}}\left(\left(V', m_{\bullet}^{V'}\right) ,\left(A, m_{\bullet}^{A} \right) \right), l_{\bullet}\right) $$ is a complete filtered $L_{\infty}$-algebra with respect to $\mathcal{I}^{\bullet}$ and the maps in \eqref{notcompl} are strict morphisms between complete filtered $L_{\infty}$-algebras.
For $V'$ of finite type, we denote the complete ideal generated by $\delta^{*}\: : \:\left(  \widehat{T}\left( \left(V'[1]\right)^{*} \right)\right)^{1} \to\widehat{T}\left( \left(V'[1]^{0} \right)^{*} \right)$ by $\overline{\mathcal{R}}_{0}\subset \widehat{T}\left( \left(V'[1]^{0} \right)^{*} \right)$ and if $m_{\bullet}^{V'}$ is $C_{\infty}$, we denote by ${\mathcal{R}}_{0}\subset \widehat{\mathbb{L}}\left( \left(V'[1]^{0} \right)^{*} \right)$ the complete Lie ideal generated by $\delta^{*}$. Let $p,r$ and $\pi$ be as above. The next corollary is a standard exercise about filtrations.
\begin{cor}\label{maurercartandego0finitetype}Let $\left(A, m_{\bullet}^{A} \right) $ be a bounded below  $A_{\infty}$-algebra and let $V$ be of finite type. The isomorphism $\Psi$ induces the following isomorphism between complete graded vector spaces.
\begin{enumerate} 
	\item If $\left(V', m_{\bullet}^{V'}\right) $ is a $A_{\infty}$-algebra$$
	\operatorname{Conv}_{A_{\infty}}\left(\left(V', m_{\bullet}^{V'}\right) ,\left(A, m_{\bullet}^{A} \right) \right)\cong  A\widehat{\otimes }\left( \widehat{T}(\left( V[1]\right)^{*}) \right).$$
\item  If $\left(V', m_{\bullet}^{V'}\right) ,\left(A, m_{\bullet}^{A} \right) $ are $C_{\infty}$ we have $
		\operatorname{Conv}_{C_{\infty}}\left(\left(V', m_{\bullet}^{V'}\right) ,\left(A, m_{\bullet}^{A} \right) \right)\cong  A\widehat{\otimes }\left( \widehat{\mathbb{L}}(\left( V'[1]\right)^{*}) \right)$.	

	\item Assume that $\left( A, m_{\bullet}^{A}\right) $ is unital and $\left( V, m_{\bullet}^{V}\right) $ is $1-A_{\infty}$. Then 
	\[
	\operatorname{Conv}_{1-A_{\infty}}\left(\left(V', m_{\bullet}^{V'}\right) ,\left(A, m_{\bullet}^{A} \right) \right)\cong  A\widehat{\otimes }\left( \widehat{T}(\left( V[1]^{0}\right)^{*})/\overline{\mathcal{R}}_{0}\right) .
	\]
	\item Assume that $\left( A, m_{\bullet}^{A}\right) $ is unital and $C_{\infty}$ and $\left( V, m_{\bullet}^{V}\right) $ is $1-C_{\infty}$. Then
	\[
	\operatorname{Conv}_{1-C_{\infty}}\left(\left(V, m_{\bullet}^{V}\right) ,\left(A, m_{\bullet}^{A} \right) \right) \cong A\widehat{\otimes }\left( \widehat{\mathbb{L}}(\left( V[1]^{0}\right)^{*})/{\mathcal{R}}_{0}\right).
	\]	
 Moreover if $\delta^{*}$ preserves $\mathcal{I}^{\bullet}$, it is a morphism of complete $L_{\infty}$-algebras.
\item  Assume that $\left(V', m_{\bullet}^{V'}\right)$ is a positively graded $P_{\infty}$-algebras such that $\mathcal{F}\left(V', m_{\bullet}^{V'}\right)=\left(V, m_{\bullet}^{V}\right)$. Let $C\in \operatorname{Conv}_{P_{\infty}}\left(\left(V', m_{\bullet}^{V'}\right) ,\left(A, m_{\bullet}^{A} \right) \right)$ be a Maurer-Cartan element. Then $\pi(C)$ is a Maurer-Cartan element in the $L_{\infty}$-algebra \\ $\operatorname{Conv}_{1-P_{\infty}}\left(\left(V, m_{\bullet}^{V}\right) ,\left(A, m_{\bullet}^{A} \right) \right)$ 
\item  Assume that $\left(V', m_{\bullet}^{V'}\right)$ is a positively graded $P_{\infty}$-algebras such that $\mathcal{F}\left(V', m_{\bullet}^{V'}\right)=\left(V, m_{\bullet}^{V}\right)$. Let $C^{1},C^{2}\in \operatorname{Conv}_{P_{\infty}}\left(\left(V', m_{\bullet}^{V'}\right) ,\left(A, m_{\bullet}^{A} \right) \right)$ be two $1$-Maurer-Cartan elements where $V'$ is positively graded. Let $f^{1}_{\bullet}, f^{2}_{\bullet}\: : \: \left(V', m_{\bullet}^{V'}\right)\to\left(A, m_{\bullet}^{A} \right) $ be the two corresponding $P_{\infty}$-maps and assume that they are $1$-homotopic. Then $\pi(C^{1}),\pi(C^{2}) $ are homotopic Maurer-Cartan elements in the $L_{\infty}$-algebra $\operatorname{Conv}_{1-P_{\infty}}\left(\left(V, m_{\bullet}^{V}\right) ,\left(A, m_{\bullet}^{A} \right) \right)$.
\end{enumerate}
\end{cor}

Let $\left(V, m_{\bullet}^{V}\right) ,\left(A, m_{\bullet}^{A} \right) $ be $P_{\infty}$-algebras, both assumed to be bounded below and that $V$ is of finite type. Then $m_{\bullet}^{V}$ corresponds to a codifferential $\delta$ on ${T}^{c}\left(V[1]\right)$.
\begin{defi}\label{defhomologicalpair}
Let $\left(V, m_{\bullet}^{V}\right) ,\left(A, m_{\bullet}^{A} \right) $ be $P_{\infty}$-algebras. Let $C$ be a Maurer-Cartan element in $\operatorname{Conv}_{P_{\infty}}\left(\left(V, m_{\bullet}^{V}\right) ,\left(A, m_{\bullet}^{A} \right) \right)$. We call $(C, \delta^{*})$ a \emph{homological pair}.
\end{defi}
 Let $\left(V', m_{\bullet}^{V'}\right)$ be a positively graded $P_{\infty}$-algebra of finite type and let  $\left(A, m_{\bullet}^{A} \right) $ be a unital $P_{\infty}$-algebra and let  $\alpha^{0}, \alpha^{1}\in \operatorname{Conv}_{P_{\infty}}\left(\left(V', m_{\bullet}^{V'}\right) ,\left(A, m_{\bullet}^{A} \right) \right)$ be Maurer-Cartan  elements. By Proposition \ref{Cinftydictionary}, there are two  $P_{\infty}$-morphisms $f^{0}_{\bullet}, f^{1}_{\bullet}\: : \: \left(V', m_{\bullet}^{V'}\right)\to\left(A, m_{\bullet}^{A} \right) $ associated to them.
\begin{prop}\label{propequivhomot} 
Let $\left(V', m_{\bullet}^{V'}\right)$ be a positively graded and of finite type $P_{\infty}$-algebra and let  $\left(A, m_{\bullet}^{A} \right) $ be a unital $P_{\infty}$-algebra.
	\begin{enumerate}
		\item Let  $\alpha^{0}, \alpha^{1}\in 	\operatorname{Conv}_{P_{\infty}}\left(\left(V', m_{\bullet}^{V'}\right) ,\left(A, m_{\bullet}^{A} \right) \right)$ be Maurer-Cartan  elements.\\ Let $f^{0}_{\bullet}, f^{1}_{\bullet}\: : \: \left(V, m_{\bullet}^{V}\right)\to\left(A, m_{\bullet}^{A} \right) $ be the two corresponding $P_{\infty}$-map. Then if they are homotopic, so are $\alpha^{0}, \alpha^{1}$. Moreover the pullback and pushforward along homotopic maps gives homotopic Maurer-Cartan elements.
	\item Let $\left(V,m_{\bullet}^{V} \right) $ be a $1-P_{\infty}$-algebra and let $\alpha^{0}, \alpha^{1}\in 	\operatorname{Conv}_{1-P_{\infty}}\left(\left(V, m_{\bullet}^{V}\right) ,\left(A, m_{\bullet}^{A} \right) \right)$ be Maurer-Cartan  elements. Let $f^{0}_{\bullet}, f^{1}_{\bullet}\: : \: \left(V, m_{\bullet}^{V}\right)\to\left(A, m_{\bullet}^{A} \right) $ be the two corresponding $P_{\infty}$-map viewed as $1$-morphism. Then if they are $1$-homotopic, it follows that $\alpha^{0}$ and $\alpha^{1}$ hare homotopic.
	 Moreover the pullback along $1$-homotopic maps and the pushforward along homotopic maps gives homotopic Maurer-Cartan elements.
	\end{enumerate}
	\end{prop}
		\begin{proof}
			We prove the first assertion. Let $H_{\bullet}$ be a homotopy between the two maps. We apply \eqref{Cinftydictionary} to $H_{\bullet}$. This gives a Maurer-Cartan element $$\alpha(t)\in \operatorname{Conv}_{A_{\infty}}\left(\left(V, m_{\bullet}^{V}\right) ,\left(\Omega(1)\otimes A, m_{\bullet}^{\Omega(1)\otimes A} \right)\right)\cong \Omega(1)\widehat{\otimes}\operatorname{Conv}_{A_{\infty}}\left(\left(V, m_{\bullet}^{V}\right) ,\left(A, m_{\bullet}^{A} \right) \right),$$
			where the last isomorphism is given by $\Psi$. On the left hand side we have the desired homotopy between $\alpha^{0}$ and $\alpha^{1}$ the sense of definition \eqref{Linftydef}. The other assertion follows similarly by using the map $\Psi$ as well and point 2 of Lemma \ref{tensorinfinity}.
		\end{proof}
		\subsection{ $1$-models and $1$-minimal models}\label{minimal model}
		\begin{defi}
 Let $\left(A,m^{A}_{\bullet}\right) $ be a non-negatively graded (unital) $P_{\infty}$-algebra. A $P_{\infty}$-sub algebra $\left(B,m^{B}_{\bullet}\right) $ is a $P_{\infty}$-algebra such that the inclusion is a strict morphism $i\: : \: B\hookrightarrow A$ of $P_{\infty}$-algebras. Hence $m_{\bullet}^{A}=m_{\bullet}^{B}$. Let $1\leq j\leq \infty$. A (unital) $P_{\infty}$-sub algebra $B$ is a $j$-model for $\left(A,m^{A}_{\bullet}\right) $ if 
\begin{enumerate}
		\item $i$ induces an isomorphism up to the $j$-th cohomology group and is injective on the $j+1$ cohomology group.
		\item the inclusion  $i^{l}\: : \: B^{l}\hookrightarrow A^{l}$ preserves non-exact elements for $0\leq l\leq j+1$.
		\end{enumerate}
		If $j=\infty$, we call $B$ a\emph{ model} for $\left(A,m^{A}_{\bullet}\right) $. A (unital) $P_{\infty}$-algebra $(W,m_{\bullet}^{W})$ is said to be minimal if $m_{1}^{W}=0$. A \emph{$j$-minimal model} of $\left(A,m^{A}_{\bullet}\right) $ consists in an (unital) morphism $g_{\bullet}\: : \: (W,m_{\bullet}^{W})\to \left(A,m^{A}_{\bullet}\right) $ such that $(W,m_{\bullet}^{W})$ is minimal and $g_{\bullet}$ induces an isomorphism up to the $j$-th cohomology group and is injective on the $j+1$ cohomology group. If $j=\infty$ we call $g_{\bullet}$ a \emph{minimal model}. 
\end{defi}	
	\begin{defi}
					Let $\left(A,m^{A}_{\bullet}\right) $ be a non-negatively graded (unital) $1-P_{\infty}$-algebra. A $1-P_{\infty}$-sub algebra $\left(B,m^{B}_{\bullet}\right) $ is a $1-P_{\infty}$-algebra such that the inclusion is a strict morphism $i\: : \: B\hookrightarrow A$ of $P_{\infty}$-algebras. A (unital) $P_{\infty}$-sub algebras $B$ is a $1$-model for $\left(A,m^{A}_{\bullet}\right) $ if $i$ is a quasi-isomorphism and the inclusion  $i^{l}\: : \: B^{l}\hookrightarrow A^{l}$ preserves non-exact elements for $0\leq l\leq 2$. A (unital) non-negatively graded $1-P_{\infty}$-algebra $(V,m_{\bullet}^{V})$ is said to be minimal if $m_{1}^{V}=0$ and if $V^{2}$ is the vector space generated by $m_{n}(a_{1}, \dots, a_{n})$ where $a_{i}\in V^{1}$. Let $B$ be a positively graded $1-P_{\infty}$-algebra. A \emph{$1$-minimal model with coefficients in $B$} consists in an (unital) $1$-quasi-isomorphism $g_{\bullet}\: : \: (V,m_{\bullet}^{V})\to \mathcal{F}\left(B,m^{A}_{\bullet}\right) $ such that $(V,m_{\bullet}^{V})$ is minimal. 
	\end{defi}	
Each $1$-minimal model for $P_{\infty}$-algebras $g_{\bullet}\: : \: \left(V', m_{\bullet}^{V'} \right) \to \left(A, m_{\bullet} \right) $ gives a $1$-minimal model $\mathcal{F}\left( {g'}_{\bullet}\right) $ for $\mathcal{F}\left(A, m_{\bullet} \right) $ via 
\[
\begin{tikzcd}
{g'}_{\bullet}\: :\:\left(V, m_{\bullet}^{V'} \right) \arrow[r, hook] & \left(V', m_{\bullet}^{V'} \right)\arrow{r}{{g}_{\bullet}} & \left(A, m_{\bullet} \right)
\end{tikzcd}
\]
where $\left(V, m_{\bullet}^{V'} \right)\subset  \left(V', m_{\bullet}^{V'} \right) $ is the $P_{\infty}$-subalgebra generated by $V^{0}$ and $V^{1}$.
	\begin{rmk}
The minimal models and $1$-minimal models can be constructed explicit via the homotopy transfer theorem (see \cite{Kadesh}, \cite{kontsoibel},\cite{Markl} and \cite{Prelie}).
	\end{rmk}
	\begin{defi}
A non-negatively graded $P_{\infty}$-algebra (resp. $1-P_{\infty}$-algebra) has connected cohomology if its cohomology in degree zero has dimension $1$. 
\end{defi}	
Let $W$ be a non-negatively graded vector space. Let $m_{\bullet}^{W}$ be a $P_{\infty}$-structure on $W^{\bullet}$. For $n>1$, let $m_{n}^{W_{+}}$ be the maps given by the composition
\[
\begin{tikzcd}
W_{+}^{\otimes n}\arrow[hook]{r}&  W^{\otimes n}\arrow{r}{m_{n}^{W}} & W_{+}.
\end{tikzcd}
\]
They are well-defined since $m_{n}^{W}$ are maps of degree $2-n$ and $(W_{+}^{\bullet}, m_{\bullet}^{W_{+}})$ is a $P_{\infty}$-algebra. Assume that $m_{1}^{W}=0$. Let $g_{\bullet}\: : \: (W, m^{W})\to (A, m^{A})$ a morphism of $P_{\infty}$-algebras. We denote by $g^{+}_{\bullet}$ the restriction of $g_{\bullet}$ to $\left(W^{\bullet}_{+}, m_{\bullet}^{W_{+}}\right)$. Then $g_{\bullet}^{+}\: : \: (W_{+}, m^{W_{+}})\to (A, m^{A})$ is a well-defined morphism of $P_{\infty}$-algebras. Assume that $(W, m^{W})$ is connected and that $ (A, m^{A})$ is unital. The map $g_{\bullet}\mapsto g_{\bullet}^{+}$ is a bijection. 
Let $\left(A,m_{\bullet}\right) $ be a non-negatively graded unital $P_{\infty}$-algebra with connected cohomology.  Let $g_{\bullet}^{1}\: : \:(W_{1}, m_{\bullet}^{1})\to \left(A,m_{\bullet}\right) $ be a $P_{\infty}$-algebra minimal model. By Proposition \ref{existence of a homotopical inverse} we get an inverse up to homotopy
\[
\begin{tikzcd}
	g^{1}_{\bullet}\: : \: ({W_{1}}, m_{\bullet}^{1})\arrow[r, shift right, ""]&(A,m_{\bullet})\: : \: f^{1}_{\bullet}	\arrow[l, shift right, ""]. 
				\end{tikzcd}
				\]	
 	\begin{prop}\label{homotopyuniqueness}Let $g_{\bullet}^{1}$ be as above.\begin{enumerate}
	\item Let $g_{\bullet}^{2}\: : \:(W_{2}, m_{\bullet}^{2})\to \left(A,m_{\bullet}\right) $ be $P_{\infty}$-algebra-morphism between non-negatively graded unital objects. Consider the diagram
		\[
		\begin{tikzcd}
		({W_{1}}, m_{\bullet}^{1})\arrow[rrr, shift right, " g^{1}_{\bullet}"']&&&(A,m_{\bullet})
		\arrow[lll, shift right, " f^{1}_{\bullet} "'] \\ \\ \\
		({W_{2}}, m_{\bullet}^{2})\arrow[uuu, shift right, "k_{\bullet}"]\arrow[uuurrr, shift right, "g^{2}_{\bullet}"']
		\end{tikzcd}
		\]
		where $k_{\bullet}:=f^{1}_{\bullet}g^{2}_{\bullet}$. For $j=1,2$ let $\alpha_{j}\in \operatorname{Conv}_{r}\left(\left( \left(W_{j}\right)  _{+}, m_{\bullet}^{j}\right) ,\left(A, m_{\bullet}^{A} \right) \right)$ be the Maurer-Cartan elements corresponding to $g_{\bullet}^{j}$ via Proposition \ref{Cinftydictionary} in the reduced convolution algebra. Then 
		$\alpha_{2}$ is homotopic to $k^{*}\left( \alpha_{1}\right)$. Moreover, $k_{\bullet}$ is an isomorphism if $g_{\bullet}^{2}$ is a $1$-minimal model.
		\item Let $g_{\bullet}^{2}\: : \:(W_{2}, m_{\bullet}^{2})\to \mathcal{F}\left(A,m_{\bullet}\right) $ be $1-P_{\infty}$-algebra-morphism between non-negatively graded unital objects. Consider the diagram
		\[
		\begin{tikzcd}
		\mathcal{F}({W_{1}}, m_{\bullet}^{1})\arrow[rrr, shift right, "\mathcal{F}\left(  g^{1}_{\bullet}\right) "']&&&\mathcal{F}(A,m_{\bullet})
		\arrow[lll, shift right, " \mathcal{F}\left( f^{1}_{\bullet} \right) "'] \\ \\ \\
		({W_{2}}, m_{\bullet}^{2})\arrow[uuu, shift right, "k_{\bullet}"]\arrow[uuurrr, shift right, "g^{2}_{\bullet}"']
		\end{tikzcd}
		\]
		where $k_{\bullet}:=\mathcal{F}\left( f^{1}_{\bullet} \right)g^{2}_{\bullet}$. For $j=1,2$, let $\overline{\alpha}_{j}$ be the Maurer-Cartan elements corresponding to $g_{\bullet}^{j}$ via Proposition \ref{Cinftydictionary1} in the degree zero reduced convolution algebra. Then 
				$\overline{\alpha}_{2}$ is homotopic to $k^{*}\left( \overline{\alpha}_{1}\right)$. Moreover, $k_{\bullet}$ is a $1$-isomorphism if $g_{\bullet}^{2}$ is a $1$-minimal model.
				\item Consider the situation at point 2  and assume that $g_{\bullet}^{2}$ is a $1$-minimal model. The $P_{\infty}$-map 
				\[
				\mathcal{E}(k_{\bullet})\: : \: 	\mathcal{E}({W_{2}}, m_{\bullet}^{2})\to	\mathcal{E}\mathcal{F}({W_{1}}, m_{\bullet}^{1})
				\]
				has a left-inverse ${k'}_{\bullet}$.
	\end{enumerate}
	In particular, $1$-minimal model for $1-P_{\infty}$-algebras are unique modulo $1$-isomorphisms.
\end{prop}
		\begin{proof}
		By Proposition \ref{existence of a homotopical inverse}, there exists a homotopy $H_{\bullet}$ between $\left( g^{1}\right)_{\bullet} \left( f^{1}\right)_{\bullet} $ and $Id_{A}$. Then $$g^{1}_{\bullet} k_{\bullet}=g^{1}_{\bullet} f^{1}_{\bullet}g^{2}_{\bullet}$$
		is homotopic to $g^{2}_{\bullet}$ via $ H_{\bullet}g^{2}_{\bullet}$. We prove the second part. If $g_{\bullet}^{2}$ is a quasi-isomorphism then $k_{1}=f^{1}g^{2}\: : \: W_{2}\to W_{1}$ is a quasi-isomorphism. Since $m_{1}^{j}=0$ for $j=1,2$, it is an isomorphism. Part 2. follows analogously. We prove part 3. By construction we have $\left( \mathcal{E}(k)\right)_{1}\: : \: W^{i}\to W^{i}$ is an isomorphism for $i\neq 2$ and  is injective for $i=2$. Then, there exists a left-inverse ${k'}\: : \: W\to W$. By using Theorem 10.4.2 in \cite{lodayVallette}, there exists an explicit $P_{\infty}$-map ${k'}_{\bullet}$ which extends ${k'}$. 
		\end{proof}
Consider the situation at point 2  in Proposition \ref{homotopyuniqueness}. We assume that $W_{i}$ are of finite type for $i=1,2$ and that $(W_{2},m_{\bullet}^{2})$ is minimal. By Proposition \ref{Cinftydictionary1}, $k_{\bullet}$ corresponds to a tensor coalgebra map $K=K^{0}\oplus K^{1}$ where
		$$K\: : \: \left( T^{c}\left(\left( W_{2}\right)_{+}[1] \right)\right)^{0}\oplus \left( \left( T^{c}\left(\left( W_{2}\right)_{+}[1] \right)\right)^{1}\right) \to \left( T^{c}\left(\left( W_{1}\right)_{+}[1] \right)\right)^{0}\oplus  \left( T^{c}\left(\left( W_{1}\right)_{+}[1] \right)\right)^{1}$$ such that 
 $\left(K\otimes K \right) \Delta=\Delta K $ and $K^0\delta^{1}=\delta^{1} K^1$. Its dual on degree-zero elements gives a map
 $$
 \left(  K^{0}\right)^{*}\: : \:\widehat{T}  \left(\left(  W_{1}^{1}\right) _{+}[1]\right) ^{*}\to \widehat{T} \left(\left(  W_{2}^{1}\right) _{+}[1]\right)^{*}
 $$
which is a morphism of complete tensor algebras. These tensor algebras are in fact Hopf algebras, where the comultiplication is given by the shuffles coproduct. 
We now assume that $k_{\bullet}$ is a $1-C_{\infty}$-morphism between $1-C_{\infty}$-algebras. It follows that $K$ vanishes on non-trivial shuffles, $K^{*}$ restricts to a morphism between complete free lie algebras
  \begin{equation}
 \left(  K^{0}\right)^{*}\: : \:\widehat{\mathbb{L}}  \left(\left(  W_{1}^{1}\right) _{+}[1]\right) ^{*}\to \widehat{\mathbb{L}} \left(\left(  W_{2}^{1}\right) _{+}[1]\right)^{*},
  \end{equation}
  and $\left( \delta^{i}\right)^{*}\: : \: \left( \widehat{T}  \left(\left(  W_{1}\right) _{+}[1]\right) ^{*}\right)^{-1}\to \widehat{T} \left(\left(  W_{i}^{1}\right) _{+}[1]\right)^{*}$
  restricts to a well-defined Lie algebra map $\left( \delta^{i}\right)^{*}\: : \: \left( \widehat{\mathbb{L}}  \left(\left(  W_{1}\right) _{+}[1]\right) ^{*}\right)^{-1}\to \widehat{\mathbb{L}} \left(\left(  W_{i}^{1}\right) _{+}[1]\right)^{*}$ for $i=1,2$. We denote by ${\mathcal{R}}_{0}^{i}\subset \widehat{\mathbb{L}}  \left(\left( W_{i}^{1}\right) _{+}[1]\right) ^{*}$ the complete Lie ideal in the complete tensor algebra generated by the image of $\left( \delta^{i}\right)^{*}$ for $i=1,2$. Let $\bar{\mathcal{R}}_{0}^{i}\subset \widehat{T}  \left(\left( W_{i}^{1}\right) _{+}[1]\right) ^{*}$ be the ideals in the complete tensor algebra generated by ${\mathcal{R}}_{0}^{i}$ for $i=1,2$. Since $\delta^{i}$ vanishes on non-trivial shuffles we get that $\bar{\mathcal{R}}_{0}^{i}$ is a Hopf ideal for $i=1,2$ as well. Hence $
   \left(  K^{0}\right)^{*}\: : \:\widehat{T}  \left(\left(  W_{1}^{1}\right) _{+}[1]\right) ^{*}/\bar{\mathcal{R}}_{0}^{1}\to \widehat{T} \left(\left(  W_{2}^{1}\right) _{+}[1]\right)^{*}/\bar{\mathcal{R}}_{0}^{2}
  $ is a morphism of complete Hopf algebras and it restricts to a Lie algebra morphism on the primitive elements. By abuse of notation we denote its restriction by
  \begin{equation}\label{mapKdual}
  K^{*}:=\left(  K^{0}\right) ^{*}\: : \:\widehat{\mathbb{L}}  \left(\left(  W_{1}^{1}\right) _{+}[1]\right) ^{*}/\mathcal{R}_{0}^{1}\to \widehat{\mathbb{L}} \left(\left(  W_{2}^{1}\right) _{+}[1]\right)^{*}/\mathcal{R}_{0}^{2}.
  \end{equation}	
 In particular, the pullback along $k_{\bullet}$
  \[
  k^{*}\: :\:  \operatorname{Conv}_{1-P_{\infty}}\left(\left(V, m_{\bullet}^{V}\right) ,\left(A, m_{\bullet}^{A} \right) \right) \to\operatorname{Conv}_{1-P_{\infty}}\left(\left(W, m_{\bullet}^{W}\right) ,\left(A, m_{\bullet}^{A} \right) \right)
  \]
  corresponds under the identification of Corollary \ref{maurercartandego0finitetype} to 
  \[
  \mathrm{Id}\widehat{\otimes }K^{*}\: : \: A\widehat{\otimes }\left( \widehat{\mathbb{L}}  \left(\left(  W_{1}^{1}\right) _{+}[1]\right) ^{*}/\mathcal{R}_{0}^{1} \right)\to A\widehat{\otimes }\left( \widehat{\mathbb{L}}  \left(\left(  W_{2}^{1}\right) _{+}[1]\right) ^{*}/\mathcal{R}_{0}^{2} \right)
  \]
\begin{prop}\label{beh}
	Consider the situation at point 2 in Proposition \ref{homotopyuniqueness} for $P_{\infty}=C_{\infty}$. Assume that $g_{\bullet}^{2}$ is a $1$-minimal model for $1-C_{\infty}$-algebras and that $W_{i}$ are of finite type for $i=1,2$. Then $
 \left(  K^{0}\right)^{*}\: : \:\widehat{T}  \left(\left(  W_{1}^{1}\right) _{+}[1]\right) ^{*}/\bar{\mathcal{R}}_{0}^{1}\to \widehat{T} \left(\left(  W_{2}^{1}\right) _{+}[1]\right)^{*}/\bar{\mathcal{R}}_{0}^{2}
 $ 
is an isomorphism of complete Hopf algebras. In particular its restriction on primitive elements $K^{*}$ is an isomorphism of complete Lie algebras. Moreover $K^{*}=\sum_{i=1}^{\infty}K_{i}^{*}$ where $K_{1}^{*}$ is induced by the dual of $k_{1}=f^{1}_{1}g^{2}_{2}\: : \: W_{2}\to W_{1}$.
 \end{prop}
\begin{proof}
By Proposition \ref{homotopyuniqueness}, $K$ has a left inverse. Hence its dual has a right inverse which induces an isomorphism on the quotient. 
\end{proof}

\section{ A $C_{\infty}$-structure on the total complex}\label{sectGetztler}
We give a very short introduction about the Dupont contraction and the results of \cite{Getz}. We introduce a $C_{\infty}$-structure that corresponds to the natural algebraic structure on the differential forms of a smooth complex simplicial manifold (see Theorem \eqref{functor}). This $C_{\infty}$-structure is in general hard to calculate. In the last we use a result of \cite{Getz} to present an almost complete formula on degree $1$-elements (see Theorems \ref{productthm>2} and \ref{productthm=2}). In this section we work on a field $\Bbbk$ of charactersitic zero.
\subsection{Cosimplicial commutative algebras}
We denote by $sSet$ the category of simplicial sets and by $\Delta\: : \:\boldsymbol{\Delta}\to sSet $ the Yoneda embedding.\\
 For each $[n]\in \boldsymbol{\Delta}$ we define the $n$-gemetric simplex
\[
\Delta_{geo}[n]:=\left\lbrace(t_{0},t_{1}, \dots, t_{n})\in\R^{n}\: | \: t_{0}+t_{1}+\dots +t_{n}=1\right\rbrace. 
\]
For each $[n]\in \boldsymbol{\Delta}, i=0,\dots, n+1$ we define the smooth maps $d^{i}\: : \: \Delta_{geo}[n]\to \Delta_{geo}[n+1]$
\[
d^{i}\left( t_{0},t_{1}, \dots, t_{n}\right)=\left( t_{0},t_{1}, \dots,t_{i-1},0,t_{i}, \dots t_{n}\right)
\]
and $s^{i}\: : \: \Delta_{geo}[n+1]\to\Delta_{geo}[n], i=0,\dots, n$ via
\[
s^{i}\left( t_{0},t_{1}, \dots, t_{n+1}\right)=\left( t_{0},t_{1}, \dots,t_{i-1},t_{i}+t_{i+1}, \dots, t_{n+1}\right).
\]
In particular, $\Delta[\bullet]_{geo}$ is a cosimplicial topological space. For each $[n]$ let $\Omega^{\bullet}(n)$ be the symmetric graded algebra (over $\Bbbk$) generated in degree $0$ by the variables $t_{0}, \dots , t_{n}$ and in degree $1$ by $dt_{0}, \dots , dt_{n}$ such that
\[
t_{0}+ \dots + t_{n}=1, \quad dt_{0}+\dots + dt_{n}=0.
\]
We equip $\Omega(n)$ with a differential $d\: : \: \Omega^{\bullet}(n)\to \Omega^{\bullet+1}(n)$ via $d\left( t_{i}\right) :=dt_{i}$. $\Omega(n)$ is the differential graded algebra of polynomial differential forms on $\Delta[n]_{geo}$. It follows that $\Omega(\bullet)$ is a simplicial commutative differential graded algebra, where the face maps $d_{i}$ are obtained via the pullback along $d^{i}$, and the codegenerancy maps are obtained via the pullback along $s^{i}$.\\
For a set $X$ we denote by $\Bbbk\left\langle X\right\rangle $ the module generated by $X$ and by $X^{\Bbbk}$ the module $\Hom_{Set}\left(X, \Bbbk \right)$. Thus for a simplicial set $X_{\bullet}$ we denote by $\Bbbk\left\langle X_{\bullet}\right\rangle $ the simplicial module $\Bbbk\left\langle X\right\rangle _{n}:=\Bbbk\left\langle X_{n}\right\rangle $ and by $X_{\bullet}^{\Bbbk}$ the cosimplicial module $\left( X^{\Bbbk}\right) _{n}:=X_{n}^{\Bbbk}$. Both of these constructions are functorials and $\left( -\right)^{\Bbbk}\: : \: sSet\to cMod$ is contravariant.\\
 Let $\boldsymbol{\Delta}$ be the simplex category and let $\Delta\: : \:\boldsymbol{\Delta}\to sSet $ be the Yoneda embedding. Then $\Delta[\bullet]$ is a cosimplicial object in the category of simplicial sets and $C_{\bullet}:=\left( \Delta[\bullet]\right)^{\Bbbk}$ is a simplicial cosimplicial module. We get that $NC_{\bullet}$ is a simplicial differential graded module. Explicitly, for a fixed $n$ we have
 \[
 \left( NC_{n}\right)^{p}:=\begin{cases}
 \Bbbk\left\langle \Hom_{Set}\left(\Delta[n]_{p}^{+}, \Bbbk \right) \right\rangle , \text{ if } & p\leq n,\\
 0, \text{ if } & p>n
 \end{cases}\]
 where $\Delta[n]_{p}^{+}$ is the set of inclusions $[p]\hookrightarrow [n]$.  A cosimplicial differential graded module is a cosimplicial object in the category $dgMod$ of differential graded modules. Explicitly, we denote this objects by $A^{\bullet, \bullet}$ where the first slot denote the cosimplicial degree and the second slot denotes the differential degree. It defines a functor $A^{\bullet, \bullet}\: : \: \boldsymbol{\Delta}\to dgMod$ and we get a bifunctor $NC_{\bullet}\otimes A^{\bullet, \bullet}\: : \: \boldsymbol{\Delta}^{op}\times \boldsymbol{\Delta}\to dgMod$. We consider the coend
 \[
 \int^{[n]\in\boldsymbol{\Delta}}NC_{n}\otimes A^{n, \bullet}\in dgMod.
 \]
 An element $v$ of degree $k$ in $\int^{[n]\in\boldsymbol{\Delta}}NC_{n}\otimes A^{n, \bullet}$ is a sequence $v:=(v_{n})_{n\in \N}$ where $v_{n}\in \left( NC_{n}\otimes A^{n, \bullet}\right)^{k}$ such that for any map $\theta\: : \: [n]\to [m]$ in $\boldsymbol{\Delta}$ we have
 \[
 \left( 1\otimes \theta^{*}\right)w_{n}=\left( \theta_{*}\otimes 1 \right)w_{m},
 \]
 where $\theta^{*}:=A^{\bullet, \bullet}(\theta)$, and $\theta_{*}:=NC_{\bullet}(\theta)$. Since
 \[
 \left( NC_{n}\otimes A^{n, \bullet}\right)^{k}=\oplus_{p+q=k}NC_{n}^{p}\otimes A^{n, q}
 \]
 we say that $v$ has bidegree $(p,q)$ if $v$ has degree $p+q$ and each $v_{n}\in NC_{n}^{p}\otimes B^{n, q}$ for each $n$. Let $\left( V,d_{V}\right) $ and $\left( W,d_{w}\right)$ be two cochain complexes, $\left( V\otimes W\right)^{\bullet}$ is again a cochain complex where the differential is
 \[
 d_{V\otimes W}\left( v\otimes w\right):=d_{V}(v)\otimes w+(-1)^{p}v\otimes d_{w}(w) 
 \]
 for $v\otimes w\in V^{p}\otimes W^{q}$. 
 The differential on $\int^{[n]\in\boldsymbol{\Delta}}NC_{n}\otimes A^{n, \bullet}$ is defined via
 \[
 \left( dv\right)_{n}:= dv_{n}
 \]
 where $d$ is the induced differential on $\left( NC_{n}\otimes A^{n, \bullet}\right)^{\bullet}$. Consider the differential graded module (called Thom-Whitney normalization, see \cite{Getz})
\[
\operatorname{Tot}_{TW}\left(A\right):=\int^{[n]\in \boldsymbol{\Delta}} \Omega(n)\otimes A^{n,\bullet  }\in dgMod.
\]
Explicitly, an element $v\in \operatorname{Tot}_{TW}\left(A\right)^{k}$ is a collection $v=\left( v_{n}\right)$ of $v_{n}\in  \left( \Omega(n)\otimes A^{n,\bullet }\right)^{k}$ such that for any map $\theta\: : \: [n]\to [m]$ in $\boldsymbol{\Delta}$ we have
\[
\left( 1\otimes \theta_{*}\right)v_{n}=\left( \theta^{*}\otimes 1 \right)v_{m},
\]
where $\theta_{*}:=A^{\bullet, \bullet}(\theta)$, and $\theta^{*}:=\Omega(n)(\theta)$. Since
\[
\left(\Omega(n)\otimes A^{n, \bullet}\right)^{k}=\oplus_{p+q=k}\Omega^{p}(n)\otimes A^{n, q},
\]
we say that $v$ has bidegree $(p,q)$ if $v$ has degree $p+q$ and each $v_{n}$ is contained in $\Omega^{p}[n]\otimes A^{n, q}$. We denote by $\operatorname{Tot}_{TW}\left(A\right)^{p,q}$ the set of elements of bidegree $(p,q)$. If $A^{\bullet, \bullet}$ is a cosimplicial unital differential graded commutative algebra, then  $\left(\operatorname{Tot}_{TW}\left(A\right), d_{\bullet,A},d_{\bullet,poly}\right)$ is a differential graded commutative algebra as well where the multiplication and the differential are
\[
(v\wedge w)_{n}:=(v)_{n}\wedge (w)_{n}, \quad (dv)_{n}:=d\left( v_{n}\right).
\] 
\subsection{The Dupont retraction}\label{CinfiDerham}
We give a short summary of the results of \cite{Getz}, where a $C_{\infty}$-structure is induced on $\int^{[n]\in\boldsymbol{\Delta}}NC_{n}\otimes A^{n,\bullet }$ (and hence on $\operatorname{Tot}_{N}\left(A\right)$) from $\operatorname{Tot}_{TW}(A)$ via the homotopy transfer theorem.

\begin{thm}[\cite{Getz2},\cite{Dupont2}]\label{getzdup}
Let $\Omega(\bullet)$, $NC_{\bullet}$ be the two simplicial differential graded modules defined above. We denote the differential of $\Omega(\bullet)$ by $d_{\bullet, poly}$. There is a diagram between simplicial graded modules
\begin{equation}\label{simplicialdiag}
\begin{tikzcd}
E_{\bullet}\: : \: NC_{\bullet}\arrow[r, shift right, ""]&\Omega(\bullet)\: : \: \int_{\bullet}\arrow[l, shift right, ""] 
\end{tikzcd}
\end{equation}
and a simplicial homotopy operator $s_{\bullet}\: : \: \Omega(\bullet)\to \Omega^{\bullet-1}(\bullet)$ between $\left( E_{\bullet}\int_{\bullet}\right) $ and the identity, i.e., $s_{\bullet}$ is a map between simplicial differential graded modules such that for each $n\geq 0$
\[
d_{n, poly}s_{n}+s_{n}d_{n, poly}=E_{n}\int_{n}-Id.
\]
In particular, for any $n$ we have $E_{n}$ is an injective quasi-isomorphism, $\int_{n}E_{n}=Id$ and \begin{align*}
\int_{n}s_{n}=s_{n}E_{n}=s_{n}^{2}=0.
\end{align*}
\end{thm}
\begin{proof}
The first statement is originally contained in \cite{Dupont2}. The second part of the theorem is proved in \cite{Getz2}.
\end{proof}
See Section \ref{sec2dimsimplex} for more details about the above maps. The homotopy transfer theorem (see \cite{Markl}) gives a $C_{\infty}$-algebra structure $m_{\bullet}^{[n]}$ on $NC_{n}$ induced by the above diagram. Let $A^{\bullet,\bullet}$ be a cosimplicial commutative algebra. The differential graded algebra $\Omega(n)\otimes A^{n, \bullet}$ is commutative as well. For any $n,m\geq 0$, we denote the  $C_{\infty}$-structure induced along the diagram
\begin{equation}\label{simplicialdiagtens}
\begin{tikzcd}
E_{n}\otimes Id \: : \:  NC_{n}^{\bullet}\otimes A^{ n,\bullet}\arrow[r, shift right, ""]& \Omega(n)\otimes A^{n, \bullet}\: : \: \int_{n}\otimes Id\arrow[l, shift right, ""] 
\end{tikzcd}
\end{equation}
by $m_{\bullet}^{n,m}$. This structure depends only on $m_{\bullet}^{[n]}$. Let $w_{1}, \dots, w_{l}\in NC_{n}\otimes A^{n,\bullet }$ be such that $w_{i}=f_{i}\otimes a_{i}$ for $i=1, \dots, l$. Then
\begin{equation}\label{reduction}
 m_{l}^{n,m}\left(w_{1}, \dots, w_{l} \right)=(-1)^{\sum_{i<j} |f_{i}||a_{j}|}m_{l}^{[n]}\left(f_{1}, \dots, f_{l} \right)\otimes \left( a_{1}\wedge \cdots \wedge a_{l}\right).
\end{equation}
The above structure defines a well-defined $C_{\infty}$-structure ${m}_{\bullet}$ on $\int^{[n]\in\boldsymbol{\Delta}}NC_{n}\otimes A^{n, \bullet}$. The maps ${m}_{\bullet}$ can be obtained in another way. We apply the coend functor on the simplicial diagram \eqref{simplicialdiagtens}. Since $s_{\bullet}$, $E_{\bullet}$, $\int_{\bullet}$ are all simplicial maps, they induce degree zero maps  $E$, $\int$ between the coends such that
	\begin{equation}\label{secondiag}
	\begin{tikzcd}
	E\: : \: \int^{[n]\in\boldsymbol{\Delta}}NC_{n}\otimes A^{n,\bullet }\arrow[r, shift right, ""]&\operatorname{Tot}_{TW}\left(A\right)\: : \: \int\arrow[l, shift right, ""] 
	\end{tikzcd}
\end{equation}
is a diagram that satisfies the same properties of the one in \eqref{simplicialdiag}. There is an unital $C_{\infty}$-structure induced on $\int^{[n]\in\boldsymbol{\Delta}}NC_{n}\otimes A^{n,\bullet }$ via the homotopy transfer theorem. By construction, this structure coincides with ${m}_{\bullet}$. Notice that the construction of $\int^{[n]\in\boldsymbol{\Delta}}NC_{n}\otimes A^{n, \bullet}$ gives a functor form the category of cosimplicial unital non-negatively graded commutative differential graded algebras ($\mathrm{cdgA}$ for short) toward the category of cochain complexes. We have a correspondence
		\begin{equation}
		\label{functor0}
		A^{\bullet,\bullet}\mapsto \left( \int^{[n]\in\boldsymbol{\Delta}}NC_{n}\otimes A^{n, \bullet},m_{\bullet}\right). 
		\end{equation}
		For a field $\Bbbk$ of characteristic zero we denote by $\left( C_{\infty}-\operatorname{Alg}\right)_{\Bbbk,str}$ the category of $C_{\infty}$-algebras on $\Bbbk$ and strict morphisms. 
		\begin{thm}\label{functorassoc0}
		The correspondence \eqref{functor0} gives a functor $\mathrm{cdgA}\to \left( C_{\infty}-\operatorname{Alg}\right)_{\Bbbk,str}$.
		\end{thm}
		\begin{proof}
		Let $f\: : \: A^{\bullet,\bullet}\to B^{\bullet,\bullet}$ be a morphism. The correspondence \eqref{functor0} is a functor toward the category of chain complexes. We denote its image by $f_{1}$. Since $f$ is a differential graded algebra map on each degree, the same argument of Lemma \ref{tensorinfinity} shows that $f_{1}$ induces a strict morphism of $C_{\infty}-$algebras.
		\end{proof}
		We give an explicit formula for $m_{\bullet}$. Fix a $n$ and a $p\leq n$. Notice that each inclusion $[p]\hookrightarrow [n]$ is equivalent to an ordered string $0\leq i_{0}<i_{i}<\dots <i_{p}\leq n$ contained in $\left\lbrace 0,1,\dots, n\right\rbrace $. For each string $0\leq i_{0}<i_{i}<\dots <i_{p}\leq n$, we denote the associated inclusion by $\sigma_{i_{0}, \dots i_{p}}\: : \:[p]\hookrightarrow [n]$, and we define the maps $\lambda_{i_{0}, \dots, i_{p}}\: : \: \Delta[n]_{p}^{+}\to \Bbbk$, via
		 \[
		 \lambda_{i_{0}, \dots, i_{p}}(\phi):=
		 \begin{cases}
		 1\text{ if }\sigma_{i_{0}, \dots i_{p}}=\phi,\\
		 0,\text{ otherwhise}.
		 \end{cases}
		 \]
		Let $v_{1}, \dots, v_{n}\in \int^{[n]\in\boldsymbol{\Delta}}NC_{n}\otimes A^{n, \bullet}$ be elements of bidegree $\left(p_{i}, q_{i} \right)$. Then $m_{n}\left(v_{1}, \dots, v_{n} \right) $ is an element of bidegree $\left(\sum_{i}pi+2-n , \sum q_{i}\right) $. Let $l:=\sum_{i}pi+2-n$, then Lemma \ref{isompsi} implies that $m_{n}\left(v_{1}, \dots, v_{n} \right) $ is completely determined by $m_{n}\left(\left( v_{1}\right)_{l} , \dots, \left( v_{n}\right)_{l}\right)_{l}$. We write
$\left( v_{i}\right)_{p_{i}}=\lambda_{0, \dots, p_{i}}\otimes a_{i}\in NC_{p_{i}}^{p_{i}}\otimes A^{p_{i}, q_{i}}$ for all the $i$. We denote by $I$ the subsets $\left\lbrace i_{0}, \dots ,i_{p}\right\rbrace \subseteq \left\lbrace 0,\dots , l\right\rbrace $, for $I=\left\lbrace i_{0}, \dots ,i_{p}\right\rbrace $ we define $|I|:=p$ and we write $\lambda_{I}$ instead of $\lambda_{0, \dots, p}$. Each $I$ corresponds to an inclusion in $\boldsymbol{\Delta}$; we denote by $\sigma_{I}\: : \: [p]\to [l]$ the map induced by $I$. We have
\begin{align}
\nonumber m_{n}\left(v_{1}, \dots, v_{n} \right)_{l} &=m_{n}^{l,l}\left(\left( v_{1}\right)_{l} , \dots, \left( v_{n}\right)_{l}\right)_{l}\\\nonumber
& = m_{n}^{l,l}\left(\sum_{|I_{{1}}|=p_{1}}\lambda_{I_{1}}\otimes \left( \sigma_{I_{1}}\right)_{*}a_{1}, \dots, \sum_{|I_{n}|=p_{n}} \lambda_{I_{n}}\otimes \left( \sigma_{I_{n}}\right)_{*}a_{n}\right)\\
& =\sum_{|I_{1}|=p_{1},\dots ,|I_{n}|=p_{n} } (-1)^{\sum_{i<j} |p_{i}||q_{j}|}m_{n}^{[l]}\left(\lambda_{I_{1}}, \dots, \lambda_{I_{n}} \right)\otimes \left( \left( \sigma_{I_{1}}\right)_{*}a_{1}\wedge \cdots \wedge \left( \sigma_{I_{n}}\right) _{*}a_{n}\right).\label{formulafinale}
\end{align}
In particular, the above formula implies that if $v_{1}, \dots, v_{n}$ are all of degree $1$, then $m_{n}\left(v_{1}, \dots, v_{n} \right)$ would only depend on
\begin{itemize}
\item the restriction of $m_{\bullet}^{[2]}$ on the elements of degree $1$, if all the $v_{i}$ are of bidegree $(1,0)$;
\item $m_{\bullet}^{[0]}$, if all the $v_{i}$ are of bidegree $(0,1)$;
\item $m_{\bullet}^{[1]}$ in the other cases.
\end{itemize}
\begin{lem}\label{inclusiontrivial}
	Consider $A^{0,\bullet}$ equipped with its differential graded algebra structure. There is a canonical inclusion $i\: :\:A^{0,\bullet}\hookrightarrow\operatorname{Tot}_{N}\left(A\right)$ which is a strict $C_{\infty}$-algebra map.
\end{lem}
\begin{proof}
	The map $i$ is clearly a cochain map. The $m_{\bullet}^{[0]}$ is trivial, and by setting $l=0$ in \eqref{formulafinale} we obtain that $i$ is strict.
\end{proof}

The $m_{\bullet}^{[1]}$ is given in \cite{Getz}. We first set a convenient basis for $NC_{1}$. Notice that the maps $E_{n}$ are all injective. This allows us to interpret $NC_{1}$ as a submodule of $\Omega^{\bullet}(1)$. Recall that $\Omega^{\bullet}(1)$ is the free differential graded commutative algebra generated by the degree-zero variables $t_{0}$, $t_{1}$ modulo the relations
\[
t_{0}+t_{1}=1,\quad dt_{0}+dt_{1}=0.
\]
 $NC_{1}^{0}$ is a two-dimensional vector space generated by $\lambda_{0}$ and $\lambda_{1}$ and $NC_{1}^{1}$ is one-dimensional generated by $\lambda_{0,1}$. We have
\begin{equation}\label{basisconv}
E_{1}\left( \lambda_{0}\right)=t_{0},\quad E_{1}\left( \lambda_{1}\right)=t_{1},\quad E_{1}\left( \lambda_{0,1}\right)=t_{0}dt_{1}-t_{1}dt_{0}.
\end{equation}
Let $t:=t_{0}$, hence $t_{1}=1-t$. Then $\Omega^{\bullet}(1)$ may be considered as the free differential graded commutative algebra generated by $t$ in degree zero and $NC_{1}$ is the subgraded module generated by $1,t,dt$. In particular $1$ is the unit of the $C_{\infty}$-structure.
\begin{prop}[\cite{Getz}]\label{Getzappl1}
The structure $m^{[1]}_{\bullet}$ on $NC_{1}$ is defined as follows:
\begin{enumerate}
\item $m_{2}^{[1]}(t,t)=t$,
\item $m_{n+1}^{[1]}(dt^{\otimes i}, t, dt^{\otimes n-i})=(-1)^{n-i}\binom{n}{i}m^{[1]}_{n+1}(t,dt,\dots ,dt)$,
\item $m_{n+1}^{[1]}(t,dt,\dots ,dt)=\frac{B_{n}}{n\text{!}}dt$, where $B_{n}$ are the second Bernoulli numbers,
\end{enumerate}
and all remaining products vanish.
\end{prop}
It remains to find a formula for $m_{n}^{[2]}|_{\left( NC_{2}^{1}\right)^{\otimes n}}$. For $n>2$ we are not aware of an explicit formula.
\begin{prop}\label{speriam} Consider $NC_{2}^{\bullet}$ equipped with $m_{\bullet}^{[2]}.$ We have $$m_{2}^{[2]}(\lambda_{01},\lambda_{02})=m_{2}^{[2]}(\lambda_{01},\lambda_{12})=m_{2}^{[2]}(\lambda_{02},\lambda_{12})=\frac{1}{6}\lambda_{012}.$$
\end{prop}
\begin{proof}
By explicit calculation. The details are given in Appendix \ref{sec2dimsimplex}.
\end{proof} 
Consider $A^{\bullet, \bullet}$ as above. We denote by $N(A)^{\bullet, \bullet}$ its bigraded bidifferential module (see Appendix \ref{cosimplicial modules} for a definition) and by $\operatorname{Tot}_{N}\left(A\right)\in dgMod$ its associated total complex. It is well known that there is a natural isomorphism $\psi\: : \: \operatorname{Tot}_{N}\left(A\right)\to\int^{[n]\in\boldsymbol{\Delta}}NC_{n}\otimes A^{n, \bullet}$ of differential graded modules (see Lemma \ref{isompsi} for a proof). With an abuse of notation we denote again by $m_{\bullet}$ the $C_{\infty}$-structure induced on $\operatorname{Tot}_{N}\left(A\right)$ via the isomorphism $\psi$. We have the following.
\begin{thm}\label{functor}
The association \begin{equation}
A^{\bullet,\bullet}\mapsto \left( \operatorname{Tot}_{N}\left(A\right),m_{\bullet}\right) 
\end{equation} is part of a functor $\mathrm{cdgA}\to \left( C_{\infty}-\operatorname{Alg}\right)_{\Bbbk,str}$.
\end{thm}
We give an explicit formula for $m_{\bullet}$ in $\operatorname{Tot}_{N}\left(A\right)$ for elements of degree $0,1$. We denote by $\tilde{\partial}\: : \: A^{\bullet,\bullet}\to A^{\bullet+1,\bullet}$ the differential given by the alternating sum of coface maps, in particular $\tilde{\partial}=d^0-d^1$ on $A^{0,\bullet}$.
\begin{thm}\label{productthm>2}Let $l>2$.
\begin{enumerate} 
\item Let $a_{1}, \dots, a_{l}\in\operatorname{Tot}^{1}_{N}\left(A\right)$ and let $b_{i}\in A^{1,0}$, $c_{i}\in A^{0,1}$ be such that $a_{i}=b_{i}+c_{i}$ for every $i=1, \dots, l$. Then
\[
m_{l}\left( a_{1}, \dots, a_{l}\right) = \sum_{i=1}^{l}(-1)^{l-1}\binom{l-1}{i-1}\frac{B_{l-1}}{\left(l-1 \right) !}b_{1}\cdots \widehat{b}_{i}\cdots b_{l}\tilde{\partial}c_{i}+m_{l}\left(c_{1}, \dots, c_{n} \right) .
\]
\item Let $x\in A^{0,0}$. Then
\[
m_{l}\left(a_{1}, \dots,a_{i-1}, x, a_{i+1}, \dots, a_{l}\right)=(-1)^{i}\binom{l-1}{i-1}\frac{B_{l-1}}{\left(l-1 \right) !}b_{1}\cdots b_{l}\left( \tilde{\partial}x\right),
\]
and if we replace some $a_{i}$ by an element in $A^{0,0}$, the above expression vanishes. 
\end{enumerate}

\end{thm}
 \begin{proof}
For two subsets $\mathcal{B}, \mathcal{C}\subseteq \left\lbrace 1, \dots, l \right\rbrace $ such that $\mathcal{B}\cup\mathcal{C}=\left\lbrace 1, \dots, l \right\rbrace$, we denote by $m_{l}\left(b_{\mathcal{B}}, c_{\mathcal{C}}\right)$ the expression $m_{l}\left(y_{1}, \dots, y_{l} \right)$ such that $y_{i}=b_{i}$ for $i\in \mathcal{B}$ and $y_{i}=c_{i}$ for $i\in \mathcal{B}$. In particular, we have $\left| m_{l}\left(b_{\mathcal{B}},  c_{\mathcal{C}}\right)\right|= \left|\mathcal{B} \right|+2-l+\left| \mathcal{C}\right|$. It follows that $\left|\mathcal{B} \right| \geq l-2$ and hence $\left|\mathcal{C} \right| \leq 2$. We have
\begin{align*}
m_{l}\left( a_{1}, \dots, a_{l}\right) & = \sum_{\mathcal{B}, \mathcal{C}}m_{l}\left(b_{\mathcal{B}}, c_{\mathcal{C}}\right)\\& =
\sum_{\left| \mathcal{B}\right|=l-2 , \left| \mathcal{C}\right|=2}m_{l}\left(b_{\mathcal{B}}, c_{\mathcal{C}}\right)+ \sum_{\left| \mathcal{B}\right|=l-1 ,  \left| \mathcal{C}\right|=1}m_{l}\left(b_{\mathcal{B}}, c_{\mathcal{C}}\right)+ \sum_{\left| \mathcal{B}\right|=l ,  \left| \mathcal{C}\right|=0}m_{l}\left(b_{\mathcal{B}}, c_{\mathcal{C}}\right).\\&
\end{align*}
The first summand vanishes by \eqref{formulafinale}. We conclude
\begin{equation}\label{step1}
m_{l}\left( a_{1}, \dots, a_{l}\right) = \sum_{i=1}^{l}m_{l}\left(b_{1}, \dots, c_{i}, \dots, b_{l}\right)+m_{l}\left(b_{1}, \dots, b_{l}\right).
\end{equation}
We calculate explicitly $m_{l}\left(b_{1}, \dots, b_{l},c_{1}\right)\in \operatorname{Tot}^{2}_{N}\left(A\right)$. We work in $\int^{[n]\in\boldsymbol{\Delta}}NC_{n}\otimes A^{n,\bullet} $ and we use the isomorphism $\psi$ (see Appendix \ref{cosimplicial modules}). In particular, for $c\in \operatorname{Tot}^{0,1}_{N}\left(A\right)$ we have a $(0,1)$ element $\psi(b)\in \int^{[n]\in\boldsymbol{\Delta}}NC_{n}\otimes A^{\bullet,n} $. Its projection at $NC_{1}^{0}\otimes A^{1,1}\subset NC_{1}^{\bullet}\otimes A^{1,\bullet}$ is
\[
\lambda_{0}\otimes \sigma_{0}^{*}\left(c \right)+\lambda_{1}\otimes \sigma_{1}^{*}\left(c \right).
\]
Let $d^{0}, d^{1}$ be the coface maps of $A$. By definition, $\sigma_{0}$ corresponds to the coface map $d^{1}$ and $\sigma_{1}$ corresponds to the coface map $d^{0}$. Hence we can write $NC_{1}^{\bullet}\otimes A^{1, \bullet}$ as $t\otimes d^{1}c +(1-t)\otimes d^{0} c$ (see \eqref{basisconv}). Similarly, an element $b$ of bidegree $(0,1)$ can be written as $-dt\otimes b \in NC_{1}^{1}\otimes A^{1,0}$. We have
\begin{align*}
\nonumber m_{l}^{1,1}\left(\phi(b)_{1}, \dots, \phi(b)_{l-1}, \phi(c)_{l}\right)_{1} 
&\nonumber =  m_{l+1}^{1,1}\left(-dt\otimes b_{1}, \dots,-dt\otimes  b_{l-1}, t\otimes d^{1}c_{l} +(1-t)\otimes d^{0} c_{l}\right)\\
&=(-1)^lm_{l+1}^{[1]}\left(dt, \dots,dt,t\right)\otimes b_{1}\cdots b_{l-1}\left( d^{1}c_{l}-d^{0}c_{l}\right)  \\
& =(-1)^l\frac{B_{l-1}}{\left(l-1 \right) !}dt\otimes\left( b_{1}\cdots b_{l-1}\right)  \tilde{\partial}c_{l}\in NC_{1}^{1}\otimes A^{1,1}
\end{align*}
where $ b_{1}\cdots b_{l-1}$ has to be understood as a multiplication inside $A^{1,0}$. The above expression defines an element of bidegree $(1,1)$ inside $\int^{[n]\in\boldsymbol{\Delta}}NC_{n}\otimes A^{\bullet,n} $. By applying $\psi^{1}$ we get
\[
m_{l}\left(b_{1}, \dots, b_{l-1},c_{l}\right)=(-1)^{l-1}\frac{B_{l-1}}{\left(l-1 \right) !}b_{1}\cdots b_{l-1}\tilde{\partial}c_{l}.
\]
Thanks to point 2 in Proposition \ref{Getzappl1}, we have
\[
m_{l}\left(b_{1}, \dots,b_{i-1}, c_{i}, b_{i+1}, \dots, b_{l}\right)=(-1)^{l-1}\binom{l-1}{i-1}\frac{B_{l-1}}{\left(l-1 \right) !}b_{1}\cdots b_{l}\tilde{\partial}c_{i}.
\]
By \eqref{step1}, we have
\begin{equation}\label{step2}
m_{l}\left( a_{1}, \dots, a_{l}\right) = \sum_{i=1}^{l}(-1)^{l-1}\binom{l-1}{i-1}\frac{B_{l-1}}{\left(l-1 \right) !}b_{1}\cdots \widehat{b}_{i}\cdots b_{l}\tilde{\partial}c_{l}+m_{l}\left(b_{1}, \dots, b_{l}\right).
\end{equation}
Now let $x\in A^{0,0}$. Then, the above computations give
\begin{align}
\nonumber m_{l}\left(a_{1}, \dots,a_{i-1}, x, a_{i+1}, \dots, a_{l}\right)& =m_{l}\left(b_{1}, \dots,b_{i-1}, x, b_{i+1}, \dots, b_{l}\right)\\
\nonumber& =(-1)^{l-1}(-1)^{l-i-1}\binom{l-1}{i-1}\frac{B_{l-1}}{\left(l-1 \right) !}b_{1}\cdots b_{l}\left( \tilde{\partial}x\right) \\
\label{step2unit}& =(-1)^{i}\binom{l-1}{i-1}\frac{B_{l-1}}{\left(l-1 \right) !}b_{1}\cdots b_{l}\left( \tilde{\partial}x\right).
\end{align}
\end{proof}
The following corollary follows directly from the previous proof.
\begin{cor}\label{corresiduemap}
	Let $n>2$ and and let $a_{1}, \dots, a_{n}\in \operatorname{Tot}^{1}_{N}\left(A\right)$. We have
	\[
	m_{n}\left( a_{1}, \dots, a_{n}\right)=d_{1}+d_{2}+d_{3},
	\]
	where $d_{1}\in \operatorname{Tot}^{0,2}_{N}\left(A\right)$, $d_{2}\in \operatorname{Tot}^{1,1}_{N}\left(A\right)$ and $d_{3}\in\operatorname{Tot}^{0,2}_{N}\left(A\right)$. Then $d_{1}=0$ and $d_{3}=m_{n}\left( b_{1}, \dots, b_{n}\right) $.
\end{cor}
\begin{thm}\label{productthm=2}
Let $a_{1},a_{2}\in\operatorname{Tot}^{1}_{N}\left(A\right)$ and let $b_{i}\in \operatorname{Tot}^{1,0}_{N}\left(A\right)$, $c_{i}\in \operatorname{Tot}^{0,1}_{N}\left(A\right)$ such that $a_{i}=b_{i}+c_{i}$ for $i=1,2$. Then 
\begin{align*}
m_{2}\left( a_{1},a_{2}\right)& = m_{2}\left( c_{1},c_{2}\right)+ m_{2}\left( b_{1},c_{2}\right)+m_{2}\left( c_{1},b_{2}\right)+m_{2}\left( b_{1},b_{2}\right),
\end{align*}
where
\begin{enumerate}
\item $m_{2}(c_{1}, c_{2})=c_{2}c_{2}\in A^{0,2},$
\item $m_{2}\left(b_{1},c_{2} \right)= -\frac{1}{2}b_{1}\tilde{\partial}c_{2}+b_{1}d^{0}c_{2}$,
\item $m_{2}\left(b_{1},b_{2} \right)= \frac{1}{6}\left(- \left(d^0b_{1}\left(d^1b_{2}+d^2b_{2} \right)  \right)
+ \left(d^1b_{1}\left(d^0b_{2}-d^2b_{2} \right)  \right)
 +  \left(d^2b_{1}\left(d^0b_{2}+d^1b_{2} \right)  \right)\right).$
\end{enumerate}
Let $x,y\in \operatorname{Tot}^{0}_{N}\left(A\right)$. Then, $m_{2}(c_{1},x)=c_{1}x\in \operatorname{Tot}^{0,1}_{N}\left(A\right)$ , $m_{2}(x,y)=xy\in A^{0,1}$, and $m_{2}\left( b_{1},x\right)=-\frac{1}{2}b_{1}\tilde{\partial}x+b_{1}d^{0}x\in\operatorname{Tot}^{1,0}_{N}\left(A\right)$.
\end{thm}
\begin{proof}
The second and last terms of $m_{2}(a_{1},a_{2})$ can be calculated by a computation similar to the proof above. We have
\begin{align*}
\nonumber m_{2}^{1,1}\left(\psi(b)_{1},  \psi(c)_{2}\right)_{1} 
&\nonumber =  m_{2}^{1,1}\left(-dt\otimes b_{1}, t\otimes d^{1}c_{2} +(1-t)\otimes d^{0} c_{2}\right)\\
&=-m_{2}^{[1]}\left(dt,t\right)\otimes b_{1}c_{2}+m_{2}^{[1]}\left(dt,1\right)\otimes \left(- b_{1}d^{0}c_{2}\right)  \\
& +m_{2}^{[1]}(dt,t)\otimes \left(b_{1}d^{0}c_{2} \right) \\
& =dt\otimes \left( B_{1}b_{1}\tilde{\partial}c_{2}-b_{1}d^{0}c_{2}\right),
\end{align*}
and thus $m_{2}\left(b_{1},c_{2} \right)= -B_{1}b_{1}\tilde{\partial}c_{2}+b_{1}d^{0}c_{2}$. It remains to add $m_{2}\left(b_{1},b_{2}\right)$ where $B_{1}=\frac{1}{2}$. The expression for $m_{2}\left(b_{1},x\right)$ can be computed in the same way.
By  Appendix \ref{sec2dimsimplex}, $NC_{2}$ is the graded vector space generated by $\lambda_{0}, \lambda_{1},  \lambda_{2}$ in degree $0$, $\lambda_{12}, \lambda_{02},  \lambda_{01}$ in degree $1$ and $\lambda_{012}$ in degree 2. In particular, for $b\in \operatorname{Tot}^{1,0}_{N}\left(A\right)$ we have a $(1,0)$ element $\psi(b)\in \int^{[n]\in\boldsymbol{\Delta}}NC_{n}\otimes A^{\bullet,n} $. Its projection at $ NC_{2}^{\bullet}\otimes A^{2,\bullet}$ is
\[
\lambda_{12}\otimes d^{0}\left(b \right)+\lambda_{02}\otimes d^{1}\left(b \right)+\lambda_{01}\otimes d^{2}\left(b \right).
\]
By Proposition \ref{speriam} we have
\begin{align*}
& m_{2}^{2,2}\left(\psi\left( b_{1}\right) ,\psi\left( b_{2}\right)\right)_{2}=\\
& =m_{2}^{2,2}\left(\lambda_{12}\otimes d^{0}\left(b_{1} \right)+\lambda_{02}\otimes d^{1}\left(b_{1} \right)+\lambda_{01}\otimes d^{2}\left(b_{1} \right) ,\lambda_{12}\otimes d^{0}\left(b_{2} \right)+\lambda_{02}\otimes d^{1}\left(b_{2} \right)+\lambda_{01}\otimes d^{2}\left(b_{2} \right)\right)_{2}\\
&=\frac{1}{6}\lambda_{012}\otimes \left( -\left(d^0b_{1}\left(d^1b_{2}+d^2b_{2} \right)  \right)
+ \left(d^1b_{1}\left(d^0b_{2}-d^2b_{2} \right)  \right)
 +  \left(d^2b_{1}\left(d^0b_{2}+d^1b_{2} \right)  \right)\right). \\
\end{align*}
Then 
\[
m_{2}\left(b_{1},b_{2} \right)= \frac{1}{6}\left( -\left(d^0b_{1}\left(d^1b_{2}+d^2b_{2} \right)  \right)
+ \left(d^1b_{1}\left(d^0b_{2}-d^2b_{2} \right)  \right)
 +  \left(d^2b_{1}\left(d^0b_{2}+d^1b_{2} \right)  \right)\right).
\]
\end{proof}

\section{Geometric connections}\label{third}
We put the $C_{\infty}$-structure of Section \ref{sectGetztler} on the cosimplicial module of differential forms on a simplicial manifold $M_{\bullet}$. We show that Maurer-Cartan elements induce a flat connection form on $M_{0}$ (Corollary \ref{corpronilpot}). 
\subsection{Simplicial De Rham theory}\label{sectionsimplicial De Rham Theory}
 We recall some basic notions about simplicial manifolds. Our main reference is \cite{Dupont} and \cite{Dupont2}. By complex manifold we mean a smooth complex manifold. Let $\mathrm{Diff}_{\C}$ be the category of complex manifold. We denote with $\mathrm{dgA}$ the category of complex commutative non-negatively graded differential graded algebras.
\begin{defi}
A \emph{simplicial manifold} $M_{\bullet}$ is a simplicial object in $\mathrm{Diff}_{\C}$.
\end{defi}
Each smooth complex manifold $M$ can be viewed as a constant simplicial manifold $M_{\bullet}$, where $M_{n}:=M$ and all the degenerancy and face maps are equal to the identity. Let $G$ be a Lie group, and let $M$ be a manifold equipped with a left smooth (or holomorphic) $G$-action. We define the simplicial manifold $M_{\bullet}G$ as follows: $M_{n}G=M\times G^{n}$ and
The face maps $d^{i}\: : \: M_{n}G\to M_{n-1}G$ for $i=0,1,\dots n$ are
\[
d^{i}(x,g_{1}, \dots , g_{n}):=\begin{cases}
& (g_{1}x, g_{2}, \dots , g_{n}),\text{ if }i=0,\\
& (x,g_{1}, \dots , g_{i}g_{i+1}, \dots ,g_{n}),\text{ if }1<i<n\\
& (x,g_{1}, \dots , g_{n-1}),\text{ if }i=n.
\end{cases}
\]
The degenerancy maps $s^{i}\: : \: M_{n}G\to M_{n+1}G $ are defined via
\[
s^{i}(x,g_{1}, \dots , g_{n}):=\left(x,g_{1}, \dots, g_{i}, e, g_{i+1}, \dots , g_{n} \right) 
\]
for $i=1, \dots ,n$.
\begin{defi}
	We call the simplicial manifold $M_{\bullet}G$ the action groupoid.
\end{defi}
In particular the geometric realization of $M_{\bullet}G$ is weakly equivalent to the Borel construction $EG\times_{G}M$ and if the action of $G$ is free, the projection $\pi\: : \: EG\times_{G}M\to M/G$ is a homotopy equivalence. Let $A_{DR}\: : \: \mathrm{Diff}_{\C}\to \mathrm{dgA}$ be the smooth complex De Rham functor, i.e., $A_{DR}\left( M\right) $ is the differential graded algebra of smooth complex valued differential forms on $M$. $A_{DR}$ is contravariant and $A_{DR}\left(M_{\bullet} \right)$ is a cosimplicial complex commutative differential graded algebra for any simplicial manifold $M_{\bullet}$. As explained in the previous section, $\operatorname{Tot}_{TW}\left( A_{DR}\left(M_{\bullet} \right)\right) $ is a differential graded commutative algebra over $\C$. We obtain a contravariant functor
\[
\operatorname{Tot}_{TW}\left( A_{DR}\left(- \right)\right)\: :\: \mathrm{Diff}^{\boldsymbol{\Delta}^{op}}_{\C}\to \mathrm{dgA}.
\]
 A smooth map between simplicial manifolds $f_{\bullet}\: : \: M_{\bullet}\to N_{\bullet}$ induces a morphism of differential graded algebras via $\left( f^{*}(w)\right)_{n}:=f^{*}_{n}(w_{n})$. If $M_{\bullet}$ is a constant simplicial manifold then $\operatorname{Tot}_{TW}\left( A_{DR}\left(M_{\bullet} \right)\right)$ is naturally isomorphic to $A_{DR}(M_{\bullet})$. Given a simplicial manifold $M_{\bullet}$ and consider the cosimplicial complex commutative differential graded algebra $A_{DR}\left(M_{\bullet} \right)$, by Theorem \ref{functor}, we have a functor
\[
\operatorname{Tot}_{N}\left( A_{DR}\left(- \right)\right)\: :\: \mathrm{Diff}_\C^{\boldsymbol{\Delta}^{op}}\to \left( C_{\infty}-\operatorname{Alg}\right)_{\C,str}.
\]
There is a De Rham theorem for simplicial manifolds. This is very useful because it allows us to determine the cohomology of $\operatorname{Tot}_{N}\left( A_{DR}\left(M_{\bullet} \right)\right)$. Let $\mathrm{Top}$ be the category of topological spaces. Let $T_{\bullet}\in \mathrm{Top}^{\boldsymbol{\Delta}^{op}}$ be a simplicial topological space.
Let $\boldsymbol{\Delta}_{+}$ be the subcategory of $\boldsymbol{\Delta}$ with the same objects but only injective maps. By restriction, we have a functor $T_{\bullet}\: : \: \boldsymbol{\Delta}_{+}^{op}\to \mathrm{Top}$. \\
\begin{defi}
We call the topological spaces
\[
\left| \left| T_{\bullet} \right| \right|:=\int^{[n]\in \boldsymbol{\Delta}_{+}}T_{n}\times \Delta_{geo}[n]\text{ and }\left| T_{\bullet} \right| :=\int^{[n]\in \boldsymbol{\Delta}}T_{n}\times \Delta_{geo}[n]
\]
the \emph{fat realization} and the \emph{geometric realization} of $T_{\bullet}$ respectively.
\end{defi} 
\begin{rmk}
The natural quotient map $||T_{\bullet}||\to |T_{\bullet}|$ is not a weak equivalence in general. However, if the simplicial topological space is ``good'' (see the appendix  of \cite{Segal}), then it is a weak-equivalence. In particular, for a simplicial manifold $M_{\bullet}$, this is true when its degeneracy $s^{i}$ maps are embeddings. We call this class of simplicial manifolds \emph{good} simplicial manifolds.
\end{rmk}

\begin{prop}[\cite{Dupont}]\label{Derhamthmsimpl}
Let $M_\bullet$ be a simplicial manifold, we have a sequence of multiplicative isomorphisms:
\[
H^{\bullet}\left(\operatorname{Tot}_{TW}\left( A_{DR}\left(M_{\bullet} \right)\right) \right)\cong H^{\bullet}\left(\operatorname{Tot}_{N}\left( A_{DR}\left(M_{\bullet} \right)\right)\right) \cong H^{\bullet}\left(||M_{\bullet}||, \C \right).
\]
\end{prop}
 We introduce the complex of smooth logarithmic differential forms. Let $M$ be a complex manifold. 

\begin{defi}
A normal crossing divisor $D\subset M$ is given by $\cup_{i} D_{i}$, where each $D_{i}$ is a non-singular divisor (a codimension $1$ object) and for each $p\in D_{i_{1}}\cap \dots \cap D_{i_{l}}$ there exist local coordinates $z=\left( z_{1}, \dots, z_{n}\right)\: : \: \C^{n}\to U\subset M$ near $p$ such that $U\cap D_{i_{1}}\cap \dots \cap D_{i_{l}}$ is given by the equation $\prod^{q}_{i=1}z_{j_{i}}=0$.
\end{defi}
 A differential form $w$ in $A_{DR}^{\bullet}(\log \left( D\right) )$ is a smooth complex valued differential form on $M-D$ whose extension on $M$ admits some singularites along $D$ of degree $1$. More precisely it can be viewed as a non-smooth complex differential form on $M$ such that
\begin{enumerate}
\item for any given holomorphic coordinates $z=\left( z_{1}, \dots, z_{n}\right)\: : \: \C^{n}\to U\subset M$ such that $U\cap D=\emptyset$, $w$ can be written as an ordinary  smooth complex valued differential forms 
\[
\sum f(z_{1}, \dots z_{n})d{z_{i_{1}}\wedge\dots \wedge  dz{i_{p}}}\wedge d{\overline{z}_{j_{1}}\wedge\dots \wedge d\overline{z}{j_{q}}}
\]
where $f$ is a smooth function over $U$.
\item For any given holomorphic coordinates  $z=\left( z_{1}, \dots, z_{n}\right)\: : \: \C^{n}\to U\subset M$  near $p\in D$, such that $U\cap D_{i_{1}}\cap \dots \cap D_{i_{l}}$ is given by the equation $\prod^{q}_{i=1}z_{j_{i}}=0$; $w$ can be written as 
 \[
\sum w_{J}\wedge \frac{dz_{j_{1}}}{z_{j_{1}}}\wedge\cdots \wedge\frac{dz_{j_{l}}}{z_{j_{l}}}
\]
for $j_{i}\in \left\lbrace 1, \dots, q \right\rbrace $ and where $w_{J}$ is a smooth complex valued differential forms on $U$. 
\end{enumerate}
The graded vector space $A_{DR}^{\bullet}(\log \left( D\right))$ equipped with the differential $d$ and wedge product $\wedge$ is a commutative differential graded algebra. Its cohomology gives  the cohomology of $M-{D}$.
\medskip
\begin{prop}[\cite{Deligne2}]\label{delignholom}
The inclusion $M-D\hookrightarrow M$ induces a map
\[
A_{DR}^{\bullet}(\log \left( D\right) )\to A_{DR}^{\bullet}(M-D )
\]
which is a quasi-isomorphism of differential graded algebras.
\end{prop}
\begin{defi}
Let $\mathrm{Diff}_{Div}$ be the category of complex manifolds equipped with a normal crossing divisors, i.e the objects are pairs $(M,\mathcal{D})$ where $\mathcal{D}$ is a normal crossing divisor of $M$ and the maps are holomorphic maps $f\: : \: \left(M,\mathcal{D} \right)\to \left(N,{\mathcal{D}'}\right) $ such that
\[
f^{-1}(\mathcal{D}')\subset {\mathcal{D}}.
\] 
A \emph{simplicial manifold $M_{\bullet}$ with a simplicial normal crossing divisor $\mathcal{D}_{\bullet}$ }is a simplicial objects in $\mathrm{Diff}_{Div}$.
\end{defi}
So given a simplicial manifold with divisor $(M_{\bullet}, \mathcal{D}_{\bullet})$, then $\left( M-D\right)_{\bullet}$ defines by $\left( M-\mathcal{D}\right)_{n}=M_{n}-\mathcal{D}_{n}$ is again a simplicial manifold and $A_{DR}\left(\left( M-\mathcal{D}\right)_{\bullet}\right)$ and $A_{DR}\left( \log\left( \mathcal{D}_{\bullet}\right) \right)$ are cosimplicial commutative dg algebras.
\medskip
\begin{prop}\label{simplicialcrossdiv}
The inclusion $\left( M-\mathcal{D}\right)_{\bullet} \hookrightarrow M_{\bullet}$ induces a map between chain complexes 
\[
\operatorname{Tot}_{N}\left(A_{DR}(\log \left( \mathcal{D}_{\bullet}\right) ) \right)\to\operatorname{Tot}_{N}\left(A_{DR}\left(\left( M-\mathcal{D}\right)_{\bullet}\right)  \right)
\]
which is a quasi-isomorphism.
\end{prop}
\begin{proof}
This is a direct consequence of Proposition \eqref{delignholom} and Lemma below. For a double complex $\left( A^{\bullet,\bullet},d,d'\right), $ we define a new double complex $\left(\left(  E^{A} _{1}\right)^{\bullet,\bullet},d_{1}, {d'}_{1}\right) $, where $\left( E^{A}_{1}\right)^{p,q}:=H^{p}\left(A^{\bullet,q},d \right)$, $d_{1}=d$ and ${d'}_{1}=0$.
\begin{lem}[\cite{Dupont}, Lemma 1.19]\label{spectralcomplexes}
Let $f\: : \: A^{\bullet, \bullet}\to B^{\bullet,\bullet}$ be a morphism of double complex. Assume that $A^{p,q}=B^{p,q}=0$ if $p$ or $q$ are negative and that $f\: : \:  E^{A}_{1}\to E^{B}_{1}$ is an isomorphism between cochain complexes. Then $f$ induces an isomorphism in the cohomology of the total complex.
\end{lem}
\end{proof}

\subsection{Restriction to ordinary flat connections on $M_{0}$}\label{restr}
 Let $M_{\bullet}$ be a simplicial manifold with connected cohomology. Then $A_{DR}(M_{\bullet})$ is a cosimplicial unital commutative differential graded algebra. By Theorem \ref{functor} $\operatorname{Tot}_{N}\left(A_{DR}(M_{\bullet})\right)$ is an unital $C_{\infty}$-algebras. We denote its structure by ${m}_{\bullet}$ where ${m}_{1}=D$ such that $D$ is the differential on $\operatorname{Tot}_{N}\left(A_{DR}(M_{\bullet})\right)$ defined on elements of bidegree $(p,q)$ by 
\[
D(a)=\tilde{\partial} a+(-1)^{p}da,
\]
where $\tilde{\partial}$ is differential obtained by the alternating sum of the pullback of the cofaces maps of the simplicial manifold. We denote the unit by $1$. It corresponds to the constant function at $1$ inside $A_{DR}^{0}(M_{0})$.
\begin{defi}\label{reducedhomologicalparifed}
Let $B\subseteq \left( \operatorname{Tot}_{N}\left(A_{DR}(M_{\bullet})\right), m_{\bullet}\right)$ be a $C_{\infty}$-subalgebra. Let $W$ be a non-negatively graded vector space of finite type such that $W^{0}=\C$.
\begin{enumerate}
\item Let $\left( W, m_{\bullet}^{W}\right)$ be a unital minimal $C_{\infty}$-algebra. Let $\delta$ be the corresponding differential induced on $T^{c}\left(W_{+}[1] \right) $. Let $g_{\bullet}\: : \:\left( W, m_{\bullet}^{W}\right)\to B$ be a unital $C_{\infty}$-map and let $C$ be the corresponding Maurer-Cartan element in $C\in \operatorname{Conv}_{C_{\infty}}\left(\left(W_{+}, m_{\bullet}^{W_{+}}\right) ,\left(B, m_{\bullet} \right) \right)$
We call $(C, {\delta}^{*})$ the \emph{reduced  homological pair associated to $g_{\bullet}$} or simply\emph{ reduced homological pair with coefficients in $B$}. If $g_{\bullet}\: : \:\left( W, m_{\bullet}^{W}\right)\to  \left( B, m_{\bullet}\right)$ is a minimal model and $B$ is a model, we call $(C, {\delta}^{*})$ \emph{a good reduced homological pair associated to $g_{\bullet}$ } or simply \emph{ good reduced homological pair  with coefficients in $B$}.
\item Let $\left( W, m_{\bullet}^{W}\right)$ be a unital minimal $1-C_{\infty}$-algebra. Let $\delta$ be the corresponding differential induced on $\left( T^{c}\left(W_{+}[1] \right)\right)^{0}$. 
Let $f_{\bullet}\: : \:\left( W, m_{\bullet}^{W}\right)\to \mathcal{F}\left(B, m_{\bullet} \right)$ be a unital $1$-$C_{\infty}$-map and let $\overline{C}$ be the corresponding Maurer-Cartan element in $\operatorname{Conv}_{1-C_{\infty}}\left(\left(W_{+}, m_{\bullet}^{W_{+}}\right) ,\left(B, m_{\bullet} \right) \right)$. We call it the \emph{ degree zero geometric connection associated to $f_{\bullet}$} or simply \emph{ degree zero geometric connection with coefficients in $B$}. If $f_{\bullet}\: : \:\left( W, m_{\bullet}^{W}\right)\to   \mathcal{F}\left(B, m_{\bullet} \right)$ is a $1$-minimal model and $B$ is a $1$-model, we call $\overline{C}$ \emph{a good degree zero geometric connection with coefficients in $B$}.
A $1$-minimal model $f_{\bullet}\: : \:\left( W, m_{\bullet}^{W}\right)\to   \mathcal{F}\left(B, m_{\bullet} \right)$ is said to be holomorphic (with logarithmic singularities) if the image of $f_{n}$ contains only holomorphic elements (with logarithmic singularities). If $f_{\bullet}\: : \:\left( W, m_{\bullet}^{W}\right)\to   \mathcal{F}\left(B, m_{\bullet} \right)$ is holomorphic (with logarithmic singularities) we say that $\overline{C}$ is holomorphic (with logarithmic singularities).
\end{enumerate}
\end{defi}
\begin{rmk}\label{masseyprod}
 Let $M_{\bullet}$ be a simplicial manifold with connected and finite type cohomology. Let $f_{\bullet}\: : \:\left( W, m_{\bullet}^{W}\right)\to   \mathcal{F}\left( \operatorname{Tot}_{N}\left(A_{DR}(M_{\bullet})\right), m_{\bullet}\right)$ be a $1$-minimal model. By Proposition \ref{homotopyuniqueness}, the $1-C_{\infty}$ structure $m_{\bullet}^{W}$ is unique up to isomorphism. This products can explicit determined via the homotopy transfer theorem.
\end{rmk}
We fix a $C_{\infty}$-subalgebra $B\subset \left( \operatorname{Tot}_{N}\left(A_{DR}(M_{\bullet})\right), m_{\bullet}\right)$ with connected and finite type cohomology and let $\overline{C}$ be the degree zero geometric connection associated to a $f_{\bullet}\: : \:\left( W, {m'}_{\bullet}\right)\to   \mathcal{F}\left(B, m_{\bullet} \right)$. Let $
w_{1}, \dots, w_{i}, \dots$
 be a basis of $W_{+}$. Let $X_{1}, \dots X_{n}, \dots$ be the basis of $\left( W[1]\right)^{*}$ dual to $ s\left( w_1 \right) ,\dots s\left(  w_{n}\right) , \dots\in W_{+}[1]$. Let $\mathcal{R}_{0}\subset \widehat{\mathbb{L}}\left( \left( W_{+}^{1}[1]\right)^{*} \right)$ be the completed Lie ideal generated by ${\delta'}^{*}$. Hence $\overline{C}$ can be written as
 \begin{equation}\label{formulaconnection}
 	\overline{C}=\sum{v_{i}X_i}+\sum{v_{ij}X_{i}X_{j}}+\dots+\sum{v_{i_{1}\dots i_{r}}X_{i_{1}}\dots X_{i_{r}}}+\dots \in B^{1} \hat{\otimes} \left( \widehat{\mathbb{L}}(\left( W^{1}_{+}[1]\right)^{*})/\mathcal{R}_{0}\right) .
 \end{equation}
We show how to construct a flat connection starting from $\overline{C}$. Consider the projection $${r}\: : \:B\to B^{0,1}\subset A_{DR}\left(M_{0} \right) $$ that sends forms of bidegree $(p,q)$ to $0$ if $p\neq 0$ and preserves forms of bidegree $(0,q)$. Consider $A_{DR}\left(M_{0} \right)$ equipped with its differential graded algebra structure. Then 
 \[
 \operatorname{Conv}_{1-C_{\infty}}\left( \left( W_{+}, {m'}_{\bullet}\right) ,  A_{DR}\left(M_{0} \right) \right) 
 \]
 is a differential graded Lie algebra. In particular, it is a ordinary convolution Lie  algebra (compare with \cite{lodayVallette}).
 \begin{prop}\label{pres}
The pushforward along $r$ induces a map
\[
 r_{*}:=r\widehat{\otimes}\mathrm{Id}\: : \: B \hat{\otimes} \left( \widehat{\mathbb{L}}(\left( W^{1}_{+}[1]\right)^{*})/\mathcal{R}_{0}\right) \to  A_{DR}\left(M_{0} \right) \hat{\otimes} \left( \widehat{\mathbb{L}}(\left( W^{1}_{+}[1]\right)^{*})/\mathcal{R}_{0}\right)
 		\]
which is in $\left( L_{\infty}-ALG\right)_{p}$.
 \end{prop}
 \begin{proof}
By Corollary \ref{corresiduemap}, we have
 		\[
 		r{m}_{n}(a_{1}, \dots, a_{n})={m}_{n}(ra_{1}, \dots, ra_{n})=\begin{cases}
 		d ra_{1}, & \text{ if }n=1,\\
 		{m}_{2}(ra_{1}, ra_{2}), & \text{ if }n=2,\\
 		0, & \text{ otherwise }.
 		\end{cases}
 		\]
for $a_{1}, \dots, a_{n}\in B^{1}$. Let $\overline{C}\in B \hat{\otimes} \left( \widehat{\mathbb{L}}(\left( W^{1}_{+}[1]\right)^{*})/\mathcal{R}_{0}\right) $ be a Maurer-Cartan element and let
 $${C}'\in\left(A_{DR}^{1}(M_{0})\right)\widehat{\otimes} \widehat{\mathbb{L}}  \left( \left( W^{1}_{+}[1]\right)^{*}\right), \quad \text{ resp. } \tilde{C}\in\left(A_{DR}^{0}(M_{1})\right)\widehat{\otimes} \widehat{\mathbb{L}}  \left( \left( W^{1}_{+}[1]\right)^{*}\right)$$ be such that $\overline{C}={C}'+\tilde{C}$. In particular, ${r_{*}}C=C'$ and 
 	\begin{align*}
 	0& =r_{*}\left( D C+\sum_{k>1}\frac{{l}_{k}\left(C,\dots,C\right)}{k \text{!}}\right) \\
 	& = \left( d r_{*}C+\sum_{k>1}\frac{{l}_{k}\left(r_{*}C,\dots,r_{*}C\right)}{k \text{!}}\right)\\
 	& = \left( d r_{*}C+\frac{{l}_{2}\left(r_{*}C,r_{*}C\right)}{2}\right),\\
 	\end{align*}
 	i.e., $r_{*}C$ is a Maurer-Cartan element in $\operatorname{Conv}_{1-C_{\infty}}\left( \left( W_{+}, {m'}_{\bullet}\right) ,  A_{DR}\left(M_{0} \right) \right).$
 \end{proof}
 \begin{rmk}\label{mapr}
 	Let $M_{\bullet}=M_{\bullet}G$ for some complex manifold $M$ and discrete group $G$. The morphism of simplicial manifolds $ M_{\bullet}\left\lbrace e\right\rbrace \to  M_{\bullet}G$ given by the inclusion gives the map
 	\begin{equation}
 	r\: : \: \operatorname{Tot}_{N}\left(A_{DR}\left(  M_{\bullet}G\right) \right)\to  A_{DR}\left(  M\right)
 	\end{equation}
 	which is a strict morphism of $C_{\infty}$-algebras.
 \end{rmk}
 A Lie algebra $\mathfrak{u}$ is said to be \emph{pronilpotent} if is the projective limit of finite dimensional Lie algebras
 \[
 \mathfrak{u}\cong \varprojlim_{i}\left( \mathfrak{u}/I^{i}\right),
 \]
where $I^{\bullet}$ is defined via $I^{0}=\mathfrak{u}$, $I^{i}:=\left[I^{i-1}, \mathfrak{u}\right]$ for $i\geq 1$.
\begin{cor}\label{corpronilpot}
Let $\mathfrak{u}':=\widehat{\mathbb{L}} \left( \left( W^{1}_{+}[1]\right)^{*}\right)/\mathcal{R}_{0}$.
	\begin{enumerate}
		\item We have $
		\operatorname{Conv}_{1-C_{\infty}}\left( \left( W_{+}, {m'}_{\bullet}\right) , \left( A_{DR}\left(M_{0} \right)\right) \right)=\left( {l_{\bullet}},A_{DR}(M_{0})\widehat{\otimes} \mathfrak{u}'\right).
$\\
		such that 
		$${l_{1}}=-d,\quad {l_{2}}=[-,-] ,\quad{l_{n}}=0,\text{ for }n>2,$$ where $[-,-]$ is the obvious Lie bracket on $A_{DR}^{\bullet}(M_{0})\widehat{\otimes} \mathfrak{u}'$. In particular $\mathfrak{u}'$ and $A_{DR}^{\bullet}(M_{0})\widehat{\otimes} \mathfrak{u}'$ are complete Lie algebras and $ \mathfrak{u}'$ is pronilpotent.
	\end{enumerate}
\end{cor}
\begin{proof}
 ${\delta'}^{*}$ preserves the filtration given by the power of the augmented ideal $I$ in ${\mathbb{L}} \left( \left( W^{1}_{+}[1]\right)^{*}\right)$, since it satisfies the Leibniz rule and ${\delta'}^{*}I\subset I^{2}$ by ${m'}_{1}=0$. Hence $ \mathfrak{u}'$ can be written as a projective limit of finite dimensional nilpotent Lie algebras and the lie algebras $A_{DR}^{\bullet}(M_{0})\widehat{\otimes} \mathfrak{u}'$, $ \mathfrak{u}'$ are complete.
\end{proof}
For a Lie algebra $\mathfrak{u}'$ we define the adjoint action $\operatorname{ad}\: : \: \mathfrak{u}'\to \operatorname{End}\left(\mathfrak{u}' \right)$ via $\operatorname{ad}_{v}(w):=\left[v,w \right] $.
\begin{thm}\label{flatconnection on M0}
Let $M_{\bullet}$, $B$, $\overline{C}$ and $\mathfrak{u}'$ be as above. Consider the adjoint action $\operatorname{ad}$ of $\mathfrak{u}'$ on itself. Then $d-r_{*}\overline{C}$ defines a flat connection on the trivial bundle over $M_{0}$ with fiber $\mathfrak{u}'$, where the latter is considered to be equipped with the adjoint action. The connection is holomorphic (with logarithmic singularities) if so is $\overline{C}$.
\end{thm}
\begin{proof}
Since $r_{*}\overline{C}$ is a Maurer-Cartan element it defines a flat connection. 
\end{proof}
\begin{rmk}\label{Chenhomologicalpair}
 In \cite[Theorem 1.3.1]{extensionChen} Chen constructs a good homological pair $\left(C, \delta^{*} \right)$ on a complex manifold $M$ such that $d-\pi(C)$ is a flat connection on $M\times\mathfrak{u}$ and $d-\pi(C)$ satisfies Theorem \ref{Chenthm} (see  \cite[Theorem 2.1.1]{extensionChen}). By definition, such a homological pair corresponds to a minimal model $g_{\bullet}\: : \:\left( W, m_{\bullet}^{W}\right)\to  \left( A_{DR}(M), d, \wedge\right)$. In \cite{Huebsch}, it is shown that $g_{\bullet}$ is constructed via the homotopy transfer theorem. In \cite{Prelie}, it is shown that $g_{\bullet}$ has a quasi-inverse $f_{\bullet}$ (explicitly constructed) such that $f_{\bullet}g_{\bullet}=\mathrm{Id}_{W}$. 
\end{rmk}
\begin{defi}
Let $\overline{C}$ be a good degree zero geometric connection. We call $\mathfrak{u}=\widehat{\mathbb{L}} \left( \left( W^{1}_{+}[1]\right)^{*}\right)/\mathcal{R}_{0}$ the \emph{fiber Lie algebra of the simplicial manifold} $M_{\bullet}$. 
 \end{defi}	
 The next proposition shows that the fiber Lie algebra is well-defined.
\begin{prop}\label{fibermalcev} Let $M_{\bullet}$ and $N_{\bullet}$ be simplicial manifolds with finite type connected cohomology. Assume that there is a smooth map $p_{\bullet}\: : \: \:N_{\bullet}\to M_{\bullet}$ that induces a quasi-isomorphism in cohomology.  
\begin{enumerate}
\item The fiber Lie algebra of $M_{\bullet}$ is isomorphic to the fiber Lie algebra of $N_{\bullet}$. 
\item Let $G$ be a discrete group acting properly and discontinuously on a complex manifold $M$. In this case the quotient $M/G$ is again a complex manifold. Consider the action groupoid $M_{\bullet}G$. Assume that the cohomology of $M/G$ is of finite type. The fiber Lie algebra of $M_{\bullet }G$ is the Malcev Lie algebra of $\pi_{1}\left( M/G\right) $.
\end{enumerate}
In particular the fiber Lie algebra is independent (up to isomorphism) by the choice of good degree zero geometric connection. 
\end{prop}
\begin{proof}
We prove 1. We fix a good reduced homological pair $(C, {\delta}^{*})$ with coefficients in $\operatorname{Tot}_{N}\left( A_{DR}\left(M_{\bullet } \right)\right) $ and a good reduced homological pair $(C', {\delta'}^{*})$ with coefficients in $\operatorname{Tot}_{N}\left( A_{DR}\left(N_{\bullet} \right)\right) $. In particular, 
\[
p^{*}\: : \:\operatorname{Tot}_{N}\left( A_{DR}\left(N_{\bullet } \right)\right) \to \operatorname{Tot}_{N}\left( A_{DR}\left(M_{\bullet } \right)\right) 
\]
is a quasi-isomorphism and a strict $C_{\infty}$-map.
We get a diagram 
\[
\begin{tikzcd}
\operatorname{Tot}_{N}\left( A_{DR}\left(M_{\bullet } \right)\right)\arrow[r,  "p^{*}"]\arrow[d, , "{g}_{\bullet}"]
 &  \operatorname{Tot}_{N}\left( A_{DR}\left(N_{\bullet } \right)\right)\arrow[d, , "{g'}_{\bullet}"]\\(W ,m_{\bullet}^{W})\arrow[u, shift right, "f_{\bullet}"]& ( W' ,m_{\bullet}^{W'})\arrow[u, shift right, "{f'}_{\bullet}"].
\end{tikzcd}
\]
of $C_{\infty}$-quasi isomorphism. In particular the map $(g'pf)_{\bullet}$ is an isomorphism. It follows that the two fiber Lie algebras are isomorphic. Let $\overline{C}$ be a good degree zero geometric connection with coefficients in $\operatorname{Tot}_{N}\left( A_{DR}\left(M_{\bullet } \right)\right)$. Its fiber Lie algebra is isomorphic to the fiber Lie algebra of $\pi(C)$ via the isomorphism $K^{*}$ as constructed in \eqref{mapKdual}. Let $\overline{C}'$ be a good degree zero geometric connection with coefficients in $\operatorname{Tot}_{N}\left( A_{DR}\left(N_{\bullet } \right)\right)$, by the same reasoning we get that its fiber Lie algebra is isomorphic to the one of $\overline{C}$. We prove 2. We set $N_{\bullet}=M_{\bullet }G$ and $M_{\bullet}$ is the constant simplicial manifold $M/G$. There is a quasi-isomorphism
\[
i\: : \:A_{DR}\left(M/G \right) \to \operatorname{Tot}_{N}\left( A_{DR}\left(M_{\bullet } G\right)\right).
\]
We consider the good homological pair  $\left(C, \delta^{*} \right)$ on constructed by Chen as in Remark \ref{Chenhomologicalpair}. Then it defines a flat connection on $M/G$ whose fiber is the Malcev Lie algebra of $\pi_{1}\left( M/G\right) $. The statement follows from 1.
\end{proof}
Let $(N,\mathcal{D})$ be a complex manifold with a normal crossing divisor. Let $G$ be a discrete group acting smoothly on $N$ and that preserves $\mathcal{D}$. Then $(N-\mathcal{D})_{\bullet}G$ is a complex manifold equipped with a simplicial normal crossing divisor. 

\begin{cor}[\cite{Morgan}]\label{Morgancor}
Let $(N,\mathcal{D})$ and $G$ be as above. Assume that $N/G$ is a connected projective non-singular variety with a normal crossing divisor $\mathcal{D}/G$. Then the fiber Lie algebra $\mathfrak{u}$ of $(N-\mathcal{D})_{\bullet}G$ is given as follows: there exists a graded vector space $V=V^{1}\oplus V^{2}$ concentrated in degree $1$ and $2$ and a completed homogeneous ideal $\mathcal{J}\subset\widehat{\mathbb{L}\left(V\right)}$ where the degree of the generators is $2$, $3$ and $4$ such that
\[
\mathfrak{u}= \widehat{\mathbb{L}\left(V\right)}/\mathcal{J}
\]
\end{cor}
\begin{proof}
This is Corollary 10.3 in \cite{Morgan}.
\end{proof}

\section{Connections and bundles}\label{newadded}
Let $M$ be a complex manifold. Let $M_{\bullet}G$ be an action groupoid where $G$ is discrete and it acts properly and discontinuosly on $M$. Assume that the chomology of $M/G$ is connected and of finite type. The results of the previous section allow the construction of the desired flat connections on $M/G$.
\subsection{Gauge equivalences}\label{gauge equiv}
The next lemma is a standard exercise about filtered vector spaces.
\begin{lem}\label{lemmapronilp}
Let $\mathfrak{u}$ be a pronilpotent graded Lie algebra concentrated in degree $0$ and let $\left( A,d,\wedge\right) $ be a non-negatively graded differential graded algebra.
\begin{enumerate}
	\item The vector space $A\widehat{\otimes }\mathfrak{u}$ is a differential graded Lie algebra where the differential is given by the tensor product and the brackets are defined via
	\[
	[a\otimes v, b\otimes w]:=\pm \left( a\wedge b\right) \otimes [v,w],
	\]
	where the signs follow from the signs rule.
	\item The Lie algebra $A^{0}\widehat{\otimes }\mathfrak{u}$ is complete with respect to the filtration induced by $I$.
	\item Let $\Omega(1)$ be the differential graded algebra of polynomials forms on the interval $[0,1]$. Consider the differential graded Lie algebra $\Omega(1)\widehat{\otimes }\left( A\widehat{\otimes }\mathfrak{u}\right) $ obtained as in point 1. Then, there is a canonical isomorphism 
	\[
	\Omega(1)\widehat{\otimes }\left( A\widehat{\otimes }\mathfrak{u}\right) \cong \left(\Omega(1)\otimes  A\right)\widehat{\otimes }\mathfrak{u}.
	\]
\end{enumerate}
\end{lem}
Let $\mathfrak{u}$ be a pronilpotent differential graded Lie algebra. The pronilpotency guarantees that the (complete) universal enveloping algebra $\mathbb{U}\left( \mathfrak{u}^{0}\right) $ is a complete Hopf algebra and its group-like elements can be visualized as $U:=\exp\left(\mathfrak{u}^{0} \right) $. Assume that $\mathfrak{u}$ is concentrated in degree zero and that it is equipped with the trivial differential. Let $A$ be a differential graded commutative algebra an assume that it is non-negatively graded. We consider the complete differential graded Lie algebra $A\widehat{\otimes}\mathfrak{u}$. The universal enveloping algebra of $\mathfrak{u}$ is equipped with the filtration induced by $I^{\bullet}$. In particular $U$ is complete with respect to such a filtration and the exponential map can be extended to
\[
Id\widehat{\otimes }\exp\: : \: A^{0}\widehat{\otimes}\mathfrak{u}\to A^{0}\widehat{\otimes}U\subset A^{0}\widehat{\otimes}\mathbb{U}\left( \mathfrak{u}^{0}\right) 
\]
 By abuse of notation we denote the above map again by $\exp$. The complete tensor product gives to $ A^{0}\widehat{\otimes}\mathbb{U}\left( \mathfrak{u}^{0}\right) $ the structure of an associative algebra. The image of $Id\widehat{\otimes }\exp$ is again a group where the inverse of $e^{h}$ is given by $e^{-h}$ for a $u\in A^{0}\widehat{\otimes}\mathfrak{u}$. This group acts on the set of Maurer-Cartan elements $MC\left(A^{0}\widehat{\otimes}\mathfrak{u} \right) $ via 
\[
e^{h}(\alpha):=e^{\operatorname{Ad}_{h}}(\alpha)+\frac{1-e^{\operatorname{Ad}_{h}}}{\operatorname{Ad}_{h}}(dh).
\]
We call this action the \emph{gauge-action} and the above group the \emph{gauge group}. If $A=A_{DR}(M)$ for some complex manifold $M$, then $d-\alpha$ defines a flat connection on $M\times \mathfrak{u}$ and the action of the gauge group can be written as $e^{h}(\alpha)=e^{h}(d-\alpha)e^{-h}$.
\begin{defi}
Two Maurer-Cartan elements $\alpha_{0},\alpha_{1}$ are said to be \emph{gauge equivalent} in $A\widehat{\otimes}\mathfrak{h}$ if there is a $h$ in $A^{0}\widehat{\otimes}\mathfrak{u}$ such that $e^{h}(\alpha_{0})=\alpha_{1}$.
\end{defi}
\begin{prop}\label{equivMaurer}
Let $A\widehat{\otimes }\mathfrak{h}$ be as above. Two Maurer-Cartan elements are gauge equivalent if and only if they are homotopy equivalent in the sense of Definition \ref{Linftydef}.
\end{prop}
\begin{proof}
We have that $A\widehat{\otimes }\mathfrak{h}$ is a complete Lie algebra with respect to $I^{\bullet}$. The result follows from \cite[Section 2.3]{Yalin}.
\end{proof}

\subsection{Factors of automorphy}\label{factorofautomorphy}
We show how to construct a flat connection on the quotient. We use the same point of view of \cite{Hain}. Let $X$ be a set and $G$ a group acting on it from the left. Let $V$ be a vector space. A \emph{factor of automorphy} is a map $F\: : \:G\times X\to \operatorname{Aut}(V) $ such that $g\: : \: X\times V\to X\times V$ defined by
$
(x,v)\mapsto (gx,F_{g}(x)v)
$
gives a group action of $G$ on $X\times V$. This is equivalent to
\[
F_{gh}(x)(v)=F_{g}(hx)F_{h}(x).
\]
Let $M$ be a complex manifold equipped with a smooth and properly discontinuous action of a discrete group $G$ (from the left). Let $V$ be a vector space and let $F\: : \:G\times M\to \operatorname{Aut}(V) $ be a factor of automorphy. Then $F$ induces a $G$-action on $M\times V$. In particular, the quotient
\[
\left( M\times V\right) /G
\]
is a vector bundle on $M/G$ where the sections satisfy $s(gx)=F_{g}(x)s(x).$ We denote by $E_{F}$ the vector bundle induced by the factor of automorphy $F$.
\begin{prop}\label{welldefined}
Let $M$ be a complex manifold equipped with a smooth action of a group $G$. Let $V$ be a complete vector space and let $F\: : \:G\times M\to \operatorname{Aut}(V) $ be a factor of automorphy. Let $\alpha\in A^{1}_{DR}(M)\widehat{\otimes }\operatorname{End}(V)$ be a $1$-form with vaules in $\operatorname{End}(V)$. 
\begin{enumerate}
	\item The connection $d-\alpha$ induces well-defined connection on $E_{F}$ if
	\[
	d-g^{*}\alpha=F_{g}\left(d-\alpha \right)F^{-1}_{g}
	\]
	for any $g\in G$.
	\item For a smooth path $\gamma\: : \: [0,1]\to M$ we denote with $T(\gamma)$ the parallel transport along $\alpha$ of $d-\alpha$ cosnidered as a connection over $M\times V$. Assume that $\alpha$ is well-defined on $E_{F}$, Then
		\[
		T(g\gamma)={F}_{g}(\gamma(0))T(\gamma){F}_{g}(\gamma(1))^{-1}.
		\]	
\end{enumerate}
\end{prop}
\begin{proof}
The two statements are in \cite{Hain}( Proposition 5.1 and Proposition 5.7 respectively).
\end{proof}
Let $M$, $G$ be as above. We fix a pronilpotent Lie algebra $\mathfrak{u}$ concentrated in degree zero and we set $U=\exp(\mathfrak{u})$. Let $\alpha\in A^{1}_{DR}(M)\widehat{\otimes }\operatorname{End}(\mathfrak{u})$ be a well-defined connection form on a $E_{F}$. We fix a $p\in M$ and we denote its class in $M/G$ by $\overline{p}$. By covering space theory, this choice induces a group homomorphism
 \[
 \rho\: : \: \pi_{1}\left(M/G,\overline{p} \right) \to G,
 \]
 which is constructed via the homotopy lifting property. For a path starting at $\overline{p}$, we denote its (unique) lift starting at $p$ with $c_{\gamma}$. In particular, $c_{\gamma}(1)=\rho(\gamma)p$. Given two loops $\gamma_{1}, \gamma_{2}$ starting at $\overline{p}$ on $M/G$, we have $c_{\gamma_{1}\cdot \gamma_{2}}=c_{\gamma_{1}}\cdot\left(\rho(\gamma_{1}) \gamma_{2}\right) $ (see \cite{Hain}, Lemma 5.8). 
 \begin{prop}\label{holonomyrep}
 	Let $U=\exp(\mathfrak{u})$ and let $p$ be as above. The parallel transport $T$ of $\alpha$ on $M\times \mathfrak{u}$ induces a group homomorphism
 	\[
 	\Theta_{0}\: : \:  \pi_{1}\left(M/G,\overline{p} \right) \to U
 	\]
 	given by $\Theta_{0}\left(\gamma \right):=T(c_{\gamma})F_{\rho(\gamma)}(p)$.
 \end{prop}
\begin{proof}
See Proposition 5.9 in \cite{Hain}.
\end{proof}
We call $\Theta_{0}$ the\emph{ holonomy representation} of $\alpha$ on $M/G$. Consider the action groupoid $M_{\bullet}G$ and assume that the cohomology of $M/G$ is connected and of finite type. Let $$g_{\bullet}\: : \: \left(W, m_{\bullet}^{W} \right)\to \mathcal{F}\left( \operatorname{Tot}_{N}^{\bullet}\left(A_{DR}(M_{\bullet}G)\right) ,m_{\bullet} \right)$$
be a $1$-minimal model and let $\overline{C}$ be its associated  good degree zero geometric connection. By the results of the previous section, $\overline{C}$ induces a flat connection form $r^{*}\overline{C}\in A_{DR}^{1}(M)\widehat{\otimes}\mathfrak{u}$ such that $d-r^{*}\overline{C}$ is a flat connection on the trivial bundle on $M$ with fiber $\mathfrak{u}$. By repeating the same reasoning for another $1$-minimal model $${g'}_{\bullet}\: : \: \left(W', m_{\bullet}^{W'} \right)\to\mathcal{F}\left( \left(A_{DR}(M/G)\right) ,m_{\bullet} \right)\subset \mathcal{F}\left( \operatorname{Tot}_{N}^{\bullet}\left(A_{DR}(M_{\bullet}G)\right) ,m_{\bullet} \right)$$	
we get a good degree zero geometric connection $\overline{C}'$ and a flat connection form $r^{*}\overline{C}'\in A_{DR}^{1}(M)\widehat{\otimes}\mathfrak{u}'$ such that $d-r^{*}\overline{C}'$ is a flat connection on the trivial bundle on $M$ with fiber Lie algebra $\mathfrak{u}'$. 
Let ${g''}_{\bullet}\: : \: (W'',m_{\bullet}^{W''})\to \left( \operatorname{Tot}_{N}^{\bullet}\left(A_{DR}(M_{\bullet}G)\right) ,m_{\bullet} \right)$ be a $C_{\infty}$-algebra minimal model. Assume that there exist a quasi-isomorphism $f_{\bullet}\: : \:  \left(A_{DR}\left( M_{\bullet}G\right) ,m_{\bullet} \right)\to(W'',m_{\bullet}^{W''})$ such that $f_{\bullet}\circ g''_{\bullet}=\mathrm{Id}_{W}$. Let $\overline{C''}$ be its associated  good degree zero geometric connection. By the results of the previous section, $\overline{C''}$ induces a flat connection form $r^{*}\overline{C''}\in A_{DR}^{1}(M)\widehat{\otimes}\mathfrak{u}$ such that $d-r^{*}\overline{C''}$ is a flat connection on the trivial bundle on $M$ with fiber $\mathfrak{u''}$.
\begin{prop}\label{gaugeuniqueness}
Let $g_{\bullet}$ and ${g'}_{\bullet}$ be as above.
\begin{enumerate}
	\item There exists a morphism of Lie algebras $K^{*}\: : \: \mathfrak{u}'\to \mathfrak{u}$ 
	such that $$r_{*}k^{*}\overline{C}',r_{*}\overline{C}\in A_{DR}(M)\widehat{\otimes } \mathfrak{u}$$ are gauge equivalent. Assume that $M_{\bullet}G=(N-\mathcal{D})_{\bullet}G$, where $\left( N,\mathcal{D}\right) $ is a complex manifold with a normal crossing divisors and $G$ is a group acting holomorphically on $N$ and that preserves $\mathcal{D}$. If ${g'}_{\bullet}$ and $g_{\bullet}$ are holomorphic with logarithmic singularities, the gauge $e^{h'}$ is in $ \exp\left(A^{0}_{DR}(M)\widehat{\otimes }\mathfrak{u}\right)$.
	\item There exists a morphism of Lie algebras ${K''}^{*}\: : \: \mathfrak{u}''\to \mathfrak{u}$ 
		such that $$r_{*}{k''}^{*}\overline{C}'',r_{*}\overline{C}\in A_{DR}(M)\widehat{\otimes } \mathfrak{u}$$ are gauge equivalent. Moreover the morphism depends only by $\mathcal{F}\left( f_{\bullet}\right) $ and $g_{\bullet}$. Assume that $M_{\bullet}G=(N-\mathcal{D})_{\bullet}G$, where $\left( N,\mathcal{D}\right) $ is a complex manifold with a normal crossing divisors and $G$ is a group acting holomorphically on $N$ and that preserves $\mathcal{D}$.If ${g''}_{\bullet}$ and $g_{\bullet}$ are holomorphic with logarithmic singularities, the gauge $e^{h''}$ is in $ \exp\left(A^{0}_{DR}(M)\widehat{\otimes }\mathfrak{u}\right)$.	
\end{enumerate}
\end{prop}
\begin{proof}
The proof of 1. is an application of Proposition \ref{homotopyuniqueness}, Proposition \ref{pres} and Proposition \ref{equivMaurer}. The second part follows by the definition of $A_{DR}^{0}(\log D)$. We prove 2. We have a diagram of $1-C_{\infty}$-algebras
\[
\begin{tikzcd}
\mathcal{F}({W''}, m_{\bullet}^{W''})\arrow[rrr, shift right, "\mathcal{F}\left(  {g''}_{\bullet}\right) "']&&&\mathcal{F}\left( \operatorname{Tot}_{N}^{\bullet}\left(A_{DR}(M_{\bullet}G)\right) ,m_{\bullet} \right) 
\arrow[lll, shift right, " \mathcal{F}\left( f_{\bullet} \right) "'] \\ \\ \\
({W}, m_{\bullet}^{W})\arrow[uuu, shift right, "{k''}_{\bullet}"]\arrow[uuurrr, shift right, "g_{\bullet}"']
\end{tikzcd}
\]
where ${k''}_{\bullet}=\mathcal{F}(f_{\bullet})g_{\bullet}$. Notice that ${g''}_{\bullet}f_{\bullet}$ is the identity in cohomology since
\[
[{g''}_{\bullet}]^{-1}=[{f}_{\bullet}][{g''}_{\bullet}][{g''}_{\bullet}]^{-1}=[{f}_{\bullet}].
\]
Then there exist an inverse $t_{\bullet}\: : \: \left( \operatorname{Tot}_{N}^{\bullet}\left(A_{DR}(M_{\bullet}G)\right) ,m_{\bullet} \right)\to \left( \operatorname{Tot}_{N}^{\bullet}\left(A_{DR}(M_{\bullet}G)\right) ,m_{\bullet} \right)$ such that $t_{\bullet}{g''}_{\bullet}f_{\bullet}$ is homotopic to the identity via an homotopy $H_{\bullet}$. In particular 
${g''}_{\bullet}f_{\bullet}$ is homotopic to $t_{\bullet}{g''}_{\bullet}$ via $H_{\bullet}{g''}_{\bullet}f_{\bullet}$.
Let $\overline{C}_{1}$, $\overline{C}_{2}$ be the Maurer-Cartan elements corresponding to $\mathcal{F}({g''}_{\bullet}f_{\bullet})g_{\bullet}$ and $\mathcal{F}(t_{\bullet}{g''}_{\bullet}f_{\bullet})g_{\bullet}$. By Proposition \ref{homotopyuniqueness}, Proposition \ref{pres} and Proposition \ref{equivMaurer} we get that $r_{*}\overline{C}_{1}$ is gauge equivalent to $r_*\overline{C}_{2}$ which is gauge equivalent to $r_{*}\overline{C}$. By Baker-Campbell-Hausdorff formula $BCH(-,-)$ on $\mathfrak{u}$ we conclude that $r_{*}\overline{C}_{1}$ is gauge equivalent to $r_{*}\overline{C}$. We have $r_{*}\overline{C}_{1}=r_{*}\left(\mathrm{Id}\otimes {K''}^{*} \right) \overline{C}''=r_{*}{k''}^{*}\overline{C}''$ where ${K''}^{*}$ is the isomorphism of Lie algebras constructed in Proposition \ref{beh} induced by $k_{\bullet}=\mathcal{F}(f_{\bullet})g_{\bullet}$. This conclude the proof.
\end{proof}

\begin{rmk}\label{explicitcalculation} If ${g''}_{\bullet}$ is a minimal model constructed via the homotopy transfer theorem (for instance the $1$-minimal model constructed in Remark \ref{Chenhomologicalpair}). In \cite{Prelie}, Theorem 5 there is an explicit formula for $ f_{\bullet}$.  
\end{rmk}

\begin{prop}\label{fiber}
Let $M$, $G$ be as above. We fix a pronilpotent Lie algebra $\mathfrak{u}$ concentrated in degree zero. Let $\alpha'$ be a Maurer-Cartan elements in $A_{DR}(M)\widehat{\otimes} \mathfrak{u}$ such that it defines a well-defined flat connection on the bundle $E_{F'}$ with fiber $\mathfrak{u}'$ and factor of automorphy ${F'}_{g}(p)=Id$. Let $\alpha$ be a Maurer-Cartan element in $A_{DR}(M)\widehat{\otimes} \mathfrak{u}$. Assume that it is gauge equivalent to $\alpha'$ via some $h\in A^{0}_{DR}(M)\widehat{\otimes} \mathfrak{u}$. Then $\alpha$ is a well-defined connection form on the bundle $E_{F}$, where $F$ is given by ${F}_{g}(p):=e^{-h(gp)}e^{h(p)}$. Assume that there exists a finite dimensional positively graded vector space $V$ and a ideal $\mathcal{J}\subset \mathbb{L}\left(V \right)$ such that $\mathfrak{u}=\mathbb{L}\left(V \right)/\mathcal{J}$. Then if $\alpha$ and the group action are holomorphic, so is ${F}_{g}(p)$.
\end{prop}
\begin{proof}
${F}_{g}(p)$ is clearly a factor of automorphy. We have
	\begin{align*}
	e^{-g^{*}h}e^{h}\left( d-\alpha\right) e^{-h}e^{g^{*}h}& =	e^{-g^{*}h}\left( d-\alpha'\right) e^{g^{*}h} \\
	&=e^{-g^{*}h}\left( d-g^{*}\alpha'\right) e^{g^{*}h} \\
	&=d-g^{*}\alpha.
	\end{align*}
	We prove the second part. By Baker-Campbell-Hausdorff formula $BCH(-,-)$ on $\mathfrak{u}$ we can define $k(p):= BCH(-g^{*}h(p),h(p))$. In particular, $k\in A^{0}_{DR}(M)\widehat{\otimes} \mathfrak{u}$.
	We assume that $\alpha$ is holomorphic. Since 
		\begin{align*}
		d-g^{*}\alpha & =e^{-g^{*}h}e^{h}\left( d-\alpha\right) e^{-h}e^{g^{*}h}\\
		& =		e^{k}\left( d-\alpha\right) e^{k} \\
		&=e^{\operatorname{Ad}_{k}}(\alpha)+\frac{1-e^{\operatorname{Ad}_{k}}}{\operatorname{Ad}_{k}}(dk) \\
		\end{align*}
		and $dk=\sum_{i}\left( \frac{\partial}{\partial z_{i}}kd{z}_{i}+\frac{\partial}{\partial \overline{z}_{i}}kd\overline{z}_{i}\right)$, we set $\overline{\partial}k:=\sum_{i}\left( \frac{\partial}{\partial \overline{z}_{i}}kd\overline{z}_{i}\right)$. Let $\left\lbrace \mathcal{I}^{i}\right\rbrace $ be the filtration induced by the degree on $\widehat{\mathbb{L}}\left(V \right)$. Let $\underline{k}\in A^{0}_{DR}(M)\widehat{\otimes} \widehat{\mathbb{L}}\left(V \right)$ be a preimage of $k$ under the projection $A^{0}_{DR}(M)\widehat{\otimes}\widehat{\mathbb{L}}\left(V \right)\to  \mathfrak{u}A^{0}_{DR}(M)\widehat{\otimes} \mathfrak{u}$.  By above, we have $\frac{1-e^{\operatorname{Ad}_{\underline{k}}}}{\operatorname{Ad}_{\underline{k}}}\overline{\partial}\underline{k}\in \mathcal{J}$. The map $\frac{1-e^{\operatorname{Ad}_{\underline{k}}}}{\operatorname{Ad}_{\underline{k}}}\overline{\partial}$ defines an automorphisms on $A^{0}_{DR}(M)\widehat{\otimes}\widehat{\mathbb{L}}\left(V \right)$ and its inverses preserves $\mathcal{J}$. In particular, $\overline{\partial}\underline{k}\in \mathcal{J}$. We write $\underline{k}=\sum_{i=1}^{\infty}\underline{k}_{i}$ where each $\underline{k}_{i}$ is has homogeneous degree in $\widehat{\mathbb{L}}\left(V \right)$. Then $\overline{\partial}\underline{k}_{1}=\mathcal{J}\cap\mathcal{I}^{1}$ and
		 \[
		 \overline{\partial}\underline{k}_{1}=k_{1}^{\operatorname{hol}}+{k'}_{1}
		 \]
		where ${k'}_{1}\in   A^{0}_{DR}(M)\widehat{\otimes} \mathcal{J}$ and $k_{1}^{\operatorname{hol}}\in \widehat{\mathbb{L}}\left(V \right)$ is a primitive element whose coefficients are holomorphic functions. We set $\underline{k}^{\operatorname{hol},1}:=k_{1}^{\operatorname{hol}}+\underline{k}_{2}+\dots$. Continuing in this way we get a
		$\underline{k}^{\operatorname{hol}}:= \underline{k}^{\operatorname{hol},\infty}\in A^{0}_{DR}(M)\widehat{\otimes}\widehat{\mathbb{L}}\left(V \right)$ whose coefficients are holomorphic functions. Let $k^{\operatorname{hol}}$ be its projection on $A^{0}_{DR}(M)\widehat{\otimes}\mathfrak{u}$. The gauge action of $k^{\operatorname{hol}}$ corresponds to the gauge action of $k$ and hence
	\[
	{F}_{g}(p):=e^{-h(gp)}e^{h(p)}=e^{k(p)}=e^{k^{\operatorname{hol}}(p)}\in \operatorname{Aut}\left( \mathfrak{u}\right) 
	\]
	\end{proof}
Notice that there is canonical bundle isomorphism $E_{F'}\to E_{F}$ defined by 
\begin{equation}\label{bundlemap}
(p,v)\mapsto (p,e^{{-h(gp)}}e^{{h(p)}}v).
\end{equation}

\begin{defi}
Let $M$ be a complex manifold and let $E$, $E'$ be two smooth  vector bundles on $M$. Let $\left(d-\alpha,E \right) $, $\left(d-\alpha',E' \right) $ be two smooth connections. They are \emph{isomorphic} if there exists a (holomorphic) bundle isomorphism $T\: : \: E\to E'$ such that 
\[
T\left( d-\alpha\right)T^{-1}=d-\alpha'.
\]
Assume that $E$, $E'$ are bundles whose fiber is a pronilpotent Lie algebra $\mathfrak{u}$ and that $\alpha,\alpha'\in A^{1}_{DR}(M)\widehat{\otimes }\mathfrak{u}$ (we consider to be equipped with the adjoint action). An isomorphism $T\left(d-\alpha,E \right) \to\left(d-\alpha',E' \right) $ is said to be \emph{induced by a (holomorphic) gauge} if there exists a (holomorphic) $h\in A^{0}_{DR}(M)\widehat{\otimes }\mathfrak{u} $ such that $T=e^{h}$. 
\end{defi}
Let ${g'}_{\bullet}\: : \: \left(W', m_{\bullet}^{W'} \right) \to A_{DR}(M/G)$ be the minimal model (in the sense of $C_{\infty}$-algebras) obtained by the good homological pair $(C', { \delta'}^{*})$ and constructed by Chen as in Remark \ref{Chenhomologicalpair}. In particular it has a quasi-inverse $f_{\bullet}$ such that $f_{\bullet}{g'}_{\bullet}=\mathrm{Id}_{W}$. Let $\pi(C')$ be the good degree zero geometric connection associated to ${g'}_{\bullet}$ and let $\mathfrak{u}'$ be its fiber Lie algebra. Let $f_{\bullet}\: : \: \left(W, m_{\bullet}^{W} \right)\to \mathcal{F}\left( \operatorname{Tot}_{N}^{\bullet}\left(A_{DR}(M_{\bullet}G)\right) ,m_{\bullet} \right)$	
be a $1$-minimal model and let $\overline{C}$ be its associated  good degree zero geometric connection with fiber Lie algebra $\mathfrak{u}$. Consider the situation of Proposition \ref{gaugeuniqueness}, then there is an isomorphism of Lie algebra $K^{*}\: : \: \mathfrak{u}'\to \mathfrak{u}$, a Maurer-Cartan element $C\in A_{DR}(M)\widehat{ \otimes}\mathfrak{u}$ and a gauge $h\in A^{0}_{DR}(M)\widehat{\otimes }\mathfrak{u}$ such that $e^{h}\left(d- r_{*}\overline{C}\right)e^{-h}= k^{*}\pi{C'}$.
\begin{thm}\label{wehaveabundle}
 Let $\pi(C')$ be the good degree zero geometric connection associated to $g_{\bullet}$ and let $\mathfrak{u}'$ be its fiber Lie algebra. Assume that there is a triple $(\tilde{C},K*, h)$ consisting in an isomorphism of Lie algebra $K^{*}\: : \: \mathfrak{u}'\to \mathfrak{u}$, a Maurer-Cartan element $\tilde{C}\in A_{DR}(M)\widehat{ \otimes}\mathfrak{u}$ and a gauge $h\in A^{0}_{DR}(M)\widehat{\otimes }\mathfrak{u}$ such that $e^{h}\left(d- \tilde{C}\right)e^{-h}= k^{*}\pi{C'}$ as Maurer-Cartan elements in $A_{DR}(M)\widehat{ \otimes}\mathfrak{u}$.
\begin{enumerate}
\item $h$ induces a smooth factor of automorphy $F_{g}(p):=e^{-h(gp)}e^{h(p)}$ such that $d- \tilde{C}$ is a well-defined flat connection on $M/G$ on the bundle $E_{F}$ where the fiber corresponds to the Malcev completion of $\pi_{1}(M/G)$.
\item Assume that $\tilde{C}$ has holomorphic coefficients and the group action on $M$ is holomorphic. Then $h$ and $F$ are holomorphic and $d- \tilde{C}$ is a well-defined holomorphic flat connection on $M/G$ on the holomorphic bundle $E_{F}$.
\item Assume that $M_{\bullet}G=(N-\mathcal{D})_{\bullet}G$, where $\left( N,\mathcal{D}\right) $ is a complex manifold with a normal crossing divisors and $G$ be a group acting holomorphically on $N$ and that preserves $\mathcal{D}$. Assume that $\tilde{C}$ is holomorphic with logarithmic singularities. Then $F\: : \:N\times \mathfrak{u}\to N\times \mathfrak{u}$ is holomorphic and $d- \tilde{C}$ is a well-defined holomorphic flat connection with logarithmic singularities on $M/G$ on the holomorphic bundle $E_{F}$.
\end{enumerate}
In particular, any $1$-minimal model ${g}_{\bullet}$ for $\mathcal{F}\left( \operatorname{Tot}_{N}^{\bullet}\left(A_{DR}(M_{\bullet}G)\right) ,m_{\bullet}\right) $ produces a triple $(\tilde{C},K^*, h)$ where $\tilde{C}=r_{*}\overline{C}$, where $\overline{C}$ is the geometric degree zero connection associated to $g_{\bullet}$ and $K^{*}$ is completely determined by $\mathcal{F}(f_{\bullet})g_{\bullet}$.   
\end{thm}
\begin{proof}
 By Proposition \ref{fiber} and Proposition \ref{fibermalcev} and we get a flat connection $\left( d- r_{*}\tilde{C}, E_{F}\right) $ whose fiber $\mathfrak{u}$ is the Malcev Lie algebra of $\pi_{1}(M/G)$ and $\left( d-k^{*}\pi{C'}, M\times \mathfrak{u}\right)$ is isomorphic via a gauge to $\left( d- r_{*}\tilde{C}, E_{F}\right) $. We calculate its monodromy representation. We consider $d-k^{*}\pi{C'}$ as flat connection on $M\times \mathfrak{u}$. Its parallel transport $T^{K}$ is defined via iterated integrals (see \cite{IteratedChen}) and gives a map from the path space $PM$ of $M$ to $\widehat{T}\left( {W'}^{1}_{+} [1]\right) ^{*}/{\bar{\mathcal{R}}'}_{0}$, i.e the complete universal enveloping algebra of $\mathfrak{u}'$. In particular, we have $T^{K}(\gamma)=K^{*}T(\gamma)$ where $\gamma\in PM$ and $T'$ is the parallel transport of $d-\pi{C'}$ considered as a flat connection on $M\times \mathfrak{u}$. Let $T$ be the parallel transport of $d-r_{*}\tilde{C}$ considered as a flat connection on $M\times \mathfrak{u}$. Since they are gauge equivalent via $h$ we have
  \[
 T(\gamma)=e^{-h(\gamma(0))} T^{K}(\gamma)e^{h(\gamma(1))}\in \widehat{T}\left( {W}^{1}_{+}[1]\right) ^{*}/{\bar{\mathcal{R}}}_{0}
  \]
 for a path $\gamma\: : \: [0,1]\to M$. We fix a $p\in M$ and we choice a representative $\overline{p}$ for its class in $M/G$. We denote with $ \rho\: : \: \pi_{1}\left(M/G,\overline{p} \right) \to G$ its induce group homomorphism. Let $\gamma$ be a loop based at $\overline{p}$ and let $c_{\gamma}$ be its unique lift starting at $p$. We denote with ${\Theta}_{0}$ the holonomy representation of $d-r_{*}\tilde{C}$ and by ${\Theta'}_{0}$ the holonomy representation of $d-\pi{C'}$, then
\begin{align*}
 {\Theta'}_{0}(\gamma) & = T^{K}(c_{\gamma})e^{-h(\rho(\gamma)p)}e^{h(p)}\\
  & = e^{-h(c_{\gamma}(0))}T(c_{\gamma})e^{h(c_{\gamma}(1))}e^{-h(\rho(\gamma)p)}e^{h(p)}\\
  & = e^{-h(p)}T(c_{\gamma})e^{h(p)}\\
  & = e^{-h(p)}{\Theta}_{0}(\gamma)e^{h(p)}.
  \end{align*}
   Let $U=\exp\left(\mathfrak{u}\right) $. Since $K^{*}$ is an Hopf algebra isomorphism, it preserves group-like elements. The above calculation shows that the monodromy representation $\Theta_{0}\: : \:  \pi_{1}\left(M/G,\overline{p} \right) \to U$ is conjugate with $K^{*}{\Theta'}_{0}$, in particular it corresponds to the Malcev completion of the fundamental group. The second statement follows by Proposition \eqref{fiber}. We prove the last statement. By Proposition \eqref{fiber} we have a well-defined map $F_{g}(p)=e^{-h(gp)}e^{h(p)}$ where $h$ is holomorphic on $N-\mathcal{D}$. By Proposition \ref{gaugeuniqueness}, we have that $h$ is smooth on $N$. Hence by Cauchy's formula in several complex variables we conclude that $h$ is holomorphic on $N$ and then $F_{g}(p)=e^{-h(gp)}e^{h(p)}$ is the factor of automorphy of a holomorphic bundle on $N/G$. The final part follows directly by Proposition \ref{gaugeuniqueness}.
    \end{proof} 
    \begin{rmk}\label{FinalRMK0}
   Given a $1$-minimal model ${g}_{\bullet}$ for $\mathcal{F}\left( \operatorname{Tot}_{N}^{\bullet}\left(A_{DR}(M_{\bullet}G)\right) ,m_{\bullet}\right) $ and let $\tilde{C}$ be its degree zero geometric connection. Then in general there may be different choices for $K^{*}$ and $h$.
    \end{rmk}
Let $M$, $G$ be as above and let $\mathfrak{u}$ be a pronilpotent Lie algebra concentrated in degree zero. Let $\alpha\in A^{1}_{DR}(M)\widehat{\otimes} \mathfrak{u}$ such that $d-\alpha$ be a well-defined connection on the bundle $E_{F}$, with fiber $\mathfrak{u}$. Let $K^{*}\: : \: \mathfrak{u}\to \mathfrak{u}'$ be a Lie algebra isomorphism. Then $d-\left(\mathrm{Id}\widehat{\otimes }K\right)\alpha=d-k^{*}\alpha$ is a well-defined connection on the bundle $E_{K^{*}F}$ with fiber $\mathfrak{u}'$.
\begin{thm}\label{wehaveabundleuniq}
 Let $\tilde{C}_{1}$ and $\tilde{C}_{2}$ be two Maurer-Cartan elements satisfying the conditions of Theorem  \ref{wehaveabundle} with factor of automorphy $F_{1}$ and $F_{2}$. Let $\mathfrak{u}_{1}$ and $\mathfrak{u}_{2}$ be the fiber Lie algebra of $\tilde{C}_{1}$ and $\tilde{C}_{2}$ respectively. 
\begin{enumerate}
\item There exists a Lie algebra automorphism $K^{*}\: : \: \mathfrak{u}_{1}\to \mathfrak{u}_{2}$ such that $\left( d- k^{*} \tilde{C}_{2}, E_{K^{*}F_{2}}\right) $ is isomorphic via a gauge to $\left( d- \tilde{C}_{1}, E_{F_{1}}\right)$ as a flat connection on $M/G$.
\item Assume that $\tilde{C}_{1}$ and $\tilde{C}_{2}$ have holomorphic coefficients. The isomorphism is induced via a holomorphic gauge $e^{h}$, where $h\in A_{DR}(M)\widehat{ \otimes}\mathfrak{u}$ is a function with holomorphic coefficients.
\item Assume that $M_{\bullet}G=(N-\mathcal{D})_{\bullet}G$, where $\left( N,\mathcal{D}\right) $ is a complex manifold with a normal crossing divisors and $G$ is a group acting holomorphically on $N$ and that preserves $\mathcal{D}$. Assume that $\tilde{C}_{1}$ and $\tilde{C}_{2}$ have holomorphic coefficients with logarithmic singularities. The isomorphism is induced via a holomorphic gauge $e^{h}$, where $h\in A_{DR}(N)\widehat{ \otimes}\mathfrak{u}$ is a function with holomorphic coefficients. 
\end{enumerate}
\end{thm}
\begin{proof}
Let $(C', { \delta'}^{*})$ be the good homological pair constructed by Chen as in Remark \ref{Chenhomologicalpair}. By Theorem \ref{wehaveabundle}, we have the following.
\begin{enumerate}
\item There is a Lie algebra isomorphism $\left( L'\right)^{*}\: : \: \mathfrak{u}'\to \mathfrak{u}_{1}$ such that $\left( d-(l')^{*}\pi{C'}, M\times \mathfrak{u}_{1}\right) $ and $\left( d- r_{*}\tilde{C}_{1}, E_{F_{1}}\right)$ are isomorphic via a gauge.
\item There is a Lie algebra isomorphism $\left( L\right)^{*}\: : \: \mathfrak{u}'\to \mathfrak{u}_{2}$ such that $\left( d-l^{*}\pi{C'}, M\times \mathfrak{u}_{2}\right) $ and $\left( d- r_{*}\tilde{C}_{2}, E_{F_{2}}\right)$ are isomoprhic via a gauge.
\end{enumerate}
 An explicit calculation shows that $$\left(d-\left({l'}^{-1}l \right)^{*} \tilde{C}_{2} , E_{\left({L'}^{-1}L \right)^{*}F_{2}} \right)\text{ and }\left(d- r_{*}\tilde{C}_{1} , E_{F_{1}} \right)$$ are isomorphic via a gauge to $\left(d- r_{*}l^{*} \tilde{C}' , M\times \mathfrak{u}' \right)$. This prove point 1. We prove point 2. Assume that $\tilde{C}_{1}$ and $\tilde{C}_{2}$ have holomorphic coefficients. Then so are $\tilde{C}_{1} $ and $\left({l'}^{-1}l \right)^{*} \tilde{C}_{2}$. These two connection forms are gauge equivalent as Maurer-Cartan elements in $A^{0}_{DR}(M)\widehat{\otimes }\mathfrak{u}_{1}$ for a certain $h\in A^{0}_{DR}(M)\widehat{\otimes }\mathfrak{u}_{1}$. The same argument used in Proposition \ref{fiber} shows that the gauge is holomorphic. Point 3 follows by point 2 and Cauchy's formulas for holomorphic functions in several variables.
\end{proof}

\subsection{Formality of the fundamental group and Hopf algebra isomorphisms}
In \cite{Sibilia1}, we show that the KZB connection presented in \cite{Damien} on the configuration space of points on the punctured torus can be constructed via Theorem \ref{wehaveabundle}. More precisely, we construct a
a $1$-minimal model via the homotopy transfer theorem which gives degree zero geometric connection $\overline{C}$ such that $d-r_{*}\overline{C}$ is the KZB connection on the configuration space of points on the punctured torus. The monodromy of this connection is used in \cite{Damien} to shows the formality of the pure braids group on the torus.
This fact can be proved via Chen's theory.  Let $\mathfrak{u}$ be a pronilpotent Lie algebra concentrated in degree zero. We denote its associated graded (with respect to the filtration $I^{\bullet}$) by $\operatorname{gr}\left( \mathfrak{u}\right)=\oplus_{i\geq 0}I^{i}/ I^{i+1}$ and the complete associated graded by $\widehat{\operatorname{gr}}\left( \mathfrak{u}\right)=\widehat{\oplus}_{i\geq 0}I^{i}/ I^{i+1}$. Notice that they are both filtered.
\begin{defi}[\cite{Damien}]
	The Lie algebra $\mathfrak{u}$ is said to be \emph{formal} if there is an isomorphism of filtered Lie algebras $\mathfrak{u}\to \widehat{\operatorname{gr}}\left( \mathfrak{u}\right)$ whose associated graded is the identity. The group $U=\exp(\mathfrak{u})$ is \emph{formal} if so is $\mathfrak{u}$. A group $G$ is \emph{formal} if so is its Malcev completion.
\end{defi}
In particular, if there exists a positively graded Lie algebra $\mathfrak{t}$ and an isomorphism of filtered Lie algebras  $\mathfrak{u}\to \widehat{ \mathfrak{t}}$ (called \emph{formality isomorphism}), then $\mathfrak{u}$ is formal and the map induced via its associated graded is an isomorhism of graded Lie algebra $\operatorname{gr}\left( \mathfrak{u}\right)\to \mathfrak{t}$.\\
Let $M$, $G$ be as in Theorem \ref{wehaveabundle}. We fix a $p\in M$ and we choice a representative $\overline{p}$ for its class in $M/G$. Let $\overline{C}$ be a good degree zero geometric connection.
The monodromy representation of $d-r_{*}\overline{C}$ gives a group homomorphism
 \[
\Theta_{0}\: : \:  \pi_{1}\left(M/G,\overline{p} \right)  \to U\subset \widehat{T}\left(  W_{+}^{1}[1]\right). ^{*}/\bar{\mathcal{R}}_{0}
 \]
We consider the Hopf algebra $\C\left[ \pi_{1}\left(M/G,\overline{p} \right) \right]$. The above map can be extended into an Hopf-algebra morphism
\begin{equation}\label{Malcevcompletionintro}
\Theta_{0}\: : \: \C\left[  \pi_{1}\left(M/G,\overline{p} \right) \right] \to H^{0}\left( {T}\left(W[1]^{*}\right), \delta^{*}\right) 
\end{equation} 
Let $J$ be the kernel of the augmentation map $\C\left[  \pi_{1}\left(M/G,\overline{p} \right) \right]\to \C$ that sends each element of $\pi_{1}(M,p)$ to $1$. The powers of $J$ (with respect to the multiplication) define a filtration $J^{i}$ and the completion $\C\left[ \pi_{1}\left(M/G,\overline{p} \right) \right]^{\wedge }$ is a complete Hopf algebra. Theorem 2.1.1 in \cite{extensionChen}  and Theorem \ref{wehaveabundle} implies the following\footnote{ Theorem 2.1.1 is not written in this way, but this fact is implies by the last isomorphism at page 209 in loc. cit.}: the map $\Theta_{0}$ preserves $J^{i}$ and it is the $J$-adic completion of $\C\left[\pi_{1}\left(M/G,\overline{p} \right) \right]$ as a Hopf algebra. Hence by looking at the group-like elements we conclude that $\Theta_{0}$ gives the Malcev completion of $\pi_{1}\left(M/G,\overline{p} \right) $. The following is immediate.
\begin{prop}\label{formality}
	 Assume that ${\mathcal{R}}_{0}$ is homogeneous, then $\pi\left(M/G, \overline{p} \right)$ is formal.
\end{prop}
In particular Chen'theory gives an explicit formality isomorphism by taking the $\log$ of the monodromy representation $\Theta_{0}$. The formality of the fundamental group of $M=\left( N-\mathcal{D}\right)/G$ where $\mathcal{D}$ is a normal crossing divisor in a smooth complex algebraic variety $N$ and $G$ is a discrete group acting smoothly on $N$ and preserving $\mathcal{D}$ follows from Corollary \ref{Morgancor} and this shows the formality of the braid group on the torus, which corresponds to the fundamental group of the configuration space of points of the punctured torus. Let $M$ be a complex manifold equipped with a smooth properly discontinuous action of a group $G$, we assume that the quotient is connected an that the cohomology is of finite type. Let $f_{\bullet}\: : \:\left( W, m_{\bullet}^{W}\right)\to   \mathcal{F}\left( \operatorname{Tot}_{N}\left(A_{DR}(M_{\bullet}G)\right), m_{\bullet}\right)$ be a $1$-minimal model. The homogeneity of ${\mathcal{R}}_{0}$ can translated in terms of $m^{W}_{\bullet}$. Let $V_{n}\subset W^{2}$ be the subspace generated by $m_{n}^{W}(w_{1}, \dots,w_{n} )$ where $w_{i}\in W$ for $i=1, \dots, n$, then ${\mathcal{R}}_{0}$ has homogeneous generators if $\bigcap_{i=1}^{n}(V_{i})=0$.
\begin{rmk}
There is a topological interpretation of the $m^{W}_{n}$. Consider the singular cohomology $C^{\bullet}\left( M/G\right)$, in particular the cup products gives a differential graded algebra structure on $C^{\bullet}\left( M/G\right)$. In \cite{Camillo}, it is constructed an $A_{\infty}$-quasi-isomorphism between $C^{\bullet}\left( M/G\right)$ and $A_{DR}(M/G)$. Since $\left( A_{DR}(M/G),d,\wedge\right) \subset \left( \operatorname{Tot}_{N}\left(A_{DR}(M_{\bullet}G)\right), m_{\bullet}\right)$ is a strict $C_{\infty}$-quasi-isomorphism, it follows that $m^{W}_{n}$ corresponds to the higher Massey products in singular cohomology of $M/G$.
\end{rmk} 
 In a forthcoming paper, we analyze the case where the action of $G$ is not properly discontinuous and we give sufficient condition to have a rational Malcev completion.

\section{Appendix}
\subsection{Proof of Lemma \ref{Algebrainfinitystruct} and corollary \eqref{Linftystructnotquot}}\label{proof}
 We start with the proof of Lemma \ref{Algebrainfinitystruct}.
 \begin{proof}
 Clearly $\partial^{2}=0$ and $M_{1}^{2}=0$. The relations \eqref{rel} for $n>1$ are equivalent to
 \begin{equation}\label{rel2'}
 \sum_{\substack{p+q+r=n\\
 p+1+r>1, q>1}}(-1)^{p+qr}m_{p+1+r}\circ \left(Id^{\otimes p}\otimes m_{q}\otimes Id^{\otimes r} \right)= D\circ m_{n}, 
 \end{equation}
 where $D\circ m_{n}:=m_{1}\circ m_{n}+(-1)^{n+1}m_{n}\circ \left(\sum_{i=1}^{n-1} Id^{\otimes i}\otimes m_{1}\otimes Id^{\otimes n-1-i} \right)$. Chose homegeneous elements $f_{1}, \dots, f_{n}$. The expression \eqref{rel2'} in our case is
 \begin{align*}
 &\sum_{\substack{p+q+r=n\\
 p+1+r>1, q>1}}(-1)^{p+qr}M_{p+1+r}\circ \left(Id^{\otimes p}\otimes M_{q}\otimes Id^{\otimes r} \right)(f_{1}\otimes  \cdots\otimes  f_{n})  =\\
 & \sum_{\substack{p+q+r=n\\
 p+1+r>1, q>1}}(-1)^{p+qr}\tilde{m}^{A}_{p+1+r} \left( \left(Id^{\otimes p}\otimes M_{q}\otimes Id^{\otimes r} \right)\circ (f_{1}\otimes  \cdots\otimes  f_{n})\right)\circ \Delta^{p+r} =\\
  &\pm \sum_{\substack{p+q+r=n\\
  p+1+r>1, q>1}}(-1)^{p+qr}\tilde{m}^{A}_{p+1+r} \left(f_{1}\otimes  \cdots \otimes  f_{p}\otimes \left( \tilde{m}^{A}_{q}\left(f_{p+1},\dots ,f_{p+q} \right)\circ \Delta^{q-1} \right) \otimes f_{p+q+1}\otimes \cdots \otimes f_{n} \right)\circ \Delta^{p+r} =\\
 &\pm \sum_{\substack{p+q+r=n\\
 p+1+r>1, q>1}}(-1)^{p+qr}\tilde{m}^{A}_{p+1+r} \left(f_{1}\otimes  \cdots \otimes  f_{p}\otimes  \tilde{m}^{A}_{q}\left(f_{p+1},\dots ,f_{p+q} \right)\otimes f_{p+q+1}\otimes \cdots \otimes f_{n} \right)\circ \Delta^{n-1} =\\
 & \sum_{\substack{p+q+r=n\\
 p+1+r>1, q>1}}(-1)^{p+qr}\tilde{m}^{A}_{p+1+r} \left(Id^{\otimes p}\otimes  \tilde{m}^{A}_{q}\otimes Id^{\otimes r} \right)\circ (f_{1}\otimes  \cdots\otimes  f_{n}) \circ \Delta^{n-1}.\\
 \end{align*}
 Set $m_{1}=\partial$, we have
 \begin{align*}
 \left( \partial\circ M_{n}+(-1)^{n+1}M_{n}\circ \left(\sum_{i=1}^{n-1} Id^{\otimes i}\otimes \partial\otimes Id^{\otimes n-1-i} \right)\right)(f_{1}\otimes  \cdots\otimes  f_{n}). \\
 \end{align*}
 The first summand is
 \begin{align*}
 & \partial\circ M_{n}(f_{1}\otimes  \cdots\otimes  f_{n})\\
 &= \tilde{m}^{A}_{1}\tilde{m}^{A}_{n}(f_{1}\otimes  \cdots\otimes  f_{n})\circ\Delta^{n-1}+  (-1)^{|M_{n}(f_{1}\otimes  \cdots\otimes  f_{n})|+1}\tilde{m}^{A}_{n}(f_{1}\otimes  \cdots\otimes  f_{n})\circ\Delta^{n-1}\circ\delta,
 \end{align*}
 the second is 
 \begin{align*}
 & (-1)^{n+1}M_{n}\circ \left(\sum_{i=1}^{n-1} Id^{\otimes i}\otimes \partial\otimes Id^{\otimes n-1-i} \right)(f_{1}\otimes  \cdots\otimes  f_{n})=\\
 & (-1)^{n+1+|f_{1}|+\dots +|f_{p}|}\sum_{i=1}^{n-1}M_{n}\left( f_{1}\otimes  \cdots \otimes  f_{p}\otimes\left(  \tilde{m}^{A}_{1} f_{p+1}+(-1)^{|f_{p+1}|+1}f_{p+1}\circ \delta \right) \otimes f_{p+2}\otimes  \cdots \otimes  f_{n} \right)=\\
 & (-1)^{n+1}\sum_{i=1}^{n-1}\tilde{m}^{A}_{n}\left( Id^{\otimes p}\otimes \tilde{m}^{A}_{1} \otimes Id^{\otimes n-p-1}\right)(f_{1}\otimes  \cdots\otimes  f_{n})\circ \Delta^{n-1}+\\
 &(-1)^{n+|f_{1}|+\dots +|f_{n}|} \tilde{m}^{A}_{n}(f_{1}\otimes  \cdots\otimes  f_{n})\circ \left(\sum_{i=1}^{n-1}Id^{\otimes p}\otimes \delta \otimes Id^{\otimes r} \right)\circ \Delta^{n-1} =\\
 & (-1)^{n+1}\sum_{i=1}^{n-1}\tilde{m}^{A}_{n}\left( Id^{\otimes p}\otimes \tilde{m}^{A}_{1} \otimes Id^{\otimes n-p-1}\right)(f_{1}\otimes  \cdots\otimes  f_{n})\circ \Delta^{n-1}+\\
 &(-1)^{n+|f_{1}|+\dots +|f_{n}|} \tilde{m}^{A}_{n}(f_{1}\otimes  \cdots\otimes  f_{n})\circ \Delta^{n-1}\circ \delta.
 \end{align*}
 Since $\tilde{m}_{\bullet}^{A}$ is an $A_{\infty}$-structure
 \begin{align*}
 & \sum_{\substack{p+q+r=n\\
 p+1+r>1, q>1}}(-1)^{p+qr}\tilde{m}^{A}_{p+1+r} \left(Id^{\otimes p}\otimes  \tilde{m}^{A}_{q}\otimes Id^{\otimes r} \right)\circ (f_{1}\otimes  \cdots\otimes  f_{n}) \circ \Delta^{n-1}=\\
 &   \tilde{m}^{A}_{1}\tilde{m}^{A}_{n}(f_{1}\otimes  \cdots\otimes  f_{n})\circ\Delta^{n-1}+(-1)^{n+1}\sum_{i=1}^{n-1}\tilde{m}^{A}_{n}\left( Id^{\otimes p}\otimes \tilde{m}^{A}_{1} \otimes Id^{\otimes n-p-1}\right)(f_{1}\otimes  \cdots\otimes  f_{n})\circ \Delta^{n-1}.
 \end{align*}
 This shows that $\left( \partial, M_{2}, M_{3},\dots\right) $ is an $A_{\infty}$-structure. The second summand is canceled by the second summand of the first summand and this shows that $\left( M_{\bullet}\right) $ is an $A_{\infty}$-structure. We prove the third statement. By direct calculation $n=1,2$, we have that $l_{n}$ and $l_{n}'$ are well-defined on  
  $L_{V[1]^{*}}\left(A \right)$. Let $n>2$, and
  $f_{1}, \dots, f_{n}\in   L_{V[1]^{*}}(A)$ be homogeneous elements. Let $\mu'(a,b)$ be a non trivial shuffle in $T^{c}(V[1])\otimes T^{c}(V[1])$. For any $n$ we have
 \begin{align*}
 &\tilde{m}_{n}^{A}\left(f_{1}\otimes \cdots\otimes f_{n} \right)\circ \Delta^{n-1}\circ\mu'(a,b) =\\
 & = \tilde{m}_{n}^{A}\left(f_{1}\otimes \cdots\otimes f_{n} \right)\big(\sum_{i\neq j}^{n} 1\otimes \cdots 1\otimes  a\otimes 1\cdots 1\otimes  b\otimes 1 \cdots \otimes 1\\
 &+ (-1)^{|a||b|}\sum_{i\neq j}^{n} 1\otimes \cdots 1\otimes  b\otimes 1\cdots 1\otimes a \otimes 1 \cdots \otimes 1\big)\\
 & = 0
 \end{align*}
 since $f_{i}(1)=0$ for each $i$.
 \end{proof}
 Proof of corollary \eqref{Linftystructnotquot}
 \begin{proof}
The first claim is immediate. We prove the second statement. Let $f_{1},\dots,f_{n}\in  \Hom^{\bullet}\left(T^{c}(V[1]^{0}), A \right)$ and assume that there is a $g$ with $\delta^{*}g=f_{i}$ for some $i$. Then
 \begin{align*}
 &\tilde{m}_{n}^{A}\left(f_{1}\cdots \otimes\delta^{*}g\otimes \cdots\otimes f_{n} \right)\circ \Delta^{n-1} =\\
 & =\tilde{m}_{n}^{A}\left(f_{1}\cdots \otimes g\otimes \cdots\otimes f_{n}\right)\circ (Id\otimes\cdots \otimes \delta\otimes\cdots\otimes  Id)\circ \Delta^{n-1} \\
 &= \tilde{m}_{n}^{A}\left(f_{1}\cdots \otimes g\otimes \cdots\otimes f_{n}\right)\circ \left( \sum_{i=1}^{n}Id\otimes\cdots \otimes \delta\otimes\cdots\otimes  Id\right) \circ \Delta^{n-1}\\
 \end{align*}
 since $\delta^{*}f_{i}=0$ for any $i$. Then 
 \begin{align*}
 & \tilde{m}_{n}^{A}\left(f_{1}\cdots \otimes g\otimes \cdots\otimes f_{n}\right)\circ \left( \sum_{i=1}^{n}Id\otimes\cdots \otimes \delta\otimes\cdots\otimes  Id\right) \circ \Delta^{n-1}=\\
 &=\tilde{m}_{n}^{A}\left(f_{1}\cdots \otimes g\otimes \cdots\otimes f_{n}\right)\circ\Delta^{n-1}\circ \delta\\
 &=\delta^{*}\circ\tilde{m}_{n}^{A}\left(f_{1}\cdots \otimes g\otimes \cdots\otimes f_{n}\right)\circ\Delta^{n-1}.\\
 \end{align*}
 The third assertion is straightforward.
 \end{proof}
 \subsection{Conormalized graded module}\label{cosimplicial modules}
 The goal of this subsection is to construct an isomorphism $\psi\: : \: \operatorname{Tot}_{N}\left(A\right)\to  \int^{[n]\in\boldsymbol{\Delta}}NC_{n}\otimes A^{n,\bullet }$ between differential modules. Fix a field $\Bbbk$ of charactersitic zero. Let $B^{\bullet}$ be a cosimplicial module. The \emph{conormalized graded module} $N(B)^{\bullet}$ is a (cochain) graded modules defined as follows,
 \[
 N(B)^{p}:=\begin{cases}
 B^{0}, & \text{ if }p=0,\\
 \bigcap_{i=0}^{p}\operatorname{Ker}\left(s^{i}_{p-1}\: : \: B^{p}\to B^{p-1} \right) , & \text{ otherwise}.
 \end{cases}
 \]
 where $s^{p}_{i}$ for $i=0,\dots p-1$ are the codegenerancy maps.
 The differential $\partial\: : \: N(B)^{p}\to N(B)^{p+1}$ is given by the alternating sum of the coface maps
 \[
 \partial=\sum_{i=0}^{p}(-1)^{i}d^{i}.
 \]
 In particular $N^{\bullet}$ is a functor from the category of cosimplicial modules $cMod$ toward the category of (cochain) differential graded modules $dgMod$.\\
 A cosimplicial differential graded module is a cosimplicial objects in $A^{\bullet, \bullet}\in dgMod$ where the first slot denotes the cosimplicial degree and the second slot denotes the differential degree. It can be visualized as a sequence of cosimplicial modules
 \[
 \begin{tikzcd}
 A^{\bullet,0 }\arrow{r}{d}& A^{ \bullet,1}\arrow{r}{d}& A^{ \bullet,2}\arrow{r}{d}&\dots
 \end{tikzcd}
 \]
 If we apply the functor $N$ we turn the cosimplicial structure of each terms into a differential graded structures
 \[\begin{tikzcd}
 N\left( A^{\bullet,0 }\right)^{\bullet}\arrow{r}{d}& N\left( A^{ \bullet,1}\right)^{\bullet}\arrow{r}{d}& N\left( A^{ \bullet,2}\right)^{\bullet}\arrow{r}{d}&\dots
 \end{tikzcd}\]
 moreover since each cosimplicial maps commutes with the differentials, we get a bicomplex $\left(N(A), d, \partial \right) $, where $N(A)^{p,q}:=N(A^{p,\bullet})^{q}$. We define $\operatorname{Tot}_{N}\left(A\right)\in dgMod$ as the total complex associated to the bicomplex above. Explicitly an element $a\in \operatorname{Tot}\left(N(A) \right)^{k} $ is a collection
 \[
 (a_{0}, \dots , a_{k})\in A^{k,0}\oplus A^{1,k}\oplus
 \dots A^{0,k}
 \]
 such that each $a_{i}$ is contained in the kernel of some codegenrancy map. The following is well-known.
  
  \begin{lem}\label{isompsi}
  Let $A^{\bullet,\bullet}$ be a cosimplicial differential graded module.
  \begin{enumerate}
  \item Let $v$ be an elements of bidegree $(p,q)$. Then each $v_{n}\in NC_{n}^{p}\otimes A^{n,q}$ is equal to zero for $p>n$.
  \item $NC_{p}^{p}$ is a one dimensional module.
  
  \item An element $v$ with bidegree $(p,q)$ is completely determined by 
  \[
  v_{p}\in NC_{p}^{p}\otimes A^{p,q}.
  \]
  \item There is an isomorphism between differential modules $\psi\: : \: \operatorname{Tot}_{N}\left(A\right)\to  \int^{[n]\in\boldsymbol{\Delta}}NC_{n}\otimes A^{n,\bullet }$ such that for $v$ with bidegree $(p,q)$  we have
  \[
  \psi(b)_{p}=b\in NC_{p}^{p}\otimes A^{p,q}
  \]
  \end{enumerate}
  \end{lem}
 \begin{proof}
 The first two points are immediate.
 Fix a $n$ and a $p\leq n$. Notice that each inclusion $[p]\hookrightarrow [n]$ is equivalent to an ordered string $0\leq i_{0}<i_{i}<\dots <i_{p}\leq n$ contained in $\left\lbrace 0,1,\dots, n\right\rbrace $. For each string $0\leq i_{0}<i_{i}<\dots <i_{p}\leq n$ we denote the associated inclusion by $\sigma_{i_{0}, \dots i_{p}}\: : \:[p]\hookrightarrow [n]$, and we define the maps $\lambda_{i_{0}, \dots, i_{p}}\: : \: \Delta[n]_{p}^{+}\to \Bbbk$, via
 \[
 \lambda_{i_{0}, \dots, i_{p}}(\phi):=
 \begin{cases}
 1\text{ if }\sigma_{i_{0}, \dots i_{p}}=\phi,\\
 0,\text{ otherwhise}
 \end{cases}
 \]
 Clearly $\left\lbrace \lambda_{i_{0}, \dots i_{p}}\right\rbrace _{0\leq i_{0}<i_{i}<\dots <i_{p}\leq n}$ is a basis of  $NC_{n}^{p}$. It turns out that $v_{n}$ can be written as 
 \[
 v_{n}=\sum_{\theta \in \Delta[n]_{p}^{+}}\lambda_{i_{0}, \dots, i_{p}}\otimes b^{i_{0}, \dots, i_{p}}
 \]
 for some $b^{i_{0}, \dots i_{p}}\in B^{n,q}$. Let $v_{p}\in NC_{p}^{p}\otimes A^{p,q}$. Since $NC_{p}^{p}$ is one dimensional we write $v_{p}=\lambda_{0, \dots, p}\otimes b$, for some $b\in A^{p,q}$. We shows that $b^{i_{0}, \dots i_{p}}$ is completely determined by $b$. Fix a $0\leq i_{0}<i_{i}<\dots <i_{p}\leq n$. Then the above relations read as follows
 \[
 \left( 1\otimes\sigma_{i_{0}, \dots i_{p}}^{*}\right)v_{p}=\left({\sigma_{i_{0}, \dots i_{p}}}_{*}\otimes 1 \right)v_{n}.
 \]
 In particular the map ${\sigma_{i_{0}, \dots i_{p}}}_{*}\: : \: NC_{n}^{p}\to  NC_{p}^{p}$ is the linear map that sends $\lambda_{i_{0}, \dots i_{p}}$ to $\lambda_{{0}, \dots, {p}}$ and $\lambda_{j_{0}, \dots j_{p}}$ to $0$ for $\left( j_{0}, \dots j_{p}\right) \neq \left( i_{0}, \dots i_{p}\right) $.
 Then
 \begin{align*}
 \lambda_{{0}, \dots ,{p}}\otimes \sigma_{i_{0}, \dots i_{p}}^{*}(b)=\left( 1\otimes\sigma_{i_{0}, \dots i_{p}}^{*}\right)v_{p}=\left( {\sigma_{i_{0}, \dots i_{p}}}_{*}\otimes 1 \right)v_{n}=\lambda_{{0}, \dots ,{p}}\otimes b^{i_{0}, \dots i_{p}}.
 \end{align*}
 Since each degree $k$ elements can be written uniquely as a sum of elements of bidegree $(p,q)$ with $p+q=k$ we get a natural isomorphism
 $\psi\: : \: \operatorname{Tot}_{N}\left(A\right)\to  \int^{[n]\in\boldsymbol{\Delta}}NC_{n}\otimes A^{n,\bullet }$ of differential graded modules.
 \end{proof}
 \subsection{The $C_{\infty}$-structure on the 2 dimensional simplex}\label{sec2dimsimplex}
This diagram originally defined in \cite{Dupont2} is intensively studied in \cite{Getz2}. We define two maps between simplicial differential graded module (see \cite{Dupont2})
 \[
 \begin{tikzcd}
 E_{\bullet}\: : \: NC_{\bullet}^{\bullet}\arrow[r, shift right, ""]&\Omega^{\bullet}(\bullet)\: : \: \int_{\bullet}.\arrow[l, shift right, ""] 
 \end{tikzcd}
 \] 
 \begin{enumerate}
 \item Fix a $[n]\in \boldsymbol{\Delta}$. For each $p\leq n$ we define  $\int_{n} \: : \: \Omega^{p}(n)\to NC_{n}^{p}$ via
 \[
 \int_{n}(w)(\sigma_{i_{0}, \dots, i_{p}}):= \int_{\Delta_{geo}[p]}\sigma_{i_{0}, \dots, i_{p}}^{*}(w)
 \]
 i.e, we pulled back $w$ along the smooth inclusion $\sigma_{i_{0}, \dots, i_{p}}\: : \: \Delta[p]_{geo}\to \Delta[n]_{geo}$ and we integrate along the geometric standard $p$-simplex. $\int \: : \: \Omega^{\bullet}(n)\to NC_{n}$  is indeed a map between graded modules. The Stokes'theorem implies 
 \[
 \int_{n}(dw)(\sigma_{i_{0}, \dots, i_{p}})=\sum_{j=0}^{p}(-1)^{j}\int_{n}(w)(\sigma_{i_{0}, \dots, \hat{i_{j}}, \dots i_{p}}),
 \]
 i.e $\int_{n} \: : \: \Omega^{\bullet}(n)\to NC_{n}$ is a map between differential graded modules. Moreover, the above construction is compatibAle with the simplicial structure, i.e $\int_{\bullet} \: : \: \Omega^{\bullet}(\bullet)\to NC_{\bullet}^{*}$ is a map between simplicial differential graded modules.
 \item We define the quasi inverse of $\int_{\bullet}$. We fix a $[n]\in \boldsymbol{\Delta}$. For each string $0\leq<i_{0}<i_{1}<\dots <i_{p}\leq n$ we define the \emph{Whitney elementary form} $\omega_{i_{0}, \dots, i_{p}}\in \Omega^{p}(n)$ via
 \[
 \omega_{i_{0}, \dots , i_{p}}:=k!\sum_{j=0}^{p}(-1)^{j}t_{i_{j}}dt_{i_{0}}\wedge \dots \wedge \hat{dt_{i_{j}}}\wedge \dots \wedge dt_{i_{p}}
 \]
 We define $E_{n}\: : \: NC_{n}^{p}\hookrightarrow \Omega^{p}(n)$ via
 \[
 E_{n}(\lambda):=\sum_{0\leq i_{0}< \dots < i_{p}\leq n} \lambda(\sigma_{i_{0}, \dots , i_{p}})\omega_{i_{0}, \dots , i_{p}}
 \]
 The above map defines a map between differential graded modules $E_{n}\: : \: NC_{n}\hookrightarrow \Omega^{\bullet}(n)$. Since the construction is compatible with the simplicial maps we get a map between simplicial differential graded modules $E_{\bullet}\: : \: NC_{\bullet}^{\bullet}\hookrightarrow \Omega^{\bullet}(\bullet)$. Moreover since
 \[
 \int_{\Delta[n]_{geo}}t_{1}^{a_{1}}\cdots t_{n}^{a_{n}}dt_{1}\wedge dt_{n}:=\frac{a_{1}!\cdots a_{n}!}{\left( a_{1}+\dots a_{n}+n\right)!}
 \]
 an easy computation demonstrates that
 \[
 \left( \int_{\bullet}\right) \circ E_{\bullet}=Id_{ NC_{\bullet}^{\bullet}}.
 \]
 \end{enumerate}
 It remains to construct an explicit simplicial homotopy between $ E_{\bullet}\circ\left( \int_{\bullet}\right)$ and $\mathrm{Id}_{ \Omega^{\bullet}(\bullet)}$. We recall the construction of \cite{Dupont2} (see also \cite{Getz2}). Fix a $n$, for $0\leq i\leq n$ we define the map $\phi_{i}\: : \: [0,1]\times \Delta_{geo}[n]\to \Delta_{geo}[n]$ via
 \[
 \phi_{i}\left(u, t_{0}, \dots, t_{n}\right):=\left(t_{0}+(1-u)\delta_{i0}, \dots,t_{0}+(1-u)\delta_{in} \right)  
 \]
 Let $\pi\: : \: [0,1]\times \Delta_{geo}[n]\to \Delta_{geo}[n]$ be the projection at the second coordinate and let  $\pi_{*}\: : \: \Omega^{\bullet}(n)\to \Omega^{\bullet-1}(n)$ be the integration along the fiber. We define $h^{i}_{n}\: : \: \Omega^{\bullet}(n)\to \Omega^{\bullet-1}(n)$ via
 \[
 h_{n}^{i}(w):=\pi_{*}\circ\phi^{*}_{i}(w)
 \]
  We define $s_{n}\: : \: \Omega^{p}(n)\to \Omega^{p-1}(n)$ as follows: 
 \[
 s_{n}w:=\sum_{j=0}^{p-1} \sum_{0\leq i_{0}<\dots<i_{j}\leq n}\omega_{i_{0}\dots i_{j}}\wedge h^{i_{j}}_{n}\dots  h^{i_{0}}_{n}(w)
 \]
 We start with the proof of Proposition \ref{speriam}. 
 \begin{proof}
 Recall that $\Omega^{\bullet}(2)$ is the free differential graded commutative algebra generated by the degree zero variables $t_{0}$, $t_{1}$ and $t_{2}$ modulo the relations
 \[
 t_{0}+t_{1}+t_{2}=1,\quad dt_{0}+dt_{1}+dt_{2}=0.
 \]
 On the other hand $NC_{2}^{0}$ is the vector space generated by $\lambda_{0}$, $\lambda_{1}$ and $\lambda_{2}$, $NC_{2}^{1}$ is generated by $\lambda_{01},\lambda_{02} $ and $\lambda_{12}$, and $NC_{2}^{2}$ is the one dimensional vector space generated by $\lambda_{0012}$. We have
\begin{enumerate}
\item $E_{2}\left( \lambda_{0}\right)=t_{0}$, $E_{2}\left( \lambda_{1}\right)=t_{1}$, and $E_{2}\left( \lambda_{2}\right)=t_{2}$.
\item $E_{2}\left( \lambda_{01}\right)=t_{0}dt_{1}-t_{1}dt_{0}$, $E_{2}\left( \lambda_{02}\right)=t_{0}dt_{2}-t_{2}dt_{0}$, and $E_{2}\left( \lambda_{12}\right)=t_{1}dt_{2}-t_{2}dt_{1}$.
\item $E_{2}\left( \lambda_{012}\right)=2t_{0}dt_{1}dt_{2}-2t_{1}dt_{0}dt_{2}+2t_{2}dt_{0}dt_{1}$.
\end{enumerate}
We prove a). For degree reason we have $m_{2}(\lambda_{01},\lambda_{02})=\mu_{01/02}\lambda_{012}$. By the homotopy transfer theorem (see \cite{kontsoibel}) we have
\begin{align*}
\mu_{01/02} &=\left( \int_{\Delta[2]}E_{2}\left( \lambda_{01}\right)E_{2}\left( \lambda_{02}\right)\right) \\
& = \int_{\Delta[2]} t_0 (t_0dt_1dt_2  -t_1 dt_0dt_2  +t_2 dt_0dt1 )\\
& =\frac{1}{6}
\end{align*}
and analogously the other cases.
\end{proof}

\end{document}